\documentclass[11pt]{article}

\usepackage{xcolor}
\usepackage{amsmath}
\usepackage{mathtools}
\usepackage{amssymb,mathrsfs,euscript}
\usepackage{enumerate}
\usepackage[normalem]{ulem}
\usepackage{cancel}
\usepackage{enumitem}
\usepackage{tikz}
\usepackage{dsfont}
\usepackage{verbatim,authblk,amsthm}
\usepackage[round]{natbib}
\usepackage{float,graphicx,subcaption}

\newcommand*\circled[1]{\tikz[baseline=(char.base)]{
            \node[shape=circle,draw,inner sep=2pt] (char) {#1};}}
            
\newcommand{\inner}[2]{{\left\langle #1, #2 \right\rangle}}            
\newcommand{\norm}[1]{\left\|#1\right\|}

\newcommand{\X}{\mathcal X}
\newcommand{\Y}{\mathcal Y}
\newcommand{\R}{\mathbb R}
\newcommand{\N}{\mathbb N}
\newcommand{\E}{\mathbb E}
\newcommand{\h}{\mathscr H}
\newcommand{\hh}{\boldsymbol{H}}
\newcommand{\Id}{\boldsymbol{I}}
\newcommand{\one}{\boldsymbol{1}}
\newcommand{\hgh}{\Tilde{\hh}_{s}^{1/2} G \Tilde{\hh}_{s}^{1/2}}
\newcommand{\B}{\mathcal B}
\newcommand{\M}{\mathcal M}

\newcommand{\id}{\mathfrak I}
\newcommand{\T}{\mathcal T}
\newcommand{\A}{\mathcal A}
\newcommand{\PP}{\mathcal P}
\newcommand{\Ntl}{\mathcal{N}_{2}(\lambda)}

\newcommand{\Ntlsq}{\mathcal{N}^2_{2}(\lambda)}

\newcommand{\Nol}{\mathcal{N}_{1}(\lambda)}

\newcommand{\PQ}{R}

\newcommand{\Cs}{(C_1+C_2)}
\newcommand{\Cl}{C_{\lambda}}
\newcommand{\II}{\mathds{1}}
\newcommand{\K}{\kappa}
\newcommand{\kk}{K}

\newcommand{\op}{\EuScript{L}^\infty(\h)}
\newcommand{\opl}{\EuScript{L}^\infty(\Lp)}
\newcommand{\opH}{\EuScript{L}^\infty(H)}

\newcommand{\hs}{\EuScript{L}^2(\h)}
\newcommand{\hsH}{\EuScript{L}^2(H)}
\newcommand{\Lp}{L^{2}(R)}
\newcommand{\range}{\text{Ran}} 

\newcommand{\SgL}{\Sigma_{PQ,\lambda}^{-1/2}}
\newcommand{\SgLh}{\hat{\Sigma}_{PQ,\lambda}^{-1/2}}
\newcommand{\SL}{\Sigma_{PQ,\lambda}}
\newcommand{\SLh}{\hat{\Sigma}_{PQ,\lambda}}

\newcommand{\gS}{g_{\lambda}(\Sigma_{PQ})}
\newcommand{\gShalf}{g^{-1/2}_{\lambda}(\Sigma_{PQ})}
\newcommand{\gSh}{g_{\lambda}(\hat{\Sigma}_{PQ})}
\newcommand{\gShh}{g^{1/2}_{\lambda}(\hat{\Sigma}_{PQ})}

\newcommand{\gT}{g_{\lambda}(\T)}
\newcommand{\gl}{g_{\lambda}}

\newcommand{\U}{u}
\newcommand{\gSL}{g^{1/2}_{\lambda}(\Sigma_{PQ})}

\newcommand{\gSLh}{g^{1/2}_{\lambda}(\hat{\Sigma}_{PQ})}
\newcommand{\igSLh}{g^{-1/2}_{\lambda}(\hat{\Sigma}_{PQ})}

\newcommand{\stat}{\hat{\eta}_{\lambda}}

\newcommand{\htens}{\otimes_{\h}}
\newcommand{\Htens}{\otimes_{H}}

\newcommand{\ltens}{\otimes_{\Lp}}

\newcommand{\qqh}{\hat{q}_{1-\alpha}^{B,\lambda}}

\newcommand{\qqe}{q_{1-\alpha}^{\lambda}}

\newcommand{\cd} {|\Lambda|}

\renewcommand{\epsilon}{\varepsilon}

\DeclarePairedDelimiter{\floor}{\lfloor}{\rfloor}

\numberwithin{equation}{section}
\newtheorem{theorem}{Theorem}
\numberwithin{theorem}{section}

\newtheorem{corollary}[theorem]{Corollary}

\theoremstyle{rem}
\newtheorem{remark}{Remark}
\numberwithin{remark}{section}

\newtheorem{appxthm}{Theorem}[section]
\newtheorem{appxlem}[appxthm]{Lemma}

\theoremstyle{definition}

\theoremstyle{remark}

\begin{document}

\title{Spectral Regularized Kernel Two-Sample Tests}
\author{Omar Hagrass}
\author{Bharath K. Sriperumbudur}
\author{Bing Li}
\affil{Department of Statistics,  
Pennsylvania State University\\
University Park, PA 16802, USA.\\
\texttt{\{oih3,bks18,bxl9\}}@psu.edu}
\date{}
\maketitle

\begin{abstract}
 Over the last decade, an approach that has gained a lot of popularity to tackle nonparametric testing problems on general (i.e., non-Euclidean) domains is based on the notion of reproducing kernel Hilbert space (RKHS) embedding of probability distributions.  The main goal of our work is to understand the optimality of two-sample tests constructed based on this approach. First, we show the popular MMD (maximum mean discrepancy)  two-sample test to be not optimal in terms of the separation boundary measured in Hellinger distance. Second, we propose a modification to the MMD test based on spectral regularization by taking into account the covariance information (which is not captured by the MMD test) and prove the proposed test to be minimax optimal with a smaller separation boundary than that achieved by the MMD test. Third, we propose an adaptive version of the above test which involves a data-driven strategy to choose the regularization parameter and show the adaptive test to be almost minimax optimal up to a logarithmic factor.  Moreover, our results hold for the permutation variant of the test where the test threshold is chosen elegantly through the permutation of the samples. Through numerical experiments on synthetic and real data, we demonstrate the superior performance of the proposed test in comparison to the MMD test and other popular tests in the literature.
\end{abstract}
\setlength{\parskip}{4pt}
\section{Introduction}\label{sec:intro}
Given $\mathbb{X}_N:=(X_i)_{i=1}^N \stackrel{i.i.d.}{\sim} P$, and $\mathbb{Y}_M:=(Y_j)_{j=1}^M \stackrel{i.i.d.}{\sim} Q$, where $P$ and $Q$ are defined on a measurable space $\X$, the problem of two-sample testing is to test $H_0: P=Q$ against $H_1: P\ne Q$. This is a classical problem in statistics that has attracted a lot of attention both in the parametric (e.g., $t$-test, $\chi^2$-test) and nonparametric (e.g., Kolmogorov-Smirnoff test, Wilcoxon signed-rank test) settings \citep{lehmann}. However, many of these tests either rely on strong distributional assumptions or cannot handle non-Euclidean data that naturally arise in many modern applications.

Over the last decade, an approach that has gained a lot of popularity to tackle nonparametric testing problems on general domains is based on the notion of reproducing kernel Hilbert space (RKHS) \citep{Aronszajn} embedding of probability distributions (\citealp{Smola}, \citealp{classifiability}, \citealp{rkhs}).
Formally, the RKHS embedding of a probability measure $P$ is defined as
$$\mu_P=\int_{\X}K(\cdot,x)\,dP(x)\in\mathscr{H},$$
where $K:\X\times\X\rightarrow\mathbb{R}$ is the unique reproducing kernel (r.k.) associated with the RKHS $\mathscr{H}$ with $P$ satisfying  $\int_\X \sqrt{K(x,x)}\,dP(x)<\infty$. If $K$ is characteristic (\citealp{JMLR_metrics}, \citealp{JMLR_universal}), this embedding induces a metric on the space of probability measures, called the \emph{maximum mean discrepancy} ($\mathrm{MMD}$) or the \emph{kernel distance} (\citealp{gretton12a}, \citealp{NipsGretton}), defined as \begin{equation}D_{\mathrm{MMD}}(P,Q) = \norm{\mu_P-\mu_Q}_{\h}.\label{Eq:MMD}\end{equation} MMD has the following variational representation (\citealp{gretton12a}, \citealp{JMLR_metrics}) given by
\begin{equation}
D_{\mathrm{MMD}}(P,Q) := \sup_{f \in \h : \norm{f}_{\h}\leq 1} \int_{\X} f(x)\,d(P-Q)(x),\label{eq:variational}
\end{equation}
which clearly reduces to \eqref{Eq:MMD} by using the reproducing property: $f(x)=\inner{f}{K(\cdot,x)}_{\h}$ for all $f \in \h$, $x\in\mathcal{X}$. We refer the interested reader to \citet{JMLR_metrics}, \citet{bernouli2016}, and \citet{Carl} for more details about $D_{\mathrm{MMD}}$. 
\citet{gretton12a} proposed a test based on the asymptotic null distribution of the $U$-statistic estimator of $D^2_{\mathrm{MMD}}(P,Q)$, defined as
\begin{align}\hat{D}_\mathrm{MMD}^2(X,Y)&=\frac{1}{N(N-1)}\sum_{i \neq j}K(X_i,X_j)+\frac{1}{M(M-1)}\sum_{i \neq j}K(Y_i,Y_j) \nonumber\\
&\qquad\qquad-\frac{2}{NM}\sum_{i,j}K(X_i,Y_j) \notag,  \end{align}
and showed it to be consistent. Since the asymptotic null distribution does not have a simple closed form---the distribution is that of an infinite sum of weighted chi-squared random variables with the weights being the eigenvalues of an integral operator associated with the kernel $K$ w.r.t.~the distribution $P$---, several approximate versions of this test have been investigated and are shown to be asymptotically consistent (e.g., see \citealp{gretton12a} and \citealp{FastKernel}). Recently, \citet{MingYuan} and \citet{MMDagg} showed these tests based on $\hat{D}^2_{\text{MMD}}$ to be not optimal in the minimax sense but modified them to achieve a minimax optimal test by using translation-invariant kernels on $\mathcal{X}=\mathbb{R}^d$. However, 
since the power of these kernel methods lies in handling more general spaces and not just $\mathbb{R}^d$, the main goal of this paper is to construct minimax optimal kernel two-sample tests on general domains.

Before introducing our contributions, first, we will introduce the minimax framework pioneered by \citet{Burnashev} and \citet{Ingester1, Ingester2}  to study the optimality of tests, which is essential to understanding our contributions and their connection to the results of \citet{MingYuan} and \citet{MMDagg}. Let $\phi(\mathbb{X}_N,\mathbb{Y}_M)$ be any test that rejects $H_0$ when $\phi=1$ and fails to reject $H_0$ when $\phi=0$. Denote the class of all such asymptotic (\emph{resp.} exact) $\alpha$-level tests to be $\Phi_\alpha$ (\emph{resp.} $\Phi_{N,M,\alpha}$). Let $\mathcal{C}$ be a set of probability measures on $\mathcal{X}$. The Type-II error of a test $\phi\in \Phi_\alpha$ (\emph{resp.} $\in\Phi_{N,M,\alpha}$) w.r.t.~$\mathcal{P}_\Delta$ is defined as
$R_\Delta(\phi)=\sup_{(P,Q)\in\mathcal{P}_\Delta}\mathbb{E}_{P^N\times Q^M}(1-\phi),$ 
where $$\mathcal{P}_\Delta:=\left\{(P,Q)\in \mathcal{C}^2:\rho^2(P,Q)\ge \Delta\right\},$$ is the class of $\Delta$-separated alternatives in probability metric $\rho$, with $\Delta$ being referred to as the \emph{separation boundary} or \emph{contiguity radius}. Of course, the interest is in letting $\Delta\rightarrow 0$ as $M,N\rightarrow \infty$ (i.e., shrinking alternatives) and analyzing $R_\Delta$ for a given test, $\phi$, i.e., whether $R_\Delta(\phi)\rightarrow 0$. In the asymptotic setting, the \emph{minimax separation} or \emph{critical radius} $\Delta^*$ is the fastest possible order at which $\Delta\rightarrow 0$ such that $\lim\inf_{N,M\rightarrow\infty}\inf_{\phi\in\Phi_\alpha}R_{\Delta^*}(\phi)\rightarrow 0$, i.e., for any $\Delta$ such that $\Delta/\Delta^*\rightarrow\infty$, there is no test $\phi\in\Phi_\alpha$ that is consistent over $\mathcal{P}_\Delta$. A test is \emph{asymptotically minimax optimal} if it is consistent over $\mathcal{P}_\Delta$ with $\Delta \asymp\Delta^*$. On the other hand, in the non-asymptotic setting, the minimax separation $\Delta^*$ is defined as the minimum possible separation, $\Delta$ such that $\inf_{\phi\in\Phi_{N,M,\alpha}}R_{\Delta}(\phi)\le\delta$, for $0<\delta<1-\alpha$. A test $\phi\in\Phi_{N,M,\alpha}$ is called \emph{minimax optimal} if $R_\Delta(\phi)\le \delta$ for some $\Delta\asymp\Delta^*$. In other words, there is no other $\alpha$-level test that can achieve the same power with a better separation boundary. 

In the context of the above notation and terminology, \citet{MingYuan} considered distributions with densities (w.r.t.~the Lebesgue measure) belonging to \begin{equation}\mathcal{C}=\left\{f:\mathbb{R}^d\rightarrow\mathbb{R}\,:\,f\,\,\text{a.s. continuous and}\,\,\Vert f\Vert_{W^{s,2}_d}\le M\right\},\label{Eq:C}\end{equation} where $\Vert f\Vert^2_{W^{s,2}_d}=\int (1+\Vert x\Vert^2_2)^s|\hat{f}(x)|^2\,dx$ with $\hat{f}$ being the Fourier transform of $f$, $\rho(P,Q)=\Vert p-q\Vert_{L^2}$ with $p$ and $q$ being the densities of $P$ and $Q$, respectively, and showed the minimax separation to be $\Delta^*\asymp (N+M)^{-4s/(4s+d)}$. Furthermore, they chose $K$ to be a Gaussian kernel on $\mathbb{R}^d$, i.e., $K(x,y)=\exp(-\Vert x-y\Vert^2_2/h)$ in $\hat{D}^2_{\text{MMD}}$ with $h\rightarrow 0$ at an appropriate rate as $N,M\rightarrow\infty$  (reminiscent of kernel density estimators) in contrast to fixed $h$ in \citet{gretton12a}, and showed the resultant test to be asymptotically minimax optimal w.r.t.~$\mathcal{P}_\Delta$ based on \eqref{Eq:C} and $\Vert\cdot\Vert_{L^2}$. \citet{MMDagg} extended this result to translation-invariant kernels (particularly, as the product of one-dimensional translation-invariant kernels) on $\mathcal{X}=\mathbb{R}^d$ with a shrinking bandwidth and showed the resulting test to be minimax optimal even in the non-asymptotic setting. While these results are interesting, the analysis holds only for $\mathbb{R}^d$ as the kernels are chosen to be translation invariant on $\mathbb{R}^d$, thereby limiting the power of the kernel approach.

In this paper, we employ an operator theoretic perspective to understand the limitation of $D^2_{\text{MMD}}$ and propose a regularized statistic that mitigates these issues without requiring $\mathcal{X}=\mathbb{R}^d$. In fact, the construction of the regularized statistic naturally gives rise to a certain $\mathcal{P}_\Delta$ which is briefly described below. To this end, define $R:=\frac{P+Q}{2}$ and $u:=\frac{dP}{dR}-1$ which is well defined as $P\ll R$. It can be shown that $D^2_{\text{MMD}}(P,Q)=4\langle \mathcal{T}u,u\rangle_{L^2(R)}$ where $\mathcal{T}:L^2(R)\rightarrow L^2(R)$ is an integral operator defined by $K$ (see Section~\ref{Sec:non-optimal} for details), which is in fact a self-adjoint positive trace-class operator if $K$ is bounded. Therefore, $D^2_{\text{MMD}}(P,Q)=\sum_i \lambda_i \langle u,\tilde{\phi}_i\rangle^2_{L^2(R)}$ where $(\lambda_i,\tilde{\phi}_i)_i$ are the eigenvalues and eigenfunctions of $\mathcal{T}$. Since $\mathcal{T}$ is trace-class, we have $\lambda_i\rightarrow 0$ as $i\rightarrow\infty$, which implies that the Fourier coefficients of $u$, i.e., $\langle u,\tilde{\phi}_i\rangle_{L^2(R)}$, for large $i$, are down-weighted by $\lambda_i$. In other words, $D^2_{\text{MMD}}$ is not powerful enough to distinguish between $P$ and $Q$ if they differ in the high-frequency components of $u$, i.e., $\langle u,\tilde{\phi}_i\rangle_{L^2(R)}$ for large $i$. On the other hand, $$\Vert u\Vert^2_{L^2(R)}=\sum_i \langle u,\tilde{\phi}_i\rangle^2_{L^2(R)}=\chi^2\left(P\left\Vert\right.\frac{P+Q}{2}\right)=\frac{1}{2}\int_\mathcal{X}\frac{(dP-dQ)^2}{d(P+Q)}
=:\underline{\rho}^2(P,Q)$$ does not suffer from any such issue, with $\underline{\rho}$ being a probability metric that is topologically equivalent to the Hellinger distance (see Lemma~\ref{Lem: distance} 
and \citealp[p. 47]{LeCam}). With this motivation, we consider the following modification to $D^2_{\text{MMD}}$: $$\eta_\lambda(P,Q)=4\left\langle \mathcal{T}g_\lambda(\mathcal{T})u,u\right\rangle_{L^2(R)},$$
where $g_\lambda:(0,\infty)\rightarrow (0,\infty)$, called the \emph{spectral regularizer} \citep{Engl.et.al} is such that $\lim_{\lambda\rightarrow 0}xg_\lambda(x)\asymp 1$ as $\lambda\rightarrow 0$ (a popular example is the Tikhonov regularizer, $g_\lambda(x)=\frac{1}{x+\lambda}$), i.e., $\mathcal{T}g_\lambda(\mathcal{T})\approx \Id$, the identity operator---refer to \eqref{Eq:glambda} for the definition of $g_\lambda(\mathcal{T})$. In fact, in Section~\ref{Sec:spec}, we show  $\eta_\lambda(P,Q)$ to be equivalent to $\Vert u\Vert^2_{L^2(R)}$, i.e., $\Vert u\Vert^2_{L^2(R)}\lesssim\eta_\lambda(P,Q)\lesssim \Vert u\Vert^2_{L^2(R)}$ if $u\in \text{Ran}(\mathcal{T}^\theta),\,\theta>0$ and $\Vert u\Vert^2_{L^2(R)}\gtrsim \lambda^{2\theta}$, where $\text{Ran}(A)$ denotes the range space of an operator $A$, $\theta$ is the smoothness index (large $\theta$ corresponds to ``smooth" $u$), and $\mathcal{T}^\theta$ is defined by choosing $g_\lambda(x)=x^\theta,\,x\ge 0$ in \eqref{Eq:glambda}. This naturally leads to the class of $\Delta$-separated alternatives,
\begin{equation}\PP:=\PP_{\theta,\Delta} := \left\{(P,Q): \frac{dP}{d\PQ}-1 \in \range (\T^{\theta}),\,\,\underline{\rho}^2(P,Q) \geq \Delta\right\},
\label{Eq:alternative-theta}
\end{equation}
for $\theta>0$, where $\text{Ran}(\T^\theta),\,\theta\in(0,\frac{1}{2}]$ can be interpreted as an interpolation space obtained by the real interpolation of $\mathscr{H}$ and $L^2(R)$ at scale $\theta$ \citep[Theorem 4.6]{Steinwart2012MercersTO}---note that the real interpolation of Sobolev spaces and $L^2(\mathbb{R}^d)$ yields Besov spaces \citep[p. 230]{Adams}. To compare the class in \eqref{Eq:alternative-theta} to that obtained using \eqref{Eq:C} with $\rho(\cdot,\cdot)=\Vert \cdot-\cdot\Vert_{L^2(\R^d)}$, note that the smoothness in \eqref{Eq:alternative-theta} is determined through $\text{Ran}(\mathcal{T}^\theta)$ instead of the Sobolev smoothness where the latter is tied to translation-invariant kernels on $\R^d$. Since we work with general domains, the smoothness is defined through the interaction between $K$ and the probability measures in terms of the behavior of the integral operator, $\mathcal{T}$. In addition, as $\lambda\rightarrow 0$, $\eta_\lambda(P,Q)\rightarrow \underline{\rho}^2(P,Q)$, while $D^2_{\text{MMD}}(P,Q)\rightarrow \Vert p-q\Vert^2_2$ as $h\rightarrow 0$, where $D^2_{\text{MMD}}$ is defined through a translation invariant kernel on $\mathbb{R}^d$ with bandwidth $h>0$. Hence, we argue that \eqref{Eq:alternative-theta} is a natural class of alternatives to investigate the performance of $D^2_{\text{MMD}}$ and $\eta_\lambda$. In fact, recently, \citet{Krishna} considered an alternative class similar to \eqref{Eq:alternative-theta} to study goodness-of-fit tests using $D^2_{\text{MMD}}$.
\vspace{-3mm}
\subsection*{Contributions}
The main contributions of the paper are as follows:\vspace{1.5mm}\\
\emph{(i)}  First, in Theorem~\ref{thm: MMD}, we show that the test based on $\hat{D}^2_{\text{MMD}}$ cannot achieve a separation boundary better than $(N+M)^{\frac{-2\theta}{2\theta+1}}$ w.r.t.~$\mathcal{P}$ in \eqref{Eq:alternative-theta}. However, this separation boundary depends only on the smoothness of $u$, which is determined by $\theta$ but is completely oblivious to the \emph{intrinsic dimensionality} of the RKHS, $\mathscr{H}$, which is controlled by the decay rate of the eigenvalues of $\mathcal{T}$. To this end, by taking into account the intrinsic dimensionality of $\mathscr{H}$, we show in Corollaries~\ref{coro:poly-minimax} and \ref{coro:exp-minimax} (also see Theorem~\ref{thm:minimax}) that the minimax separation w.r.t.~$\mathcal{P}$ is $(N+M)^{-\frac{4\theta\beta}{4\theta\beta+1}}$ for $\theta>\frac{1}{2}$ if $\lambda_i\asymp i^{-\beta}$, $\beta>1$, i.e., the eigenvalues of $\mathcal{T}$ decay at a polynomial rate $\beta$, and is $\sqrt{\log(N+M)}/(N+M)$  if $\lambda_i\asymp e^{-i}$, i.e., exponential decay. These results clearly establish the non-optimality of the MMD-based test.\vspace{1.5mm}\\
\emph{(ii)} To resolve this issue with MMD, in Section~\ref{subsec:oracle}, we propose a spectral regularized test based on $\eta_\lambda$ and show it to be minimax optimal w.r.t.~$\mathcal{P}$ (see Theorems~\ref{thm: Type1 error}, \ref{thm:Type II} and Corollaries~\ref{coro:poly}, \ref{coro:exp}). Before we do that, we first provide an alternate representation for $\eta_\lambda(P,Q)$ as $\eta_\lambda(P,Q)=\Vert g^{1/2}_\lambda(\Sigma_R)(\mu_P-\mu_Q)\Vert^2_\mathscr{H}$, which takes into account the information about the covariance operator, $\Sigma_R$ along with the mean elements, $\mu_P$, and $\mu_Q$, thereby showing resemblance to Hotelling's $T^2$-statistic~\citep{lehmann} and its kernelized version \citep{Harchaoui}. This alternate representation is particularly helpful to construct a two-sample $U$-statistic \citep{Hoeffding} as a test statistic (see Section~\ref{subsec:test-statistic}), which has a worst-case computational complexity of $O((N+M)^3)$ in contrast to $O((N+M)^2)$ of the MMD test (see Theorem~\ref{thm: computation}). However, the drawback of the test is that it is not usable in practice since the critical level depends on $(\Sigma_R+\lambda I)^{-1/2}$, which is unknown since $R$ is unknown. Therefore, we refer to this test as the \emph{Oracle test.}\vspace{1.5mm}\\
\emph{(iii)} In order to make the Oracle test usable in practice, in Section~\ref{sec:perm}, we propose a permutation test (e.g., see \citealp{lehmann}, \citealp{Pesarin}, and \citealp{permutations}) leading to a critical level that is easy to compute (see Theorem~\ref{thm: permutations typeI}), while still being minimax optimal w.r.t.~$\mathcal{P}$ (see Theorem~\ref{thm: permutations typeII} and Corollaries~\ref{coro:poly:perm},~\ref{coro:exp:perm}).
However, the minimax optimal separation rate is tightly controlled by the choice of the regularization parameter, $\lambda$, which in turn depends on the unknown parameters, $\theta$ and $\beta$ (in the case of the polynomial decay of the eigenvalues of $\mathcal{T}$). This means the performance of the permutation test depends on the choice of $\lambda$. To make the test completely data-driven, in Section~\ref{subsec:adaptation}, we present an aggregate version of the permutation test by aggregating over different $\lambda$ and show the resulting test to be minimax optimal up to a $\log\log$ factor (see Theorems~\ref{thm:perm adp typeI} and \ref{thm: perm adp typeII}). In Section~\ref{subsec:kernel-choice}, we discuss the problem of kernel choice and present an adaptive test by jointly aggregating over $\lambda$ and kernel $K$, which we show to be minimax optimal up to a $\log\log$ factor (see Theorem~\ref{thm: perm adp kernel typeII}). 
\vspace{1.5mm}\\
\emph{(iv)} Through numerical simulations on benchmark data, we demonstrate the superior performance of the spectral regularized test in comparison to the adaptive MMD test \citep{MMDagg}, Energy test \citep{Energy} and Kolmogorov-Smirnov (KS) test \citep{KS,Fasano}, in Section~\ref{sec:experiments}.\vspace{.25mm}\\

\vspace{-3mm}
All these results hinge on Bernstein-type inequalities for the operator norm of a self-adjoint Hilbert-Schmidt operator-valued U-statistics \citep{kpca}. A closely related work to ours is by \citet{Harchaoui} who consider a regularized MMD test with $g_\lambda(x)=\frac{1}{x+\lambda}$ (see Remark~\ref{rem:zaid} for a comparison of our regularized statistic to that of \citealp{Harchaoui}). However, our work deals with general $g_\lambda$, and our test statistic is different from that of \citet{Harchaoui}. In addition, our tests are non-asymptotic and minimax optimal in contrast to that of \citet{Harchaoui}, which only shows asymptotic consistency against fixed alternatives and provides some asymptotic results against local alternatives.

\section{Definitions \& Notation}
For a topological space $\X$, $L^r(\X,\mu)$ denotes the Banach space of $r$-power $(r\geq 1)$ $\mu$-integrable function, where $\mu$ is a finite non-negative Borel measure on $\X$. For $f \in L^r(\X,\mu)=:L^r(\mu)$, $\norm{f}_{L^r(\mu)}:=(\int_{\X}|f|^r\,d\mu)^{1/r}$ denotes the $L^r$-norm of $f$. $\mu^n := \mu \times \stackrel{n}{...} \times \mu$ is the $n$-fold product measure. $\h$ denotes a reproducing kernel Hilbert space with a reproducing kernel $K: \X \times \X \to \R$. $[f]_{\sim}$ denotes the equivalence class of the function $f$, that is the collection of functions $g \in L^r(\X,\mu)$ such that $\norm{f-g}_{L^r(\mu)}=0$. For two measures $P$ and $Q$, $P \ll Q$ denotes that $P$ is dominated by $Q$ which means, if $Q(A)=0$ for some measurable set $A$, then $P(A)=0$.

Let $H_1$ and $H_2$ be abstract Hilbert spaces. $\EuScript{L}(H_1,H_2)$ denotes the space of bounded linear operators from $H_1$ to $H_2$. For $S \in \EuScript{L}(H_1,H_2)$, $S^*$ denotes the adjoint of $S$. $S \in \EuScript{L}(H) := \EuScript{L}(H,H)$ is called self-adjoint if $S^*=S$. For $S \in \EuScript{L}(H)$, $\text{Tr}(S)$, $\norm{S}_{\EuScript{L}^2(H)}$, and $\norm{S}_{\EuScript{L}^{\infty}(H)}$ denote the trace, Hilbert-Schmidt and operator norms of $S$, respectively. For $x,y \in H$, $x \otimes_{H} y$ is an element of the tensor product space of $H \otimes H$ which can also be seen as an operator from $H \to H$ as $(x \otimes_{H}y)z=x\inner{y}{z}_{H}$ for any $z \in H$.

For constants $a$ and $b$, $a \lesssim b$ (resp. $a \gtrsim b$) denotes that there exists a positive constant $c$ (\emph{resp.} $c'$) such that $a\leq cb$ (\emph{resp.} $a \geq c' b)$. $a \asymp b$ denotes that there exists positive constants $c$ and $c'$ such that  $cb \leq a \leq c' b$. We denote $[\ell]$ for $\{1,\ldots,\ell\}$.

\section{Non-optimality of $D^2_{\text{MMD}}$ test}\label{Sec:non-optimal}
In this section, we establish the non-optimality of the test based on $D^2_{\text{MMD}}$. First, we make the following assumption throughout the paper.\vspace{1.5mm}\\
$(A_0)$  $(\mathcal{X} ,\mathcal{B})$ is a 
second countable (i.e., completely separable) space endowed with Borel $\sigma$-algebra $\mathcal{B}$. $(\h,K)$ is an RKHS of real-valued functions on $\X$ with a continuous reproducing kernel $K$ satisfying
$\sup_{x} K(x,x) \leq \K.$ \vspace{1.5mm}\\
The continuity of $K$ ensures that $K(\cdot,x):\X \to \h$ is Bochner-measurable for all $x \in \X$, which along with the boundedness of $K$ ensures that $\mu_P$ and $\mu_Q$ are well-defined \citep{Dincu}. Also the separability of $\mathcal{X}$ along with the continuity of $K$ ensures that $\h$ is separable \citep[Lemma 4.33]{svm}. Therefore, 
\begin{align}
D^2_{\mathrm{MMD}}(P,Q) 
&= \inner{\int_{\X}K(\cdot,x)\, d(P-Q)(x)}{\int_{\X}K(\cdot,x)\, d(P-Q)(x)}_{\h}\nonumber\\
&=4 \inner{\int_{\X}K(\cdot,x)u(x)\, d\PQ(x)}{\int_{\X}K(\cdot,x)u(x)\, d\PQ(x)}_{\h},\label{eq:mmd}
\end{align}
where $R=\frac{P+Q}{2}$ and $u=\frac{dP}{dR}-1$. Define $\id : \h \to \Lp$, $f \mapsto [f - \E_{\PQ}f]_{\sim}$, which is usually referred in the literature as the \emph{inclusion operator} (e.g., see \citealp[Theorem 4.26]{svm}), where $\E_{\PQ}f=\int_{\X}f(x)\,d\PQ(x)$. It can be shown  \citep[Proposition C.2]{kpca} that $\id^* : \Lp \to \h$, $f \mapsto \int K(\cdot,x)f(x)\,d\PQ(x)-\mu_{\PQ} \E_{\PQ}f$. 
Define $\T:=\id \id^* : \Lp \to \Lp$. It can be shown \citep[Proposition C.2]{kpca} that $\T= \Upsilon- (1\ltens1)\Upsilon-\Upsilon(1\ltens1)+(1\ltens1)\Upsilon(1\ltens1)$, where $\Upsilon: \Lp \to \Lp$, $f \mapsto \int K(\cdot,x)f(x)\,d\PQ(x)$. Since $K$ is bounded, it is easy to verify that $\T$ is a trace class operator, and thus compact. Also, it is self-adjoint and positive, thus spectral theorem \citep[Theorems VI.16, VI.17]{Reed} yields that
$$\T = \sum_{i \in I} \lambda_i \Tilde{\phi_i} \ltens \Tilde{\phi_i},$$
where $(\lambda_i)_i \subset \R^+ $ are the eigenvalues and $(\Tilde{\phi}_i)_i$ are the orthonormal system of eigenfunctions (strictly speaking classes of eigenfunctions) of $\T$ that span $\overline{\range(\T)}$ with the index set $I$ being either countable in which case $\lambda_i \to 0$ or finite. In this paper, we assume that the set $I$ is countable, i.e., infinitely many eigenvalues.
Note that $\Tilde{\phi_i}$ represents an equivalence class in $\Lp$. By defining $\phi_i:= \frac{\id^* \Tilde{\phi_i}}{\lambda_i}$, it is clear that $\id\phi_i=[\phi_i-\E_\PQ\phi_i]_{\sim}=\Tilde{\phi_i}$ and $\phi_i \in \h$. Throughout the paper, $\phi_i$ refers to this definition. Using these definitions, we can see that 
\begin{align}
D^2_{\mathrm{MMD}}(P,Q)&= 4 \inner{\id^*u}{\id^*u}_{\h}
=4\inner{\T u}{u}_{\Lp} = 4\sum_{i\geq 1} \lambda_i \langle u,\Tilde{\phi_i}\rangle^2_{\Lp}.\label{Eq:equiv}
\end{align}
\begin{remark}\label{rem:mmd11}
From the form of $D^2_{\mathrm{MMD}}$ in \eqref{eq:mmd}, it  seems more natural to define  $\id : \h \to \Lp$, $f \mapsto [f]_{\sim}$, so that $\id^* : \Lp \to \h$, $f \mapsto \int K(\cdot,x)f(x)\,d\PQ(x)$, leading to $D^2_{\mathrm{MMD}}(P,Q)=4\langle \id^*u,\id^* u\rangle_{\h}$---an expression similar to \eqref{Eq:equiv}. However, since $u \in \emph{\range}(\T^{\theta})$, $\theta>0$ as specified by $\mathcal{P}$, it is clear that $u$ lies in the span of the eigenfunctions of $\T$, while being orthogonal to constant functions in $\Lp$ since $\langle u,1\rangle_{\Lp}=0$. Defining the inclusion operator with centering as proposed under \eqref{eq:mmd} guarantees that the eigenfunctions of $\T$ are orthogonal to constant functions since $\lambda_i\langle 1,\tilde{\phi}_i\rangle_{L^2(R)}=\langle 1,\id\id^*\tilde{\phi}_i\rangle_{L^2(R)}=\langle \id^*1,\id^*\tilde{\phi}_i\rangle_{L^2(R)}=0$, which implies that constant functions are also orthogonal to the space spanned by the eigenfunctions, without assuming that the kernel $K$ is degenerate with respect to $R$, i.e., $\int K(\cdot,x)\,dR(x)=0$. The orthogonality of eigenfunctions to constant functions is crucial in establishing the minimax separation boundary, which relies on constructing a specific example of $u$ from the span of eigenfunctions that is orthogonal to constant functions (see the proof of Theorem \ref{thm:minimax}). On the other hand, the eigenfunctions of $\id\id^*$ with $\id$ as considered in this remark are not guaranteed to be orthogonal to constant functions in $L^2(R)$.
\end{remark}
\vspace{-4mm}
Suppose $u\in \text{span}\{\tilde{\phi}_i:i\in I\}$. Then $\sum_{i\geq1}\langle u,\Tilde{\phi_i}\rangle^2_{\Lp}=\norm{u}_{\Lp}^2\stackrel{(*)}{=}\underline{\rho}^2(P,Q),$ where $\underline{\rho}^2(P,Q):=\frac{1}{2}\int \frac{(dP-dQ)^2}{dP+dQ}$ and $(*)$ follows from Lemma~\ref{Lem: distance} 
by noting that $\Vert u\Vert^2_{L^2(R)}=\chi^2(P||R)$. As mentioned in Section~\ref{sec:intro}, $D_\mathrm{MMD}$ might not capture the difference between between $P$ and $Q$ if they differ in the higher Fourier coefficients of $u$, i.e., $\langle u,\tilde{\phi}_i\rangle_{L^2(R)}$ for large $i$. 

The following result shows that the test based on $\hat{D}_{\mathrm{MMD}}^2$ cannot achieve a separation boundary of order better than $(N+M)^{\frac{-2\theta}{2\theta+1}}$.

\begin{theorem}[Separation boundary of MMD test] \label{thm: MMD}
Suppose $(A_0)$ holds. Let $N\geq 2$, $M\geq 2$, $M \leq N \leq DM$, for some constant $D>1$, $k \in \{1,2\},$ and \begin{equation}\sup_{(P,Q)\in \PP}\norm{\T^{-\theta}u}_{\Lp}<\infty.\label{Eq:rangecond}\end{equation}
Then for any $\alpha>0$, $\delta>0,$  $P_{H_0}\{\hat{D}_{\mathrm{MMD}}^2 \geq \gamma_k\} \leq \alpha,$
$$\inf_{(P,Q)\in \PP}P_{H_1}\{\hat{D}_{\mathrm{MMD}}^2  \geq \gamma_k\} \geq 1-k\delta,\,\,k=1,2,$$
where $\gamma_1 = \frac{2\sqrt{6}\kappa}{\sqrt{\alpha}}\left(\frac{1}{N}+\frac{1}{M}\right)$, $\gamma_2 = q_{1-\alpha},$
$$\Delta_{N,M}:=\Delta=c_k(\alpha,\delta)(N+M)^{\frac{-2\theta}{2\theta+1}},$$ 
$c_1(\alpha,\delta)\asymp\max\{\alpha^{-1/2},\delta^{-1}\}$ and $c_2(\alpha,\delta)\asymp \delta^{-1}\log \frac{1}{\alpha}$, with $q_{1-\alpha}$ being the $(1-\alpha)-$quantile of the permutation function of $\hat{D}_{\mathrm{MMD}}^2$ based on $(N+M)!$ permutations of the samples $\left(\mathbb{X}_N, \mathbb{Y}_M\right)$.   \\
Furthermore, suppose $\Delta_{N,M} (N+M)^{\frac{2\theta}{2\theta+1}} \to 0$ as $N,M \to \infty$ and one of the following holds: (i) $\theta \geq \frac{1}{2}$, (ii) $\sup_{i}\norm{\phi_i}_{\infty} < \infty$, $\theta > 0.$ Then for any decay rate of $(\lambda_i)_i$,  
$$\liminf_{N,M \to \infty} \inf_{(P,Q)\in \PP}P_{H_1}\{\hat{D}_{\mathrm{MMD}}^2  \geq \gamma_k\} < 1.$$\vspace{-5mm}
\end{theorem} 
\vspace{-4mm}
\begin{remark}
(i) The MMD test with threshold $\gamma_1$ is simply based on Chebyshev's inequality, while the data-dependent threshold $\gamma_2$ is based on a permutation test. Theorem~\ref{thm: MMD} shows that both these tests yield a separation radius of $(N+M)^{\frac{-2\theta}{2\theta+1}}$, which in fact also holds if 
$\gamma_2:=q_{1-\alpha}$ is replaced by its Monte-Carlo approximation using only $B$ random permutations instead of all $(M+N)!$ as long as $B$ is large enough. This can be shown using the same approach as in Lemma~\ref{lemma:DKW for quantile}. \vspace{1mm}\\
(ii) Theorem \ref{thm: MMD} shows that the power of the test based on $\hat{D}_{\mathrm{MMD}}^2$ does not go to one, even when $N,M \to \infty$, which implies that asymptotically the separation boundary of such test is of order  $(N+M)^{\frac{-2\theta}{2\theta+1}}$. For the threshold $\gamma_1$, we can also show a non asymptotic result that if $\Delta_{N,M} < d_{\alpha}(N+M)^{\frac{-2\theta}{2\theta+1}}$ for some  $d_{\alpha}>0$, then $\inf_{(P,Q)\in \PP}P_{H_1}\{\hat{D}_{\mathrm{MMD}}^2  \geq \gamma\} < \delta.$  However, for the threshold $\gamma_2$, since our proof technique depends on the asymptotic distribution of $\hat{D}_{\mathrm{MMD}}^2$, the result is presented in the asymptotic setting of $N, M \to \infty$. \vspace{1mm}\\
(iii) The condition in \eqref{Eq:rangecond} implies that $u\in\emph{Ran}(\T^\theta)$ for all $P,Q\in\mathcal{P}$. Note that $\emph{Ran}(\T^{1/2})=\h$, i.e., $u\in \h$ if $\theta=\frac{1}{2}$ and for $\theta>\frac{1}{2}$, $\emph{Ran}(\T^{\theta})\subset\h$. When $\theta<\frac{1}{2}$, $u\in L^2(R)\backslash\h$ with the property that: for all $G>0$, $\exists f \in \h$ such that $\norm{f}_{\h} \leq G$ and $\norm{u-f}_{\Lp}^2\lesssim G^{\frac{-4\theta}{1-2\theta}}$. In other words, $u$ can be approximated by some function in an RKHS ball of radius $G$ with the approximation error decaying polynomially in $G$ \citep[Theorem 4.1]{Cucker}.\vspace{1mm}\\
(iv) The uniform boundedness condition $\sup_i\norm{\phi_i}_{\infty} < \infty$ does not hold in general, for example, see \citet[Theorem 5]{unibound}, which shows that for $\X=S^{d-1}$, where $S^{d-1}$ denotes the $d$-dimensional unit sphere,  $\sup_i\norm{\phi_i}_{\infty}= \infty$, for all $d\geq3$ for any kernel of the form $K(x,y)=f(\langle x,y\rangle_2),\,x,y\in \X$ with $f$ being continuous. An example of such a kernel is the Gaussian kernel on $S^{d-1}$. On the other hand, when $d=2$, the Gaussian kernel satisfies the uniform boundedness condition. Also, when $\X \subset R^d$, the uniform boundedness condition is satisfied by the Gaussian kernel \citep{Steinwart2006AnED}. 
In this paper, we provide results both with and without the assumption of $\sup_i\norm{\phi_i}_{\infty} < \infty$ to understand the impact of the assumption on the behavior of the test. We would like to mention that this uniform boundedness condition has been used in the analysis of the impact of regularization in kernel learning (see \citealp[p. 531]{Mendelson-10}).\vspace{1mm}\\
(v) Theorem~\ref{thm: MMD} can be viewed as a generalization and extension of Theorem 1 in \citet{Krishna}, which shows the separation boundary of the goodness-of-fit test, $H_0:P=P_0$ vs. $H_1:P\ne P_0$ based on a V-statistic estimator of $D_{\mathrm{MMD}}^2$ to be of the order $N^{-1/2}$ when $\theta = \frac{1}{2}$. In their work, the critical level is chosen from the asymptotic distribution of the test statistic under the null with $P_0$ being known, assuming the uniform boundedness of the eigenfunctions of $\T$ and $\int K(\cdot,x)\,dP_0(x)=0$. Note that the zero mean condition is not satisfied by many popular kernels including the Gaussian and Mat\'{e}rn kernels. In contrast, Theorem 3.1 deals with a two-sample setting based on a $U$-statistic estimator of $D_{\mathrm{MMD}}^2$, with no requirement of the uniform boundedness assumption and $\int K(\cdot,x)\,dR(x)=0$, while allowing arbitrary $\theta>0$, and the critical levels being non-asymptotic (permutation and concentration-based).    
\end{remark}

\vspace{-3mm}
The following result provides general conditions on the minimax separation rate w.r.t.~$\mathcal{P}$, which together with Corollaries \ref{coro:poly-minimax} and  \ref{coro:exp-minimax} demonstrates the non-optimality of the MMD tests presented in Theorem~\ref{thm: MMD}.

\begin{theorem}[Minimax separation boundary] \label{thm:minimax}
Suppose $\lambda_i \asymp L(i)$, where $L(\cdot)$ is a strictly decreasing function on $(0,\infty)$, and $M \leq N \leq DM$. Then, for any $0\leq \delta \leq 1-\alpha,$ there exists $c(\alpha,\delta)$  such that if 
$$(N+M)\Delta_{N,M}  \leq  c(\alpha,\delta) \sqrt{\min\left\{L^{-1}\left(\Delta_{N,M}^{1/2\theta}\right),L^{-1}\left(\Delta_{N,M}\right)\right\}}$$ then $$R^*_{\Delta_{N,M}}:= \inf_{\phi \in \Phi_{N,M,\alpha}} R_{\Delta_{N,M}}(\phi)> \delta,$$  where $R_{\Delta_{N,M}}(\phi):=\sup_{(P,Q)\in\PP}\E_{P^N \times Q^M}[1-\phi].$\\
Furthermore if $\sup_{k}\norm{\phi_k}_{\infty} < \infty$, then the above condition on $\Delta_{N,M}$ can be replaced by
$$(N+M)\Delta_{N,M}  \leq  c(\alpha,\delta) \sqrt{\min\left\{L^{-1}\left(\Delta_{N,M}^{1/2\theta}\right),\Delta_{N,M}^{-2}\right\}}.$$\vspace{-4mm}
\end{theorem}

\begin{corollary}[Minimax separation boundary-Polynomial decay] \label{coro:poly-minimax}
Suppose $\lambda_i \asymp i^{-\beta}$, $\beta>1$. Then $$\Delta^{*}_{N,M}  \asymp  (N+M)^{\frac{-4\theta\beta}{4\theta\beta+1}},$$ provided 
one of the following holds: (i) $\theta \geq \frac{1}{2}$, (ii) $\sup_{i}\norm{\phi_i}_{\infty} < \infty$, $\theta \geq \frac{1}{4\beta}.$\vspace{-4mm}
\end{corollary}

\begin{corollary}[Minimax separation boundary-Exponential decay] \label{coro:exp-minimax}
Suppose $\lambda_i \asymp e^{-\tau i}$, $\tau>0$, $\theta > 0$. Then for all $(N+M)\geq k_{\alpha,\delta}$ we have $$\Delta^{*}_{N,M}  \asymp  \frac{\sqrt{\log (N+M)}}{N+M},$$ 
provided 
one of the following holds: (i) $\theta \geq \frac{1}{2}$, (ii) $\sup_{i}\norm{\phi_i}_{\infty} < \infty$, $\theta >0.$

\end{corollary}
\begin{remark} \label{rem:minimax}
(i) Observe that for any bounded kernel satisfying $\sup_{x} K(x,x) \leq \K,$ $\T$ is a trace class operator. This implies $L(\cdot)$ has to satisfy $i L(i)\rightarrow 0$ as $i\rightarrow\infty$. 
Without further assumptions on the decay rate (i.e., if we allow the space $\PP$ to include any decay rate of order $o(i^{-1})$), then we can show that $\Delta^{*}_{N,M}  \asymp  (N+M)^{\frac{-4\theta}{4\theta+1}},$ provided that $\theta \geq \frac{1}{2}$ (or $\theta \geq \frac{1}{4}$ in the case of  $\sup_{i}\norm{\phi_i}_{\infty} < \infty$). However, assuming specific decay rates (i.e., considering a smaller space $\PP$), the separation boundary can be improved, as shown in Corollaries \ref{coro:poly-minimax} and \ref{coro:exp-minimax}. 
Note that $\inf_{\beta>1} \frac{4\theta\beta}{4\theta\beta+1}=\frac{4\theta}{4\theta+1}>\frac{2\theta}{2\theta+1}$ and $1>\frac{2\theta}{2\theta+1}$ for any $\theta>0$, implying that the separation boundary of MMD is larger than the minimax separation boundary w.r.t.~$\mathcal{P}$ irrespective of the decay rate of the eigenvalues of $\mathcal{T}$.\vspace{1mm}\\
(ii) The minimax separation boundary depends only on $\theta$ (the degree of smoothness of $u$), and $\beta$ (the decay rate of the eigenvalues), with $\beta$ controlling the smoothness of the RKHS. While one may think that the minimax rate is independent of $d$, it is actually not the case since $\beta$ depends on $d$. \citet[Section 5.2]{unibound} provides examples of kernels with explicit decay rates of eigenvalues. For example, when $\X=S^{d-1}$, $d\geq 2$, (i) The Spline kernel, defined as $K(x,t)=1+\frac{1}{|S^{d-1}|}\sum_{k=1}^{\infty} \lambda_k c(d,k)P_k(d;\inner{x}{t}_2)$, where $P_k(d;t)$ denotes the Legendre polynomial of degree $k$ in dimension $d,$ and $c(d,k)$ is some normalization constant depending on $d$ and $k$, has $\lambda_k \asymp \left(k(k+d-2)\right)^{-\beta}$, for $\beta > \frac{d-1}{2},$ (ii) The polynomial kernel with degree $h$, defined as $K(x,t)=(1+\inner{x}{t}_2)^h$, has $(k+h+d-2)^{-2h-d+\frac{3}{2}}\lesssim \lambda_k \lesssim (k+h+d-2)^{-h-d+\frac{3}{2}}$, and (iii) The Gaussian kernel with bandwidth $\sigma^2$ satisfies $\lambda_k \asymp \left(\frac{2e}{\sigma^2}\right)^{k}(2k+d-2)^{-k-\frac{d-1}{2}}$, for $\sigma > \sqrt{2/d}.$
\end{remark}

\vspace{-3mm}

\section{Spectral regularized MMD test}\label{Sec:spec}
To address the limitation of the MMD test, in this section, we propose a spectral regularized version of the MMD test and show it to be minimax optimal w.r.t.~$\mathcal{P}$. To this end, we define the \emph{spectral regularized discrepancy} as
$$\eta_{\lambda}(P,Q) : = 4\inner{\T\gl(\T)u}{u}_{\Lp},$$
where the spectral regularizer, $g_\lambda:(0,\infty)\rightarrow (0,\infty)$ satisfies $\lim_{\lambda\rightarrow 0} xg_\lambda(x)\asymp 1$  (more concrete assumptions on $\gl$ will be introduced later). By functional calculus, we define $g_\lambda$ applied to any compact, self-adjoint operator $\mathcal{B}$ defined on a separable Hilbert space, $H$ as 
\begin{equation}g_{\lambda}(\B) : = \sum_{i\geq 1} g_{\lambda}(\tau_i) (\psi_i \otimes_H \psi_i) + g_{\lambda}(0)\left(\Id - \sum_{i\geq 1} \psi_i \otimes_H \psi_i\right),\label{Eq:glambda}\end{equation}
where $\B$ has the spectral representation, $\B= \sum_{i}\tau_i \psi_i \otimes_H \psi_i$ with $(\tau_i,\psi_i)_i$ being the eigenvalues and eigenfunctions of $\B$. A popular example of $g_\lambda$ is $g_\lambda(x)=\frac{1}{x+\lambda}$, yielding $g_\lambda(\mathcal{B})=(\mathcal{B}+\lambda \Id)^{-1}$, which is well known as the \emph{Tikhonov regularizer}. We will later provide more examples of spectral regularizers that satisfy additional assumptions. 
\vspace{-2mm}
\begin{remark}
We would like to highlight that the common definition of $g_\lambda(\mathcal{B})$ in the inverse problem literature (see \citealp[Section 2.3]{Engl.et.al}) does not include the term $g_{\lambda}(0)(\Id - \sum_{i\geq 1} \psi_i \otimes_H \psi_i)$, which represents the projection onto the space orthogonal to $\emph{span}\{\psi_i:i\in I\}$. The reason for adding this term is to ensure that $\gl(\B)$ is invertible whenever $\gl(0) \neq 0$. Moreover, the condition that $\gl(\B)$ is invertible will be essential for the power analysis of our test.
\end{remark}
\vspace{-2mm}
\noindent{Based} on the definition of $g_\lambda(\mathcal{T})$, it is easy to verify that $\T g_\lambda(\T)=\sum_{i\ge 1}\lambda_ig_\lambda(\lambda_i)(\tilde{\phi}_i\otimes_{L^2(R)} \tilde{\phi}_i)$ so that $\langle\T g_\lambda(\T)u,u\rangle_{\Lp}
=\sum_{i\ge 1} \lambda_i g_\lambda(\lambda_i)\langle u,\tilde{\phi}_i\rangle^2_{\Lp}
\rightarrow \sum_{i\ge 1} \langle u,\tilde{\phi}_i\rangle^2_{\Lp}\stackrel{(*)}{=}\Vert u\Vert^2_{\Lp}$ as $\lambda\rightarrow 0$ where $(*)$ holds if $u\in\text{span}\{\tilde{\phi}_i:i\ge 1\}$. In fact, it can be shown that $\Vert u\Vert^2_{L^2(R)}$ and $\langle\T g_\lambda(\T)u,u\rangle_{\Lp}$ are 
equivalent up to constants if $u\in \text{Ran}(\T^\theta)$ and $\lambda$ is large enough compared to $\Vert u\Vert^2_{\Lp}$ (see Lemma~\ref{lemma: bounds for eta}). 
Therefore, the issue with $D_{\text{MMD}}$ can be resolved by using $\eta_\lambda$ as a discrepancy measure to construct a test. In the following, we present details about the construction of the test statistic and the test using $\eta_\lambda$. To this end, we first provide an alternate representation for $\eta_\lambda$ which is very useful to construct the test statistic. Define $\Sigma_R:=\Sigma_{PQ}=\id^*\id:\mathscr{H}\rightarrow\mathscr{H}$, which is referred to as the \emph{covariance operator}. It can be shown \citep[Proposition C.2]{kpca} that $\Sigma_{PQ}$ is a positive, self-adjoint, trace-class operator, and can be written as
\begin{align}
 \Sigma_{PQ}&=\int_{\X} (K(\cdot,x)-\mu_R) \htens (K(\cdot,x)-\mu_R)\, dR(x)\nonumber\\
 &=\frac{1}{2} \int_{\X \times \X}(K(\cdot,x)-K(\cdot,y))\htens(K(\cdot,x)-K(\cdot,y))\,d\PQ(x)\,d\PQ(y),\label{Eq:cov-pop}
\end{align}
where $\mu_R = \int_{\X} K(\cdot,x)\,dR(x)$. Note that
\begin{align}
\eta_\lambda(P,Q)&=4\langle \T g_\lambda(\T)u,u\rangle_{\Lp}\stackrel{(\dagger)}{=}4\langle \id g_\lambda(\Sigma_{PQ})\id^*u,u\rangle_{\Lp}\nonumber\\
&=4\langle g_\lambda(\Sigma_{PQ})\id^*u,\id^*u\rangle_\mathscr{H}=\langle g_\lambda(\Sigma_{PQ}) (\mu_P-\mu_Q),\mu_P-\mu_Q\rangle_\mathscr{H}\nonumber\\
&=\norm{\gSL(\mu_P-\mu_Q)}_{\h}^2,\label{eq:cov-rep}
\end{align}
where $(\dagger)$ follows from Lemma \ref{lemma: bounds for g}(i) 
that states $\T\gl(\T) = \id \gl(\Sigma_{PQ}) \id^*$. Define $\Sigma_{PQ,\lambda} := \Sigma_{PQ}+\lambda \Id$.
\vspace{-3mm}
\begin{remark}\label{rem:zaid}
Suppose $g_\lambda(x) =\frac{1}{x+\lambda}$. Then $\gSL= \SgL$. Note that $\inner{\Sigma_{PQ}f}{f}_{\h}$ $= \norm{f-\E_{\PQ}f}^2_{\Lp}$ for any $f \in \h$, which implies $\inner{\Sigma_{PQ,\lambda}f}{f}_{\h} = \norm{f-\E_{\PQ}f}^2_{\Lp} + \lambda \norm{f}^2_{\h}$. Therefore, $\eta_\lambda(P,Q)$ in \eqref{eq:cov-rep} can be written as
\begin{align}
 \eta_\lambda(P,Q)&=\sup_{f\in\h\,:\,\langle \Sigma_{PQ,\lambda }f,f\rangle_{\h}\leq 1}\langle f,\mu_P-\mu_Q\rangle_{\h} \nonumber\\
 &=\sup_{f \in \h\, :\, \norm{f-\E_{\PQ}f}_{\Lp}^2 + \lambda\norm{f}_{\h}^2\leq 1} \int_{\mathcal{X}} f(x) \,d(P-Q)(x).\label{eq:alter-rep}
\end{align}
This means the regularized discrepancy involves test functions that belong to a growing ball in $\mathscr{H}$ as $\lambda\rightarrow 0$ in contrast to a fixed unit ball as in the case with $D^2_\emph{MMD}$ (see \eqref{eq:variational}). \citet{Krishna} considered a similar discrepancy in a goodness-of-fit test problem, $H_0:P=P_0$ vs. $H_1:P\ne P_0$ where $P_0$ is known, by using $\eta_\lambda(P_0,P)$ in \eqref{eq:alter-rep} but with $R$ being replaced by $P_0$. In the context of two-sample testing, \citet{Harchaoui} considered a discrepancy based on kernel Fisher discriminant analysis whose regularized version is given by 
\begin{align}
&\sup_{0\ne f\in\h}\frac{\langle f,\mu_P-\mu_Q\rangle_{\h}}{\langle f,(\Sigma_P+\Sigma_Q+\lambda \Id)f\rangle_{\h}}=\sup_{f\in\h\,:\,\langle (\Sigma_{P}+\Sigma_Q+\lambda \Id)f,f\rangle_{\h}\leq 1}\langle f,\mu_P-\mu_Q\rangle_{\h}\nonumber\\
&=\sup_{f \in \h\, :\, \Vert f-\E_{P}f\Vert_{L^2(P)}^2 + \Vert f-\E_{Q}f\Vert_{L^2(Q)}^2+\lambda\norm{f}_{\h}^2\leq 1} \int_{\mathcal{X}} f(x) \,d(P-Q)(x),\nonumber
\end{align}
where the constraint set in the above variational form is larger than the one in \eqref{eq:alter-rep} since $\Sigma_{PQ}=\frac{1}{4}\left[2\Sigma_P+2\Sigma_Q+(\mu_P-\mu_Q)\htens(\mu_P-\mu_Q)\right]$.
\end{remark}
\subsection{Test statistic}\label{subsec:test-statistic}
Define $A(x,y):=K(\cdot,x)-K(\cdot,y)$. Using the representation,
\begin{align}
\eta_\lambda(P,Q)=\int_{\X^4} \langle g_\lambda(\Sigma_{PQ}) A(x,y), A(u,w)\rangle_{\h}
\,dP(x)\,dP(u)\,dQ(y)\,dQ(w),\label{Eq:4-int}
\end{align}
obtained by expanding the r.h.s.~of \eqref{eq:alter-rep}, and of $\Sigma_{PQ}$ in \eqref{Eq:cov-pop}, we construct an estimator of $\eta_\lambda(P,Q)$ as follows, based on $\mathbb{X}_N$ and $\mathbb{Y}_M$. We first split the samples $(X_i)_{i=1}^N$ into $(X_i)_{i=1}^{N-s}$ and $(X^1_i)_{i=1}^s:=(X_i)_{i=N-s+1}^N$, and  $(Y_j)_{j=1}^M$ to  $(Y_j)_{j=1}^{M-s}$ and $(Y^1_j)_{j=1}^s:=(Y_j)_{j=M-s+1}^M$. Then, the samples $(X^1_i)_{i=1}^s$ and $(Y^1_j)_{j=1}^s$ are used to estimate the covariance operator $\Sigma_{PQ}$ while $(X_i)_{i=1}^{N-s}$ and $(Y_i)_{i=1}^{M-s}$ are used to estimate the mean elements $\mu_P$ and $\mu_Q$, respectively. Define $n:=N-s$ and $m:=M-s$. Using the form of $\eta_\lambda$ in \eqref{Eq:4-int}, we estimate it using a two-sample $U$-statistic \citep{Hoeffding},
\begin{equation}\label{stat_def}
  \stat:=\frac{1}{n(n-1)}\frac{1}{m(m-1)}\sum_{1\leq i\neq j \leq n}\sum_{1\leq i'\neq j' \leq m} h(X_i,X_j,Y_{i'},Y_{j'}),  
\end{equation}
where $$h(X_i,X_j,Y_{i'},Y_{j'}):= \inner{\gShh A(X_i,Y_{i'})}{\gShh A(X_j,Y_{j'})}_{\h},$$
and 
\begin{equation*}
\hat{\Sigma}_{PQ} :=\frac{1}{2s(s-1)}\sum_{i\neq j}^{s} (K(\cdot,Z_i)-K(\cdot,Z_j)) \htens (K(\cdot,Z_i)-K(\cdot,Z_j)),
\end{equation*}
which is a one-sample $U$-statistic estimator of $\Sigma_{PQ}$ based on $Z_i=\alpha_iX^1_i+(1-\alpha_i)Y^1_i$, for $1\leq i \leq s$, where $(\alpha_i)_{i=1}^s \stackrel{i.i.d.}{\sim} \text{Bernoulli}(\frac{1}{2})$. It is easy to verify that $(Z_i)_{i=1}^s \stackrel{i.i.d.}{\sim} R$.
Note that $\hat{\eta}_\lambda$ is not exactly a $U$-statistic since it involves $\hat{\Sigma}_{PQ}$, but conditioned on $(Z_i)_{i=1}^s$, one can see it is exactly a two-sample $U$-statistic. When $g_\lambda(x)=\frac{1}{x+\lambda}$, in contrast to our estimator which involves sample splitting, \citet{Harchaoui} estimate $\Sigma_P+\Sigma_Q$ using a pooled estimator, and $\mu_P$ and $\mu_Q$ through empirical estimators, using all the samples, thereby resulting in a kernelized version of Hotelling's $T^2$-statistic \citep{lehmann}. However, we consider sample splitting for two reasons: (i) To achieve independence between the covariance operator estimator and the mean element estimators, which leads to a convenient analysis, and (ii) to reduce the computational complexity of $\hat{\eta}_\lambda$ from $(N+M)^3$ to $(N+M)s^2$. By writing \eqref{stat_def} as 
\begin{eqnarray*}
    \stat&{}={}& \frac{1}{n(n-1)}\sum_{i\neq j}\inner{ g_{\lambda}^{1/2}(\hat{\Sigma}_{PQ})\kk(\cdot,X_i)}{g_{\lambda}^{1/2}(\hat{\Sigma}_{PQ})\kk(\cdot,X_j)}_{\h} \\ 
    &&\qquad+ \frac{1}{m(m-1)}\sum_{i\neq j}\inner{ g_{\lambda}^{1/2}(\hat{\Sigma}_{PQ})\kk(\cdot,Y_i)}{g_{\lambda}^{1/2}(\hat{\Sigma}_{PQ})\kk(\cdot,Y_j)}_{\h}\\ 
    &&\qquad\qquad-\frac{2}{nm}\sum_{i,j}\inner{ g_{\lambda}^{1/2}(\hat{\Sigma}_{PQ})\kk(\cdot,X_i)}{g_{\lambda}^{1/2}(\hat{\Sigma}_{PQ})\kk(\cdot,Y_j)}_{\h},
\end{eqnarray*}
the following result shows that $\hat{\eta}_\lambda$ can be computed through matrix operations and by solving a finite-dimensional eigensystem.
\vspace{-3mm}
\begin{theorem} \label{thm: computation}
Let $(\hat{\lambda}_i, \hat{\alpha_i})_i$ be the eigensystem of $\frac{1}{s} \Tilde{\hh}_s^{1/2}K_s\Tilde{\hh}_s^{1/2}$ where $K_s:=[K(Z_i,Z_j)]_{i,j \in [s]}$, $\hh_s= \Id_s -\frac{1}{s}\one_s \one_s^\top$, and $\Tilde{\hh}_s=\frac{s}{s-1}\hh_s$. Define $G := \sum_{i}\left( \frac{g_{\lambda}(\hat{\lambda}_i)-g_{\lambda}(0)}{\hat{\lambda}_i}\right)\hat{\alpha}_i\hat{\alpha}_i^\top.$ Then
\begin{equation*}
\stat=\frac{1}{n(n-1)}\left(\circled{\emph{\small{1}}}-\circled{\emph{\small{2}}}\right)+\frac{1}{m(m-1)}\left(\circled{\emph{\small{3}}}-\circled{\emph{\small{4}}}\right) - \frac{2}{nm}\circled{\emph{\small{5}}}, 
\end{equation*}
where $\circled{\emph{\small{1}}}=\one_n^\top A_1\one_n$, $\circled{\emph{\small{2}}}=\emph{Tr}(A_1)$, $\circled{\emph{\small{3}}}=\one_n^\top A_2\one_n$, $\circled{\emph{\small{4}}}=\emph{Tr}(A_2)$, and $$\circled{\emph{\small{5}}}=\one_m^\top\left(g_{\lambda}(0)K_{mn}+\frac{1}{s}K_{ms}\Tilde{\hh}^{1/2}_{s}G\Tilde{\hh}^{1/2}_{s}K_{ns}^\top\right)\one_n,$$ with $A_1:=g_{\lambda}(0)K_n+\frac{1}{s}K_{ns}\Tilde{\hh}^{1/2}_{s}G\Tilde{\hh}^{1/2}_{s}K_{ns}^\top$ and $A_2:=g_{\lambda}(0)K_m+\frac{1}{s}K_{ms}\Tilde{\hh}^{1/2}_{s}G\Tilde{\hh}^{1/2}_{s}K_{ms}^\top$. Here 
$K_n:=[K(X_i,X_j)]_{i,j \in [n]}$, $K_m:=[K(Y_i,Y_j)]_{i,j \in [m]}$, $[K(X_i,Z_j)]_{i \in [n],j \in [s]}$ $=:K_{ns}$,  \\
$K_{ms}:=[K(Y_i,Z_j)]_{i \in [m],j \in [s]}$, and  $K_{mn}:=[K(Y_i,X_j)]_{i \in [m],j \in [n]}$.
\end{theorem}
\vspace{-3mm}
Note that in the case of Tikhonov regularization, $G=\frac{-1}{\lambda}(\frac{1}{s}\Tilde{\hh}^{1/2}_sK_s\Tilde{\hh}^{1/2}_s+\lambda \Id_s)^{-1}$. The complexity of computing $\stat$ is given by $O(s^3+m^2+n^2+ns^2+ms^2)$, which is comparable to that of the MMD test if $s=o(\sqrt{N+M})$, otherwise the proposed test is computationally more complex than the MMD test.
\vspace{-3mm}
\subsection{Oracle test}\label{subsec:oracle}
Before we present the test, we make the following assumptions on $\gl$ which will be used throughout the analysis.
\begin{align*}
    &(A_1) \, \sup_{x \in \Gamma}|x \gl(x)| \leq C_1, &&(A_2) \, \sup_{x \in \Gamma}|\lambda\gl(x)| \leq C_2, \\ 
    &(A_3) \, \sup_{\{x\in \Gamma: x\gl(x)<B_3\}} |B_3-x\gl(x)|x^{2\varphi} \leq C_3 \lambda^{2\varphi},
    & &(A_4) \, \inf_{x \in \Gamma} \gl(x)(x+\lambda) \geq C_4,
\end{align*}
where $\Gamma : = [0,\K]$, $\varphi \in (0,\xi]$ with the constant $\xi$ being called the \emph{qualification} of $\gl$, and $C_1$, $C_2$, $C_3$, $B_3$, $C_4$ are finite positive constants, all independent of $\lambda>0$. Note that ($A_3$) necessarily implies that $x\gl(x) \asymp 1$ as $\lambda \to 0$ (see Lemma~\ref{lem:cnst}), 
and $\xi$ controls the rate of convergence, which combined with $(A_1)$ yields upper and lower bounds on $\eta_{\lambda}$ in terms of $\norm{u}_{\Lp}^2$ (see Lemma~\ref{lemma: bounds for eta}).
\vspace{-2mm}
\begin{remark}
In the inverse problem literature (see \citealp[Theorems 4.1, 4.3 and
Corollary 4.4]{Engl.et.al}; \citealp[Definition 1]{BAUER200752}), $(A_1)$ and $(A_2)$ are common assumptions with $(A_3)$ being replaced by $\sup_{x\in \Gamma} |1-x\gl(x)|x^{2\varphi} \leq C_3 \lambda^{2\varphi}$. These assumptions are also used in the analysis of spectral regularized kernel ridge regression \citep{BAUER200752}. However, $(A_3)$ is less restrictive than $\sup_{x\in \Gamma} |1-x\gl(x)|x^{2\varphi} \leq C_3 \lambda^{2\varphi}$ and allows for higher qualification for $\gl$. For example, when $\gl(x)=\frac{1}{x+\lambda}$, the condition $\sup_{x\in \Gamma} |1-x\gl(x)|x^{2\varphi} \leq C_3 \lambda^{2\varphi}$ holds only for $\varphi\in(0,\frac{1}{2}]$, while $(A_3)$ holds for any $\varphi>0$ (i.e., infinite qualification with no saturation at $\varphi=\frac{1}{2}$) by setting $B_3 = \frac{1}{2}$ and $C_3=1$, i.e., $\sup_{x\leq \lambda}(\frac{1}{2}-\frac{x}{x+\lambda})x^{2\varphi}\leq \lambda^{2\varphi}$ for all $\varphi>0$. Intuitively, the standard assumption from inverse problem literature is concerned about the rate at which $x\gl(x)$ approaches 1 uniformly, however in our case, it turns out that we are interested in the rate at which $\eta_{\lambda}$ becomes greater than $c\norm{u}_{\Lp}^2$ for some constant $c>0$, leading to a weaker condition. $(A_4)$ is not used in the inverse problem literature but is crucial in our analysis (see Remark~\ref{rem-a4}(iii)).
\end{remark}
\vspace{-3mm}
Some examples of $g_\lambda$ that satisfy $(A_1)$--$(A_4)$ include the Tikhonov regularizer,   $\gl(x)=\frac{1}{x+\lambda}$, and Showalter regularizer, $\gl(x)=\frac{1-e^{-x/\lambda}}{x}\mathds{1}_{\{x\ne 0\}}+\frac{1}{\lambda}\mathds{1}_{\{x=0\}}$,
where both have qualification $\xi=\infty$. Note that the spectral cutoff regularizer is defined as
$\gl(x)= \frac{1}{x}\mathds{1}_{\{x\ge \lambda\}}$
satisfies $(A_1)$--$(A_3)$ with $\xi=\infty$ but unfortunately does not satisfy $(A_4)$ since $g_\lambda(0)=0$. Now, we are ready to present a test based on $\hat{\eta}_\lambda$ where $g_\lambda$ satisfies $(A_1)$--$(A_4)$. Define 
$$\Nol := \text{Tr}(\SgL\Sigma_{PQ}\SgL)\,\,\,\text{and}\,\,\,
\Ntl := \norm{\SgL\Sigma_{PQ}\SgL}_{\hs},$$
which capture the intrinsic dimensionality (or degrees of freedom) of $\h$. $\Nol$ appears quite heavily in the analysis of kernel ridge regression (e.g., \citealp{Caponnetto-07}). The following result provides a critical region with level $\alpha$.
\vspace{-2mm}
\begin{theorem}[Critical region--Oracle] \label{thm: Type1 error}
Let $n \geq 2$ and $m\geq 2$. Suppose $(A_0)$--$(A_2)$ hold. Then for any $\alpha >0$ and $\frac{140\K}{s}\log \frac{48\K s}{\alpha} \leq \lambda \leq \norm{\Sigma_{PQ}}_{\op}$, 
$$P_{H_0}\{\stat \geq \gamma\} \leq \alpha,$$
where $\gamma = \frac{6\sqrt{2}\Cs\Ntl}{\sqrt{\alpha}}\left(\frac{1}{n}+\frac{1}{m}\right).$
Furthermore if $C:=\sup_i\norm{\phi_i}_{\infty} < \infty$, the above bound holds for $136C^2\Nol\log\frac{24\Nol}{\delta} \leq s$ and $\lambda \leq \norm{\Sigma_{PQ}}_{\op}$.
\end{theorem}
First, note that the above result yields an $\alpha$-level test that rejects $H_0$ when $\hat{\eta}_\lambda\ge \gamma$. But the critical level $\gamma$ depends on $\Ntl$ which in turn depends on the unknown distributions $P$ and $Q$. Therefore, we call the above test the \emph{Oracle test}. Later in Sections~\ref{sec:perm} and \ref{subsec:adaptation}, we present a completely data-driven test based on the permutation approach that matches the performance of the Oracle test. Second, the above theorem imposes a condition on $\lambda$ with respect to $s$ in order to control the Type-I error, where this restriction can be weakened if we further assume the uniform boundedness of the eigenfunctions, i.e., $\sup_i\norm{\phi_i}_{\infty} < \infty$. Moreover, the condition on $\lambda$ implies that $\lambda$ cannot decay to zero faster than $\frac{\log s}{s}$. 
Next, we will analyze the power of the Oracle test. Note that 

$P_{H_1}(\hat{\eta}_\lambda\ge \gamma) = P_{H_1}(\hat{\eta}_\lambda-\eta_\lambda\ge \gamma-\eta_\lambda)\ge 1-\text{Var}_{H_1}(\hat{\eta}_\lambda)/(\eta_\lambda-\gamma)^2$, which implies that the power is controlled by $\text{Var}_{H_1}(\hat{\eta}_\lambda)$ and $\eta_{\lambda}$. Lemma~\ref{lemma: bounds for eta} 
shows that $\eta_{\lambda} \gtrsim \norm{u}_{\Lp}^2$, provided that $\gl$ satisfies $(A_3)$, $u \in \text{\range}(\T^{\theta})$, $\theta>0$ and $\underline{\rho}^2(P,Q)\stackrel{(*)}{=}\norm{\U}_{\Lp}^2 \gtrsim \lambda^{2 \min(\theta,\xi)}$, 
where $(*)$ follows from Lemma~\ref{Lem: distance}. 
Combining this bound with the bound on $\lambda$ in Theorem~\ref{thm: Type1 error} provides a condition on the separation boundary. Additional sufficient conditions on $\Delta_{N,M}$ and $\lambda$ are obtained by controlling $\text{Var}_{H_1}(\hat{\eta}_\lambda)/\Vert u\Vert^4_{\Lp}$ to achieve the desired power, which is captured by the following result. 
\vspace{-3mm}
\begin{remark}\label{rem-a4}
(i) Note that while $\emph{Var}_{H_1}(\hat{D}_{\mathrm{MMD}}^2)$ is lower than $\emph{Var}_{H_1}(\stat)$ (see the proofs of Theorem~\ref{thm: MMD} and Lemma~\ref{Lemma: bounding expectations}), 
the rate at which $D^2_{\mathrm{MMD}}$ approaches $\norm{u}^2_{\Lp}$ is much slower than that of $\stat$ (see Lemmas \ref{lem:MMD bounds} and \ref{lemma: bounds for eta}). 
Thus one can think of this phenomenon as a kind of estimation-approximation error trade-off for the separation boundary rate.\vspace{1mm}\\
(ii) Observe from the condition $\underline{\rho}^2(P,Q)=\norm{\U}_{\Lp}^2 \gtrsim \lambda^{2 \min(\theta,\xi)} \Vert\T^{-\theta}\U\Vert_{\Lp}^2$ that larger $\xi$ corresponds to a smaller separation boundary. Therefore, it is important to work with regularizers with infinite qualification, such as Tikhonov and Showalter.\vspace{1mm}\\
(iii) The assumption $(A_4)$ plays a crucial role in controlling the power of the test and in providing the conditions on the separation boundary in terms of $\underline{\rho}^2(P,Q)$. Note that $P_{H_1}(\stat>\gamma)=\mathbb{E}(P(\stat\geq \gamma|(Z_i)_{i=1}^s)\geq\mathbb{E}\left[\mathds{1}_{A}\left(1-\frac{\emph{Var}(\stat|(Z_i)_{i=1}^s)}{(\zeta-\gamma)^2}\right)\right]$ for any set $A$---we choose $A$ to be such that $\norm{\M}^2_{\mathcal{L}^\infty(\mathscr{H})}$ and $\norm{\M^{-1}}^2_{\mathcal{L}^\infty(\mathscr{H})}$ are bounded in probability with $\M:=\hat{\Sigma}^{-1/2}_{PQ,\lambda}\Sigma^{1/2}_{PQ,\lambda}$---, where $\E(\stat |(Z_i)_{i=1}^s)=:\zeta= \Vert\gShh(\mu_P-\mu_Q)\Vert_{\h}^2$.  

If $g_{\lambda}$ satisfies $(A_4)$, then it implies that $\gl(0) \neq 0$, hence $\gSLh$ is invertible. Thus, $\eta_{\lambda} \leq \Vert\gSL\igSLh\Vert_{\op}^2\zeta$. Moreover, $(A_4)$ yields that 
$\zeta \gtrsim \eta_{\lambda}$ with high probability (see Lemma \ref{lemma: denominator lower bound},
\citealp[Lemma A.1(ii)]{kpca}), which when combined with Lemma~\ref{lemma: bounds for eta}, 
yields sufficient conditions on the separation boundary to control the power.\vspace{1mm}\\
(iv) As aforementioned, the spectral cutoff regularizer does not satisfy $(A_4)$. For $g_\lambda$ that does not satisfy $(A_4)$, an alternative approach can be used to obtain a lower bound on $\zeta$. Observe that $\zeta = \langle\id(\gSh-\gS)\id^*u,u\rangle_{\Lp}+ \eta_{\lambda}$, hence $$\zeta \geq \eta_{\lambda}-\Vert\id(\gSh-\gS)\id^*\Vert_{\op}\norm{u}_{\Lp}^2.$$  However, the upper bound that we can achieve on $\Vert\id(\gSh-\gS)\id^*\Vert_{\op}$ is worse than the bound on $\Vert\gSL\igSLh\Vert_{\op}^2$, eventually leading to a worse separation boundary. Improving such bounds (if possible) and investigating this problem can be of independent interest and left for future analysis.
\end{remark}

\vspace{-2mm}
For the rest of the paper, we make the following assumption.\vspace{1.5mm}\\
$(B)\qquad$ $M<N<DM$ for some constant $D\geq1$.\vspace{1.5mm}\\
This assumption helps to keep the results simple and presents the separation rate in terms of $N+M$. If this assumption is not satisfied, the analysis can still be carried out but leads to messy calculations with the separation rate depending on $\text{min}(N,M)$. 

\vspace{-2mm}
\begin{theorem}[Separation boundary--Oracle] \label{thm:Type II}
Suppose $(A_0)$--$(A_4)$ and $(B)$ hold. Let  $s=d_1N=d_2M$, $\sup_{(P,Q) \in \PP} \norm{\T^{-\theta}u}_{\Lp} < \infty,$ $\norm{\Sigma_{PQ}}_{\op}\geq\lambda=d_{\theta}\Delta_{N,M}^{\frac{1}{2\Tilde{\theta}}}$, for some constants $0<d_1,d_2<1$ and $d_{\theta}>0$, where $d_\theta$ is a constant that depends on $\theta$. For any $0<\delta\le 1$, if  $N+M \geq \frac{32\K d_1}{\delta}$ and $\Delta_{N,M}$ satisfies 
$$\frac{\Delta_{N,M}^{\frac{2\Tilde{\theta}+1}{2\tilde{\theta}}}}{\mathcal{N}_{2}\left(d_{\theta}\Delta_{N,M}^{\frac{1}{2\tilde{\theta}}}\right)} \gtrsim \frac{d_{\theta}^{-1}\delta^{-2}}{ (N+M)^{2}}, \qquad\,
\frac{\Delta_{N,M}}{\mathcal{N}_{2}\left(d_{\theta}\Delta_{N,M}^{\frac{1}{2\tilde{\theta}}}\right)} \gtrsim\frac{(\alpha^{-1/2}+\delta^{-1})}{ N+M},$$ and 
$$\Delta_{N,M} \geq c_{\theta} \left(\frac{N+M}{\log(N+M)}\right)^{-2\tilde{\theta}},$$ 
then 
\begin{equation}\inf_{(P,Q) \in \PP} P_{H_1}\left\{\stat
\geq \gamma \right\} \geq 1-2\delta,\label{Eq:type-2}\end{equation}
where $\gamma = \frac{6\sqrt{2}\Cs\Ntl}{\sqrt{\alpha}}\left(\frac{1}{n}+\frac{1}{m}\right)$, $c_{\theta}>0$ is a constant that depends on $\theta$, and $\Tilde{\theta}=\min(\theta,\xi)$. 

Furthermore, suppose $N+M \geq \max\{32\delta^{-1},e^{d_1/272C^2}\}$ and $C:=\sup_{i}\norm{\phi_i}_{\infty}$ $< \infty$. Then \eqref{Eq:type-2} holds when the above  conditions on $\Delta_{N,M}$ are replaced by
$$\frac{\Delta_{N,M}}{\mathcal{N}_{1}\left(d_{\theta}\Delta_{N,M}^{\frac{1}{2\tilde{\theta}}}\right)} \gtrsim \frac{\delta^{-2}}{(N+M)^{2}}, \qquad
\frac{\Delta_{N,M}}{\mathcal{N}_{2}\left(d_{\theta}\Delta_{N,M}^{\frac{1}{2\tilde{\theta}}}\right)} \gtrsim \frac{(\alpha^{-1/2}+\delta^{-1})}{ N+M},$$ and  
$$\mathcal{N}_{1}\left(d_{\theta}\Delta_{N,M}^{\frac{1}{2\tilde{\theta}}}\right) \lesssim \left(\frac{N+M}{\log(N+M)}\right).$$
\end{theorem}
The above result is too general to appreciate the performance of the Oracle test. The following corollaries to Theorem~\ref{thm:Type II} investigate the separation boundary of the test under the polynomial and exponential decay condition on the eigenvalues
of $\Sigma_{PQ}$.

\begin{corollary}[Polynomial decay--Oracle] \label{coro:poly}
Suppose $\lambda_i \lesssim i^{-\beta}$, $\beta>1$. Then there exists $k_{\theta,\beta} \in \N$ such that for all $N+M \geq k_{\theta,\beta}$ and for any $\delta>0$, 
$$\inf_{(P,Q) \in \PP} P_{H_1}\left\{\stat
\geq \gamma \right\} \geq 1-2\delta,$$
when
$$\Delta_{N,M} =
\left\{
	\begin{array}{ll}
c(\alpha,\delta,\theta)\left(N+M\right)^{\frac{-4\tilde{\theta}\beta}{4\Tilde{\theta}\beta+1}},  &  \ \ \Tilde{\theta}> \frac{1}{2}-\frac{1}{4\beta} \\		
c(\alpha,\delta,\theta)\left(\frac{\log (N+M)}{N+M}\right)^{2\tilde{\theta}}, & \ \ \Tilde{\theta} \leq \frac{1}{2}-\frac{1}{4\beta}
	\end{array}
\right.,$$
with $c(\alpha,\delta,\theta)\gtrsim(\alpha^{-1/2}+\delta^{-2}+d_3^{2\tilde{\theta}})$ for some constant $d_3>0$. Furthermore, if $\sup_{i}\norm{\phi_i}_{\infty} < \infty$, then 
$$\Delta_{N,M} =
\left\{
	\begin{array}{ll}
c(\alpha,\delta,\theta,\beta)\left(N+M\right)^{\frac{-4\tilde{\theta}\beta}{4\Tilde{\theta}\beta+1}},  &  \ \  \Tilde{\theta}>\frac{1}{4\beta} \\
c(\alpha,\delta,\theta,\beta)\left(\frac{\log(N+M)}{N+M}\right)^{2\tilde{\theta}
		\beta}, & \ \  \Tilde{\theta} \leq \frac{1}{4\beta}
	\end{array}
\right.,$$
where $c(\alpha,\delta,\theta,\beta)\gtrsim(\alpha^{-1/2}+\delta^{-2}+d_4^{2\tilde{\theta}\beta})$ for some constant $d_4>0$.
\end{corollary}

\begin{corollary}[Exponential decay--Oracle] \label{coro:exp}
Suppose $\lambda_i \lesssim e^{-\tau i}$, $\tau>0$. Then for any $\delta>0$, there exists $k_{\theta}$ such that for all $N+M\geq k_{\theta}$, 
$$\inf_{(P,Q) \in \PP} P_{H_1}\left\{\stat
\geq \gamma \right\} \geq 1-2\delta,$$
when
$$\Delta_{N,M} =
\left\{
	\begin{array}{ll}
	c(\alpha,\delta,\theta)\frac{\sqrt{\log(N+M)}}{N+M},  &  \ \ \Tilde{\theta}> \frac{1}{2} \\
		c(\alpha,\delta,\theta)\left(\frac{\log(N+M)}{N+M}\right)^{2\tilde{\theta}}, & \ \  \Tilde{\theta} \le \frac{1}{2}
	\end{array}
\right.,$$
where $c(\alpha,\delta,\theta)\gtrsim \max\left\{\sqrt{\frac{1}{2\tilde{\theta}}},1\right\}(\alpha^{-1/2}+\delta^{-2}+d_5^{2\tilde{\theta}})$ for some $d_5>0$. Furthermore, if $\sup_{i}\norm{\phi_i}_{\infty} < \infty$, then 
$$\Delta_{N,M} = c(\alpha,\delta,\theta)\frac{\sqrt{\log(N+M)}}{N+M},$$
where $c(\alpha,\delta,\theta) \gtrsim \max\left\{\sqrt{\frac{1}{2\tilde{\theta}}},\frac{1}{2\tilde{\theta}},1\right\}(\alpha^{-1/2}+\delta^{-2}).$
\end{corollary}
\vspace{-6mm}
\begin{remark}\label{rem:choosing s}
(i) For a bounded kernel $K$, note that $\T$ is trace class operator. With no further assumption on the decay of $(\lambda_i)_i$, it can be shown that the separation boundary has the same expression as in Corollary \ref{coro:poly} for $\beta=1$ (see Remark \ref{rem:minimax}(i) for the minimax separation). Under additional assumptions on the decay rate, the separation boundary improves (as shown in the above corollaries) unlike the separation boundary in Theorem \ref{thm: MMD} which does not capture the decay rate.\vspace{1mm}\\
(ii) Suppose $g_\lambda$ has infinite qualification, $\xi=\infty$, then $\tilde{\theta}=\theta$. Comparison of Corollaries~\ref{coro:poly} and \ref{coro:poly-minimax} shows that the oracle test is minimax optimal w.r.t.~$\mathcal{P}$ in the ranges of $\theta$ as given in Corollary~\ref{coro:poly-minimax} if the eigenvalues of $\T$ decay polynomially. Similarly, if the eigenvalues of $\T$ decay exponentially, it follows from Corollary~\ref{coro:exp} and Corollary~\ref{coro:exp-minimax} that the Oracle test is minimax optimal w.r.t.~$\mathcal{P}$ if $\theta>\frac{1}{2}$ (\emph{resp.} $\theta>0$ if $\sup_{i}\norm{\phi_i}_{\infty} < \infty$). Outside these ranges of $\theta$, the optimality of the oracle test remains an open question since we do not have a minimax separation covering these ranges of $\theta$. 
\vspace{1mm}\\
(iii) On the other hand, if $g_\lambda$ has a finite qualification, $\xi$, then the test does not capture the smoothness of $u$ beyond $\xi$, i.e., the test only captures the smoothness up to $\xi$, which implies the test is minimax optimal only for $\theta\le \xi$. Therefore, it is important to use spectral regularizers with infinite qualification.\vspace{1mm}\\
(iv) Note that the splitting choice $s=d_1N=d_2M$ yields that the complexity of computing the test statistic is of the order $(N+M)^3$. However, it is worth noting that such a choice is just to keep the splitting method independent of $\theta$ and $\beta$. But, in practice, a smaller order of $s$ still performs well (see Section~\ref{sec:experiments}). This can be theoretically justified by following the proof of Theorem \ref{thm:Type II} and its application to Corollaries \ref{coro:poly} and \ref{coro:exp}, that for the polynomial decay of eigenvalues if $\theta > \frac{1}{2}-\frac{1}{4\beta}$, we can choose $s \asymp (N+M)^{\frac{2\beta}{4\tilde{\theta}\beta+1}}$ and still achieve the same separation boundary (up to log factor) and furthermore if $\sup_i\norm{\phi_i}_{\infty} < \infty$ and $\theta > \frac{1}{4\beta}$, then we can choose $s\asymp (N+M)^{\frac{2}{4\tilde{\theta}\beta+1}}$. Thus, as $\theta$ increases $s$ can be of a much lower order than $N+M$. Similarly, for the exponential decay case, when $\theta>\frac{1}{2}$, we can choose $s \asymp (N+M)^{\frac{1}{2\tilde{\theta}}}$ and still achieve the same separation boundary (up to $\sqrt{\log}$ factor), and furthermore if $\sup_i\norm{\phi_i}_{\infty} < \infty$, then for $\theta \geq \frac{\log(N+M) \log(\log(N+M))}{N+M}$, we can choose $s\asymp \frac{1}{\theta} \log(N+M) \log(\log(N+M))$ and achieve the same separation boundary.\vspace{1mm}\\
(v) The key intuition in using the uniform bounded condition, i.e., $\sup_i\Vert \phi_i\Vert_\infty<\infty$ is that it provides a reduction in the variance when applying Chebyshev's  (or Bernstein's) inequality, which in turn provides an improvement in the separation rate, as can be seen in Corollaries~\ref{coro:poly-minimax}, \ref{coro:exp-minimax}, \ref{coro:poly}, and \ref{coro:exp}, wherein the minimax optimal rate is valid for a large range of $\theta$ compared to the case where this assumption is not made.

\end{remark}

\vspace{-7mm}
\subsection{Permutation test}\label{sec:perm}
In the previous section, we established the minimax optimality w.r.t.~$\mathcal{P}$ of the regularized test based on $\hat{\eta}_\lambda$. However, this test is not practical because of the dependence of the threshold $\gamma$ on $\Ntl$ which is unknown in practice since we do not know $P$ and $Q$. 
In order to achieve a more practical threshold, one way is to estimate $\Ntl$ from data and use the resultant critical value to construct a test. However, in this section, we resort to ideas from permutation testing (\citealp{lehmann}, \citealp{Pesarin}, \citealp{permutations}) to construct a data-driven threshold. Below, we first introduce the idea of permutation tests, then present a permutation test based on $\hat{\eta}_\lambda$, and provide theoretical results that such a test can still achieve minimax optimal separation boundary w.r.t.~$\mathcal{P}$, and in fact with a better dependency on $\alpha$ of the order $\log\frac{1}{\alpha}$ compared to $\frac{1}{\sqrt{\alpha}}$ of the Oracle test.

Recall that our test statistic defined in Section~\ref{subsec:test-statistic} involves sample splitting resulting in three sets of independent samples,  $(X_i)_{i=1}^n \stackrel{i.i.d.}{\sim} P$, $(Y_j)_{j=1}^m \stackrel{i.i.d.}{\sim} Q$, $(Z_i)_{i=1}^s \stackrel{i.i.d.}{\sim} \frac{P+Q}{2}$. Define $(U_i)_{i=1}^n:=(X_i)_{i=1}^n$, and $(U_{n+j})_{j=1}^m:=(Y_j)_{j=1}^m$. Let $\Pi_{n+m}$ be the set of all possible permutations of $\{1,\ldots,n+m\}$ with $\pi \in \Pi_{n+m}$ be a randomly selected permutation from the $D$ possible permutations, where $D :=|\Pi_{n+m}|= (n+m)!$. Define $(X^{\pi}_i)_{i=1}^n := (U_{\pi(i)})_{i=1}^n$ and $(Y^{\pi}_j)_{j=1}^m := (U_{\pi(n+j)})_{j=1}^m$. Let $\hat{\eta}^{\pi}_{\lambda}:=\stat(X^{\pi},Y^{\pi},Z)$ be the statistic based on the permuted samples. Let $(\pi^i)_{i=1}^B$ be $B$ randomly selected permutations from $\Pi_{n+m}$. For simplicity, define $\hat{\eta}^i_{\lambda}:= \hat{\eta}^{\pi^i}_{\lambda}$ to represent the statistic based on permuted samples w.r.t.~the random permutation $\pi^i$. Given the samples $(X_i)_{i=1}^n$, $(Y_j)_{j=1}^m$ and $(Z_i)_{i=1}^s$, define $$F_{\lambda}(x):= \frac{1}{D}\sum_{\pi \in \Pi_{n+m}}\II(\hat{\eta}^{\pi}_{\lambda} \leq x)$$ to be the permutation distribution function, and define $$q_{1-\alpha}^{\lambda}:= \inf\{q \in \R: F_{\lambda}(q) \geq 1-\alpha\}.$$ Furthermore, we define the empirical permutation distribution function based on $B$ random permutations as
$$\hat{F}^{B}_{\lambda}(x):= \frac{1}{B}\sum_{i=1}^{B}\II(\hat{\eta}^{i}_{\lambda} \leq x),$$ and define $$\hat{q}_{1-\alpha}^{B,\lambda}:= \inf\{q \in \R: \hat{F}^B_{\lambda}(q) \geq 1-\alpha\}.$$ Based on these notations, the following result presents an $\alpha$-level test with a completely data-driven critical level.

\begin{theorem}[Critical region--permutation]\label{thm: permutations typeI}
For any $0<\alpha\leq 1$ and $0<w+\Tilde{w}<1$, if $B\geq  \frac{1}{2\Tilde{w}^2\alpha^2}\log\frac{2}{\alpha(1-w-\Tilde{w})}$, then
$$P_{H_0}\{\stat \geq \hat{q}_{1-w\alpha}^{B,\lambda} \} \leq \alpha.$$
\end{theorem}

Note that the above result holds for any statistic, not necessarily $\hat{\eta}_\lambda$, thus it does not require any assumption on $\gl$ as opposed to Theorem~\ref{thm: Type1 error}. This follows from the exchangeability of the samples under $H_0$ and the way $\qqe$ is defined, thus it is well known that this approach will yield an $\alpha$-level test when using $\qqe$ as the threshold. 

\vspace{-3mm}
\begin{remark}
(i) The requirement on $B$ in Theorem~\ref{thm: permutations typeI} is to ensure the proximity of $\qqh$ to $\qqe$ (refer to Lemma~\ref{lemma:DKW for quantile}), through an application of Dvoretzky-Kiefer-Wolfowitz (DKW) inequality \citep{DKW}. The parameters $w$ and $\tilde{w}$ introduced in the statement of Theorem~\ref{thm: permutations typeI} appear for technical reasons within the proof of Lemma~\ref{lemma:DKW for quantile} since Lemma~\ref{lemma:DKW for quantile} does not directly establish a relationship between the true quantile $\qqe$ and the estimated quantile $\qqh$. Instead, it associates the true quantile with an arbitrary small shift in $\alpha$ to the estimated quantile, hence the inclusion of these parameters. However, in practice, we rely solely on the quantile $\qqh$, which will be very close to the true quantile $\qqe$ for a sufficiently large $B$.\vspace{1mm}

(ii) Another approach to construct $\qqh$ is by using $\hat{F}^B_\lambda(x):=\frac{1}{B+1}\sum^B_{i=0} \mathds{1}(\hat{\eta}^i_\lambda\le x)$ instead of the above definition, where the new definition of $\hat{F}^B_\lambda$ includes the unpermuted statistic $\stat$ 
 (see \citealt[Lemma 1]{Romano} and \citealt[Proposition 1]{Albert}). The advantage of this new construction is that the Type-I error is always smaller than $\alpha$ for all $B$, i.e., no condition is needed on $B$. However, a condition on $B$ similar to that in Theorem~\ref{thm: permutations typeII} is anyway needed to achieve the required power. 
 Therefore, the condition on $B$ in Theorem \ref{thm: permutations typeI} does not impose an additional constraint in the power analysis. In practice, we observed that the new approach requires a significantly large $B$ to achieve similar power, leading to an increased computational requirement. 
 Thus, we use the construction in 
 Theorem \ref{thm: permutations typeI} in our numerical experiments
\end{remark}
\vspace{-3mm}
Next, similar to Theorem \ref{thm:Type II}, in the following result, we  
give the general conditions under which the power level can be controlled.

\begin{theorem}[Separation boundary--permutation] \label{thm: permutations typeII}
Suppose $(A_0)$--$(A_4)$ and $(B)$ hold. Let  $s=d_1N=d_2M$, $\sup_{(P,Q) \in \PP} \norm{\T^{-\theta}u}_{\Lp} < \infty,$ $\norm{\Sigma_{PQ}}_{\op} \geq\lambda=d_{\theta}\Delta_{N,M}^{\frac{1}{2\Tilde{\theta}}}$, for some constants $0<d_1,d_2<1$ and $d_{\theta}>0$, where $d_\theta$ is a constant that depends on $\theta$. For any $0<\delta\le 1$, if  $(N+M) \geq \max\left\{d_3\delta^{-1/2}\log\frac{1}{(w-\tilde{w})\alpha},32\K d_1\delta^{-1}\right\}$ for some $d_3>0$, $B\geq \frac{1}{2\Tilde{w}^2\alpha^2}\log\frac{2}{\delta}$ for any $0<\Tilde{w}<w <1$ and $\Delta_{N,M}$ satisfies 
$$\frac{\Delta_{N,M}^{\frac{2\Tilde{\theta}+1}{2\tilde{\theta}}}}{\mathcal{N}_{2}\left(d_{\theta}\Delta_{N,M}^{\frac{1}{2\tilde{\theta}}}\right)} \gtrsim \frac{d_{\theta}^{-1}(\delta^{-1}\log(1/\tilde{\alpha}))^2}{ (N+M)^{2}},\qquad\,
\frac{\Delta_{N,M}}{\mathcal{N}_{2}\left(d_{\theta}\Delta_{N,M}^{\frac{1}{2\tilde{\theta}}}\right)} \gtrsim\frac{\delta^{-1}\log(1/\tilde{\alpha})}{ N+M},$$
$$\Delta_{N,M} \geq c_{\theta} \left(\frac{N+M}{\log(N+M)}\right)^{-2\tilde{\theta}},$$
then 
\begin{equation}\inf_{(P,Q) \in \PP} P_{H_1}\left\{\stat
\geq \hat{q}_{1-w\alpha}^{B,\lambda} \right\} \geq 1-5\delta,\label{Eq:type-2-perm}\end{equation}
where $c_{\theta}>0$ is a constant that depends on $\theta$, $\tilde{\alpha} = (w-\tilde{w})\alpha$ and $\Tilde{\theta}=\min(\theta,\xi)$.

Furthermore, suppose  $C:=\sup_{i}\norm{\phi_i}_{\infty} < \infty$ and $N+M \geq \max\left\{\frac{d_3}{\sqrt{\delta}}\log\frac{1}{(w-\tilde{w})\alpha}, \frac{32}{\delta},e^{\frac{d_1}{272C}}\right\}$. Then \eqref{Eq:type-2-perm} holds when the above  conditions on $\Delta_{N,M}$ are replaced by
$$\frac{\Delta_{N,M}}{\mathcal{N}_{1}\left(d_{\theta}\Delta_{N,M}^{\frac{1}{2\tilde{\theta}}}\right)} \gtrsim \frac{(\delta^{-1}\log(1/\tilde{\alpha}))^2}{(N+M)^{2}},\qquad\frac{\Delta_{N,M}}{\mathcal{N}_{2}\left(d_{\theta}\Delta_{N,M}^{\frac{1}{2\tilde{\theta}}}\right)} \gtrsim \frac{\delta^{-1}\log(1/\tilde{\alpha})}{ N+M},$$
$$\text{and}\qquad\frac{1}{\mathcal{N}_{1}\left(d_{\theta}\Delta_{N,M}^{\frac{1}{2\tilde{\theta}}}\right)} \gtrsim \left(\frac{N+M}{\log(N+M)}\right)^{-1}.$$
\end{theorem}
The following corollaries 
specialize Theorem~\ref{thm: permutations typeII} for the case of polynomial and exponential 
decay of eigenvalues of $\T$.
\begin{corollary}[Polynomial decay--permutation] \label{coro:poly:perm}
Suppose $\lambda_i \lesssim i^{-\beta}$, $\beta>1$. Then there exists $k_{\theta,\beta} \in \N$ such that for all $N+M \geq k_{\theta,\beta}$ and for any $\delta>0$, 
$$\inf_{(P,Q) \in \PP} P_{H_1}\left\{\stat
\geq \hat{q}_{1-w\alpha}^{B,\lambda} \right\} \geq 1-5\delta,$$
when
$$\Delta_{N,M} =
\left\{
	\begin{array}{ll}
c(\alpha,\delta,\theta)\left(N+M\right)^{\frac{-4\tilde{\theta}\beta}{4\Tilde{\theta}\beta+1}},  &  \ \ \Tilde{\theta}> \frac{1}{2}-\frac{1}{4\beta} \\		
c(\alpha,\delta,\theta)\left(\frac{\log (N+M)}{N+M}\right)^{2\tilde{\theta}}, & \ \ \Tilde{\theta} \leq \frac{1}{2}-\frac{1}{4\beta}
	\end{array}
\right.,$$
with $c(\alpha,\delta,\theta)\gtrsim\max\{\delta^{-2}(\log \frac{1}{\alpha})^2,d_4^{2\tilde{\theta}}\}$ for some constant $d_4>0$. Furthermore, if $\sup_{i}\norm{\phi_i}_{\infty}$ $< \infty$, then 
$$\Delta_{N,M} =
\left\{
	\begin{array}{ll}
c(\alpha,\delta,\theta,\beta)\left(N+M\right)^{\frac{-4\tilde{\theta}\beta}{4\Tilde{\theta}\beta+1}},  &  \ \  \Tilde{\theta}>\frac{1}{4\beta} \\
c(\alpha,\delta,\theta,\beta)\left(\frac{\log(N+M)}{N+M}\right)^{2\tilde{\theta}
		\beta}, & \ \  \Tilde{\theta} \leq \frac{1}{4\beta}
	\end{array}
\right.,$$
where $c(\alpha,\delta,\theta,\beta)\gtrsim \max\{\delta^{-2}(\log \frac{1}{\alpha})^2,d_5^{2\tilde{\theta}\beta}\}$ for some constant $d_5>0$.
\end{corollary}

\begin{corollary}[Exponential decay--permutation] \label{coro:exp:perm}
Suppose 
$\lambda_i \lesssim e^{-\tau i}$,
$\tau>0$. Then for any $\delta>0$, there exists $k_{\theta}$ such that for all $N+M\geq k_{\theta}$, 
$\inf_{(P,Q) \in \PP} P_{H_1}\left\{\stat
\geq \hat{q}_{1-w\alpha}^{B,\lambda} \right\} \geq 1-5\delta$ 
when
$$\Delta_{N,M} =
\left\{
	\begin{array}{ll}
	c(\alpha,\delta,\theta)\frac{\sqrt{\log(N+M)}}{N+M},  &  \ \ \Tilde{\theta}> \frac{1}{2} \\
		c(\alpha,\delta,\theta)\left(\frac{\log(N+M)}{N+M}\right)^{2\tilde{\theta}}, & \ \  \Tilde{\theta} \le \frac{1}{2}
	\end{array}
\right.,$$
where $c(\alpha,\delta,\theta) \gtrsim \max\left\{\sqrt{\frac{1}{2\tilde{\theta}}},1\right\}\max\left\{\delta^{-2}(\log \frac{1}{\alpha})^2,d_6^{2\tilde{\theta}}\right\}$ for some constant $d_6>0$. Furthermore, if $\sup_{i}\norm{\phi_i}_{\infty} < \infty$, then 

$$\Delta_{N,M} = c(\alpha,\delta,\theta)\frac{\sqrt{\log(N+M)}}{N+M},$$
where $c(\alpha,\delta,\theta) \gtrsim \max\left\{\sqrt{\frac{1}{2\tilde{\theta}}},\frac{1}{2\tilde{\theta}},1\right\}\delta^{-2}(\log \frac{1}{\alpha})^2.$
\end{corollary}

These results show that the permutation-based test constructed in Theorem~\ref{thm: permutations typeI} is minimax optimal w.r.t.~$\mathcal{P}$, matching the rates of the Oracle test with a completely data-driven test threshold. The computational complexity of the test increases to $O(Bs^3+Bm^2+Bn^2+Bns^2+Bms^2)$ as the test statistic is computed $B$ times to calculate the threshold $\qqh$. However, since the test can be parallelized over the permutations, the computational complexity in practice is still the complexity of one permutation.

\subsection{Adaptation}\label{subsec:adaptation}
While the permutation test defined in the previous section provides a practical test, the choice of $\lambda$ that yields the minimax separation boundary depends on the prior knowledge of $\theta$ (and $\beta$ in the case of polynomially decaying eigenvalues). In this section, we construct a test based on the union (aggregation) of multiple tests for different values of $\lambda$ taking values in a finite set, $\Lambda$, that guarantees to be minimax optimal (up to $\log$ factors) for a wide range of $\theta$ (and $\beta$ in case of polynomially decaying eigenvalues). 

Define $\Lambda :=\{\lambda_L, 2\lambda_L, ... \,, \lambda_U\},$ where $\lambda_U=2^b\lambda_L$, for $b \in \N$. Clearly $|\Lambda|=b+1=1+\log_2\frac{\lambda_U}{\lambda_L}$ is the cardinality of $\Lambda$. Let $\lambda^*$ be the optimal $\lambda$ that yields minimax optimality. The main idea is to choose $\lambda_L$ and $\lambda_U$ to ensure that there is an element in $\Lambda$ that is close to $\lambda^*$ for any $\theta$ (and $\beta$ in case of polynomially decaying eigenvalues). Define $v^* := \sup\{x \in \Lambda: x \leq \lambda^*\}$. Then it is easy to see that for $\lambda_L\leq \lambda^* \leq \lambda_U$, we have $\frac{\lambda^*}{2}\leq v^* \leq \lambda^*$, in other words $v^* \asymp \lambda^*$. Thus, $v^*$ is also an optimal choice for $\lambda$ that belongs to $\Lambda$. Motivated by this, in the following, we construct an $\alpha$-level test based on the union of the tests over $\lambda \in \Lambda$ that rejects $H_0$ if one of the tests rejects $H_0$, which is captured by Theorem~\ref{thm:perm adp typeI}. The separation boundary of this test is analyzed in Theorem~\ref{thm: perm adp typeII} under the polynomial and the exponential decay rates
of the eigenvalues of $\T$, showing that the adaptive test achieves the same performance (up to a $\log\log$ factor) as that of the Oracle test, i.e., minimax optimal w.r.t.~$\mathcal{P}$ over the range of $\theta$ mentioned in Corollaries~\ref{coro:poly-minimax}, \ref{coro:exp-minimax} without requiring the knowledge of $\lambda^*$.

\begin{theorem}[Critical region--adaptation] \label{thm:perm adp typeI}
For any $0<\alpha\leq 1$ and $0<w+\Tilde{w}<1$, if $B\geq \frac{\cd^2}{2\Tilde{w}^2\alpha^2}\log\frac{2\cd}{\alpha(1-w-\Tilde{w})}$, 
then 
$$P_{H_0}\left\{\bigcup_{\lambda \in \Lambda}\stat \geq \hat{q}_{1-\frac{w\alpha}{\cd}}^{B,\lambda} \right\} \leq \alpha.$$
\end{theorem}
\begin{theorem}[Separation boundary--adaptation] \label{thm: perm adp typeII}
Suppose 
$(A_0)$--$(A_4)$ and $(B)$ hold. Let 
$\tilde{\theta}=\min(\theta,\xi),$ $s=e_1N=e_2M$ for $0<e_1,e_2<1$, and $\sup_{\theta>0}\sup_{(P,Q) \in \PP} \norm{\T^{-\theta}u}_{\Lp}$ $< \infty$. Then for any $\delta>0$, $B\geq \frac{\cd^2}{2\Tilde{w}^2\alpha^2}\log\frac{2}{\delta}$, $0<\Tilde{w}<w <1$, $0<\alpha\leq e^{-1}$, $\theta_l>0$, there exists $k$ such for all $N+M \geq k$, we have 
$$\inf_{\theta>\theta_l}\inf_{(P,Q) \in \PP} P_{H_1}\left\{\bigcup_{\lambda \in \Lambda}\stat \geq \hat{q}_{1-\frac{w\alpha}{\cd}}^{B,\lambda} \right\} \geq 1-5\delta,$$
provided one of the following cases hold:

(i) $\lambda_i \lesssim i^{-\beta}$, $1<\beta<\beta_u < \infty$, $\lambda_L = r_1\frac{\log(N+M)}{N+M}$, $\lambda_U=r_2\left(\frac{\log(N+M)}{N+M} \right)^\frac{2}{4\Tilde{\xi}+1}$, for some constants $r_1,r_2>0$, where $\Tilde{\xi}=\max(\xi,\frac{1}{4})$,  $\norm{\Sigma_{PQ}}_{\op}\geq \lambda_U$, and $$\Delta_{N,M} = c(\alpha,\delta,\theta)\max\left\{\left(\frac{\log\log(N+M)}{N+M}\right)^{\frac{4\Tilde{\theta}\beta}{4\Tilde{\theta}\beta+1}}, \left(\frac{\log(N+M)}{N+M}\right)^{2\tilde{\theta}} \right\},$$ with $c(\alpha,\delta,\theta)\gtrsim\max\left\{\delta^{-2}(\log 1/\alpha)^2,d_1^{2\tilde{\theta}}\right\}$, for some constant $d_1>0.$ Furthermore, if $\sup_{i}\norm{\phi_i}_{\infty} < \infty$, then the above conditions on $\lambda_L$ and $\lambda_U$ can be replaced by $\lambda_L = r_3\left(\frac{\log(N+M)}{N+M} \right)^{\beta_u}$, $\lambda_U=r_4\left(\frac{\log(N+M)}{N+M} \right)^\frac{2}{4\Tilde{\xi}+1}$, for some constants $r_3,r_4>0$ and $$\Delta_{N,M} = c(\alpha,\delta,\theta,\beta)\max\left\{\left(\frac{\log\log(N+M)}{N+M}\right)^{\frac{4\Tilde{\theta}\beta}{4\Tilde{\theta}\beta+1}}, \left(\frac{\log(N+M)}{N+M}\right)^{2\tilde{\theta}\beta} \right\},$$ where $c(\alpha,\delta,\theta,\beta)\gtrsim \max\left\{\delta^{-2}(\log 1/\alpha)^2,d_2^{2\tilde{\theta}\beta}\right\}$ for some constant $d_2>0.$\vspace{1mm}\\

(ii) $\lambda_i \lesssim e^{-\tau i}$,
$\tau>0$, $\lambda_L = r_5\frac{\log(N+M)}{N+M}$, $\lambda_U=r_6\left(\frac{\log(N+M)}{N+M} \right)^{1/2\xi}$ for some $r_5,r_6>0$, $\lambda_U\le \norm{\Sigma_{PQ}}_{\op}$, and $$\Delta_{N,M} = c(\alpha,\delta,\theta)\max\left\{ \frac{\sqrt{\log(N+M)}\log\log(N+M)}{N+M}, \left(\frac{\log(N+M)}{N+M}\right)^{2\tilde{\theta}} \right\},$$ where $c(\alpha,\delta,\theta) \gtrsim \max\left\{\sqrt{\frac{1}{2\tilde{\theta}}},1\right\}\max\left\{\delta^{-2}(\log 1/\alpha)^2,d_4^{2\tilde{\theta}}\right\},$ for some constant $d_4>0.$ Furthermore if  $\sup_{i}\norm{\phi_i}_{\infty} < \infty$, then the above conditions on $\lambda_L$ and $\lambda_U$ can be replaced by $\lambda_L = r_7\left(\frac{\log(N+M)}{N+M} \right)^{\frac{1}{2\theta_l}}$, $\lambda_U=r_8\left(\frac{\log(N+M)}{N+M} \right)^{\frac{1}{2\xi}}$, for some $r_7,r_8>0$ and $$\Delta_{N,M} = c(\alpha,\delta,\theta)\frac{\sqrt{\log(N+M)}\log\log(N+M)}{N+M},$$
where $c(\alpha,\delta,\theta) \gtrsim \max\left\{\sqrt{\frac{1}{2\tilde{\theta}}},\frac{1}{2\tilde{\theta}},1\right\}\delta^{-2}(\log 1/\alpha)^2.$
\end{theorem}
It follows from the above result that the set $\Lambda$ which is defined by $\lambda_L$ and $\lambda_U$ does not depend on any unknown parameters.

\begin{remark}
Theorem~\ref{thm: perm adp typeII} shows that the adaptive test achieves the same performance (up to $\log$ $\log$ factor) as that of the Oracle test but without the prior knowledge of unknown parameters, $\theta$ and $\beta$. 
In fact, following the ideas we used in the proof of Theorem \ref{thm:minimax} combined with the ideas in \citet{Ingster2000} and \citet[Theorem 6]{Krishna}, it can be shown that an extra factor of $\sqrt{\log \log(N+M)}$ is unavoidable in the expression of the adaptive minimax separation boundary compared to the non-adaptive case. Thus, our adaptive test is actually minimax optimal up to a factor $\sqrt{\log \log(N+M)}$. This gap occurs, since the approach we are using to bound the threshold $\hat{q}_{1-\frac{\alpha}{\cd}}^{B,\lambda}$ uses Bernstein-type inequality (see Lemma~\ref{lemma:bound quantile}), 
which involves a factor $\log(\cd/\alpha),$ with $\cd \lesssim \log(N+M)$, hence yielding the extra $\log\log$ factor. We expect that this gap can be filled by using a threshold that depends on the asymptotic distribution of $\hat{\eta}_{\lambda}$ (as was done in \citealp{Krishna}), yielding an asymptotic $\alpha$-level test in contrast to the exact $\alpha$-level test achieved by the permutation approach.
\end{remark} 

\subsection{Choice of kernel}\label{subsec:kernel-choice}
In the discussion so far, a kernel is first chosen which determines the test statistic, the test, and the set of local alternatives, $\mathcal{P}$. But the key question is what is the right kernel. In fact, this question is the holy grail of all kernel methods. 

To this end, we propose to start with a family of kernels, $\mathcal{K}$ and construct an adaptive test by taking the union of tests jointly over $\lambda\in \Lambda$ and $K\in\mathcal{K}$ to test $H_0:P=Q$ vs.~$H_1:\cup_{K\in\mathcal{K}}\cup_{\theta>0} \tilde{\mathcal{P}},$
where \begin{equation*}\tilde{\PP}:=\tilde{\PP}_{\theta,\Delta,K} := \left\{(P,Q): \frac{dP}{d\PQ}-1 \in \range (\T_K^{\theta}),\,\,\underline{\rho}^2(P,Q) \geq \Delta\right\},
\end{equation*}
with $\T_K$ being defined similar to $\T$ for $K\in \mathcal{K}.$ Some examples of $\mathcal{K}$ include the family of Gaussian kernels indexed by the bandwidth, $\{e^{-\Vert x-y\Vert^2_2/h},\,x,y\in\mathbb{R}^d:h\in (0,\infty)\}$; the family of Laplacian kernels indexed by the bandwidth, $\{e^{-\Vert x-y\Vert_1/h},\,x,y\in\mathbb{R}^d:h\in (0,\infty)\}$; family of radial-basis functions, $\{\int^\infty_0  e^{-\sigma\Vert x-y\Vert^2_2}\,d\mu(\sigma):\mu\in M^+((0,\infty))\}$, where $M^+((0,\infty))$ is the family of finite signed measures on $(0,\infty)$; a convex combination of base kernels, $\{\sum^\ell_{i=1}\lambda_i K_i:\sum^\ell_{i=1}\lambda_i=1,\,\lambda_i\ge 0,\forall\, i\in[\ell]\}$, where $(K_i)^\ell_{i=1}$ are base kernels. In fact, any of the kernels in the first three examples can be used as base kernels. The idea of adapting to a family of kernels has been explored in regression and classification settings under the name multiple-kernel learning and we refer the reader to \citet{gonen11a} and references therein for details.

Let $\hat{\eta}_{\lambda,K}$ be the test statistic based on kernel $K$ and regularization parameter $\lambda$. We reject $H_0$ if $\hat{\eta}_{\lambda,K} \geq \hat{q}_{1-\frac{w\alpha}{\cd|\mathcal{K}|}}^{B,\lambda,K}$ for any $(\lambda,K) \in \Lambda \times \mathcal{K}$. Similar to Theorem~\ref{thm:perm adp typeI}, it can be shown that this test has level $\alpha$ if $|\mathcal{K}|<\infty$. The requirement of $|\mathcal{K}|<\infty$ holds if we consider the above-mentioned families with a finite collection of bandwidths in the case of Gaussian and Laplacian kernels, and a finite collection of measures from $M^+((0,\infty))$ in the case of radial basis functions. 

Similar to Theorem~\ref{thm: perm adp typeII}, it can be shown  that the kernel adaptive test is minimax optimal w.r.t.~$\tilde{\mathcal{P}}$ up to a $\log\log$ factor, with the main difference being an additional factor of $\log|\mathcal{K}|$ as illustrated in the next Theorem.  We do not provide a proof since it is very similar to that of Theorem~\ref{thm: perm adp typeII} with $|\Lambda|$ replaced by $|\Lambda||\mathcal{K}|.$

\begin{theorem}[Separation boundary--adaptation over kernel]\label{thm: perm adp kernel typeII}
Suppose 
$(A_0)$--$(A_4)$ and $(B)$ hold. Let $\mathcal{A}:=\log|\mathcal{K}|$,  
$\tilde{\theta}=\min(\theta,\xi),$ $s=e_1N=e_2M$ for $0<e_1,e_2<1$, and $$\sup_{K\, \in \mathcal{K}}\sup_{\theta>0}\sup_{(P,Q) \in \tilde{\PP}} \norm{\T^{-\theta}u}_{\Lp} < \infty.$$ Then for any $\delta>0$, $0<\alpha\leq e^{-1}$,
$B\geq \frac{\cd^2 |\mathcal{K}|^2}{2\Tilde{w}^2\alpha^2}\log\frac{2}{\delta}$, $0<\Tilde{w}<w <1$, $0<\alpha\leq e^{-1}$, $\theta_l>0$, there exists $k$ such for all $N+M \geq k$, we have
$$ \inf_{K\in\mathcal{K}}\inf_{\theta>\theta_l}\inf_{(P,Q)\in\tilde{\mathcal{P}}}P_{H_1}\left\{\bigcup_{(\lambda,K) \in \Lambda\times \mathcal{K}}\hat{\eta}_{\lambda,K} \geq \hat{q}_{1-\frac{w\alpha}{\cd|\mathcal{K}|}}^{B,\lambda,K} \right\} \geq 1-5\delta,$$
provided one of the following cases hold: For any $K\in\mathcal{K}$ and $(P,Q)\in\tilde{\mathcal{P}}$,\vspace{1mm}\\
(i) 
$\lambda_i \lesssim i^{-\beta}$,
$1<\beta<\beta_u < \infty$, $\lambda_L = r_1\frac{\log(N+M)}{N+M}$, $\lambda_U=r_2\left(\frac{\log(N+M)}{N+M} \right)^\frac{2}{4\Tilde{\xi}+1}$, for some constants $r_1,r_2>0$, where $\Tilde{\xi}=\max(\xi,\frac{1}{4})$,  $\norm{\Sigma_{PQ}}_{\op}\geq \lambda_U$, and $$\Delta_{N,M} = c(\alpha,\delta,\theta)\max\left\{\left(\frac{\mathcal{A}\log\log(N+M)}{N+M}\right)^{\frac{4\Tilde{\theta}\beta}{4\Tilde{\theta}\beta+1}}, \left(\frac{\log(N+M)}{N+M}\right)^{2\tilde{\theta}} \right\},$$ with $c(\alpha,\delta,\theta)\gtrsim\max\left\{\delta^{-2}(\log 1/\alpha)^2,d_1^{2\tilde{\theta}}\right\}$, for some constant $d_1>0.$ Furthermore, if $\sup_{i}\norm{\phi_i}_{\infty} < \infty$, then the above conditions on $\lambda_L$ and $\lambda_U$ can be replaced by $\lambda_L = r_3\left(\frac{\log(N+M)}{N+M} \right)^{\beta_u}$, $\lambda_U=r_4\left(\frac{\log(N+M)}{N+M} \right)^\frac{2}{4\Tilde{\xi}+1}$, for some constants $r_3,r_4>0$ and $$\Delta_{N,M} = c(\alpha,\delta,\theta,\beta)\max\left\{\left(\frac{\mathcal{A}\log\log(N+M)}{N+M}\right)^{\frac{4\Tilde{\theta}\beta}{4\Tilde{\theta}\beta+1}}, \left(\frac{\log(N+M)}{N+M}\right)^{2\tilde{\theta}\beta} \right\},$$ where $c(\alpha,\delta,\theta,\beta)\gtrsim \max\left\{\delta^{-2}(\log 1/\alpha)^2,d_2^{2\tilde{\theta}\beta}\right\}$ for some constant $d_2>0.$\vspace{1mm}\\
(ii) 
$\lambda_i \lesssim e^{-\tau i}$,
$\tau>0$, $\lambda_L = r_5\frac{\log(N+M)}{N+M}$, $\lambda_U=r_6\left(\frac{\log(N+M)}{N+M} \right)^{1/2\xi}$ for some $r_5,r_6>0$, $\lambda_U\le \norm{\Sigma_{PQ}}_{\op}$, and $$\Delta_{N,M} = c(\alpha,\delta,\theta)\max\left\{ \frac{\mathcal{A}\sqrt{\log(N+M)}\log\log(N+M)}{N+M}, \left(\frac{\log(N+M)}{N+M}\right)^{2\tilde{\theta}} \right\},$$ where $c(\alpha,\delta,\theta) \gtrsim \max\left\{\sqrt{\frac{1}{2\tilde{\theta}}},1\right\}\max\left\{\delta^{-2}(\log 1/\alpha)^2,d_4^{2\tilde{\theta}}\right\},$ for some constant $d_4>0.$ Furthermore if  $\sup_{i}\norm{\phi_i}_{\infty} < \infty$, then the above conditions on $\lambda_L$ and $\lambda_U$ can be replaced by $\lambda_L = r_7\left(\frac{\log(N+M)}{N+M} \right)^{\frac{1}{2\theta_l}}$, $\lambda_U=r_8\left(\frac{\log(N+M)}{N+M} \right)^{\frac{1}{2\xi}}$, for some $r_7,r_8>0$ and $$\Delta_{N,M} = c(\alpha,\delta,\theta)\frac{\mathcal{A}\sqrt{\log(N+M)}\log\log(N+M)}{N+M},$$
where $c(\alpha,\delta,\theta) \gtrsim \max\left\{\sqrt{\frac{1}{2\tilde{\theta}}},\frac{1}{2\tilde{\theta}},1\right\}\delta^{-2}(\log 1/\alpha)^2.$
\end{theorem}

\section{Experiments}\label{sec:experiments}
In this section, we study the empirical performance of the proposed two-sample test by comparing it to the performance of the adaptive MMD test \citep{MMDagg}, Energy test \citep{Energy} and Kolmogorov-Smirnov (KS) test \citep{KS,Fasano}. The adaptive MMD test in \citet{MMDagg} involves using a translation invariant kernel on $\R^d$ in $D^2_{\text{MMD}}$ with bandwidth $h$ where the critical level is obtained by a permutation/wild bootstrap. Multiple such tests are constructed over $h$, which are aggregated to achieve adaptivity and the resultant test is referred to as MMDAgg. All tests are repeated 500 times and the average power is reported.

To compare the performance, in the following, we consider different experiments on Euclidean and directional data using the Gaussian kernel,  $K(x,y)=\exp\left(-\frac{\norm{x-y}_{2}^2}{2h}\right)$ and by setting $\alpha=0.05$, with $h$ being the bandwidth. For our test, as discussed in Section~\ref{subsec:kernel-choice}, we construct an adaptive test by taking the union of tests jointly over $\lambda \in \Lambda$ and $h \in W$. Let $\hat{\eta}_{\lambda,h}$ be the test statistic based on $\lambda$ and bandwidth $h$. We reject $H_0$ if $\hat{\eta}_{\lambda,h} \geq \hat{q}_{1-\frac{\alpha}{\cd|W|}}^{B,\lambda,h}$ for any $(\lambda,h) \in \Lambda \times W$. We performed such a test for $\Lambda :=\{\lambda_L, 2\lambda_L, ... \,, \lambda_U\},$ and $W:=\{w_L h_m, 2w_L h_m, ... \,, w_Uh_m\}$, where $h_m:= \text{median}\{\norm{q-q'}_2^2: q,q' \in X \cup Y\}$. In our experiments we set $\lambda_L=10^{-6}$, $\lambda_U=5,$ $w_L=0.01$, $w_U=100$, $B = 250$ for all the experiments. We chose the number of permutations $B$ to be large enough to ensure the control of Type-I error (see Figure~\ref{fig:type-I}). We show the results for both Tikhonov and Showalter regularization and for different choices of the parameter $s$, which is the number of samples used to estimate the covariance operator after sample splitting. All codes used for our spectral regularized test are available at \emph{https://github.com/OmarHagrass/Spectral-regularized-two-sample-test}.

\begin{remark} (i) For our experimental evaluations of the other tests, we used the following: \\For the Energy test, we used the "eqdist.etest" function from the R Package "energy" (for code see  https://github.com/mariarizzo/energy) with the parameter $R=199$ indicating the number of bootstrap replicates. For the KS test, we used the R package "fasano.franceschini.test" (for code see https://github.com/braunlab-nu/fasano.franceschini.test). For the MMDAgg test, we employed the code provided in \citet{MMDagg}. Since \citet{MMDagg} presents various versions of the MMDAgg test, we compared our results to the version of MMDAgg that yielded the best performance on the datasets considered in \cite{MMDagg}, which include the MNIST data and the perturbed uniform. For the rest of the experiments, we compared our test to "MMDAgg uniform" with $\Lambda[-6,1]$---see \cite{MMDagg} for details.

(ii) As mentioned above, in all the experiments, $B$ is chosen to be $250$. It has to be noted that this choice of $B$ is much smaller than that suggested by Theorems~\ref{thm:perm adp typeI} and \ref{thm: perm adp typeII} to maintain the Type-I error and so one could expect the resulting test to be anti-conservative. However, in this section's experimental settings, we found this choice of $B$ to yield a test that is neither anti-conservative nor overly conservative. Of course, we would like to acknowledge that too small $B$ would make the test anti-conservative while too large $B$ would make it conservative, i.e., loss of power (see Figure~\ref{fig:type-I} where increasing $B$ leads to a drop in the Type-I error below $0.05$ and therefore a potential drop in the power). Hence, the choice of $B$ is critical in any permutation-based test. This phenomenon is attributed to the conservative nature of the union bound utilized in computing the threshold of the adapted test. Thus, an intriguing future direction would be to explore methods superior to the union bound to accurately control the Type-I error at the desired level and further enhance the power.

(iii) In an ideal scenario, the choice of $B$ should be contingent upon $\alpha$, as evidenced in the statements of Theorems~\ref{thm:perm adp typeI} and \ref{thm: perm adp typeII}. However, utilizing this theoretical bound for the number of permutations would be computationally prohibitive, given the expensive nature of computing each permuted statistic. Exploring various approximation schemes such as random Fourier features \citep{Rahimi-08a}, Nystr\"{o}m subsampling (e.g., \citealp{Williams-01,Drineas-05}), or sketching~\citep{Yang-17}, which are known for expediting kernel methods, could offer more computationally efficient testing approaches, and therefore could allow to choose $B$ as suggested in Theorems~\ref{thm:perm adp typeI} and \ref{thm: perm adp typeII}.
\end{remark}

\begin{figure}[t]
\centering
\includegraphics[scale=0.7]{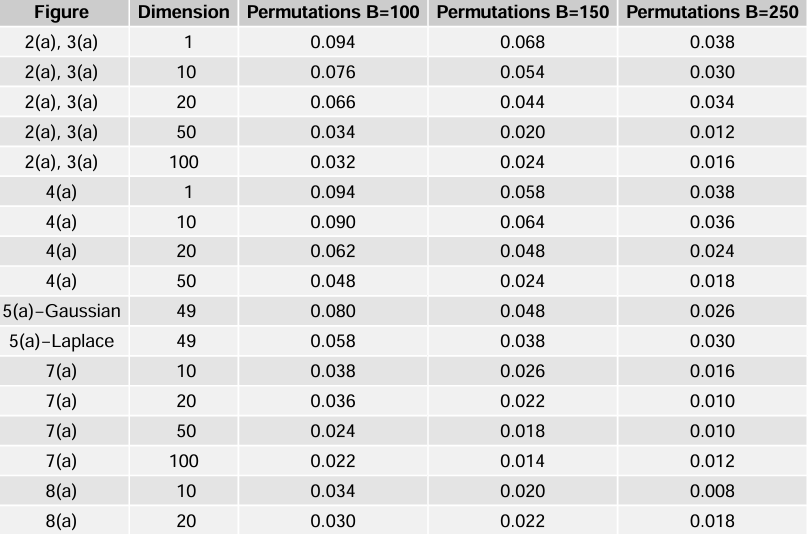}
\vspace{-2mm}
\caption{Type-I error for different number of permutations.}
\label{fig:type-I}
\end{figure}

\subsection{Bechmark datasets}
In this section, we investigate the empirical performance of the spectral regularized test and compare it to the other methods.
\subsubsection{Gaussian distribution}\label{subsec:Gaussian_experiments}
Let $P=N(0, I_d)$ and $Q=N(\mu,\sigma^2 I_d)$, where $\mu\ne 0, \sigma^2=1$ or $\mu=0$, $\sigma^2\ne 1$, i.e., we test for Gaussian mean shift or Gaussian covariance scaling, where $I_d$ is the $d$-dimensional identity matrix. Figures~\ref{fig:Gaussian}(a) and \ref{fig:Gauss-cov}(a)
show the results for the Gaussian mean shift and covariance scale experiments, where we used $s=20$ for our test. 
It can be seen from Figure~\ref{fig:Gaussian}(a) that the best power is obtained by the Energy test, followed closely by the proposed test, with the other tests showing poor performance, particularly in high-dimensional settings. On the other hand, the proposed test performs the best in Figure~\ref{fig:Gauss-cov}(a), closely followed by the Energy test. We also investigated the effect of $s$ on the test power for the Showalter method (Tikhonov method also enjoys very similar results), whose results are reported in Figures~\ref{fig:Gaussian}(b) and \ref{fig:Gauss-cov}(b). We note from these figures that lower values of $s$ perform slightly better, though overall, the performance is less sensitive to the choices of $s$. Based on these results and those presented below, as a practical suggestion, the choice $s=(N+M)/20$ is probably fine. 

\begin{figure}[t]
\begin{minipage}[b]{0.03\linewidth}
\subcaption{\hfill}
\end{minipage}
\begin{minipage}[b]{0.96\linewidth}
\centering
\includegraphics[width=\textwidth]{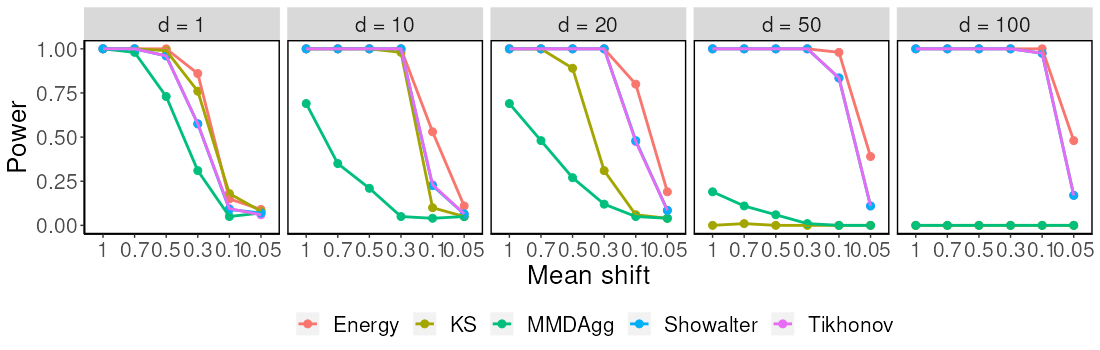}
\label{fig:Gaussain_shift}
\end{minipage}
\begin{minipage}[b]{0.03\linewidth}
\subcaption{\hfill}
\end{minipage}
\begin{minipage}[b]{0.96\linewidth}
\centering
\includegraphics[width=\textwidth]{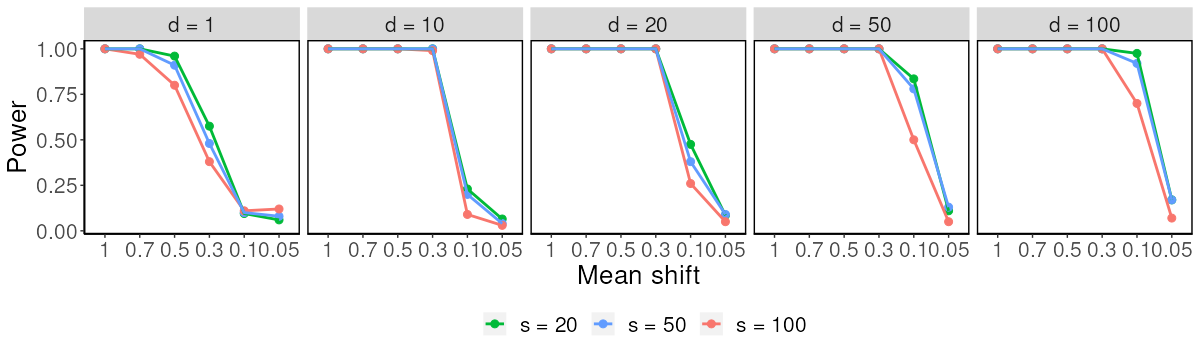}
\label{fig:Gaussain_shift_diff_s}
\end{minipage}
\vspace{-4mm}
\caption{Power for Gaussian shift experiments with different $d$ and $s$ using $N=M=200,$ where the Showalter method is used in (b).} \label{fig:Gaussian}
\vspace{-4mm}
\end{figure}

\begin{figure}[h]
\begin{minipage}[b]{0.03\linewidth}
\subcaption{\hfill}
\end{minipage}
\begin{minipage}[b]{0.96\linewidth}
\centering
\includegraphics[width=\textwidth]{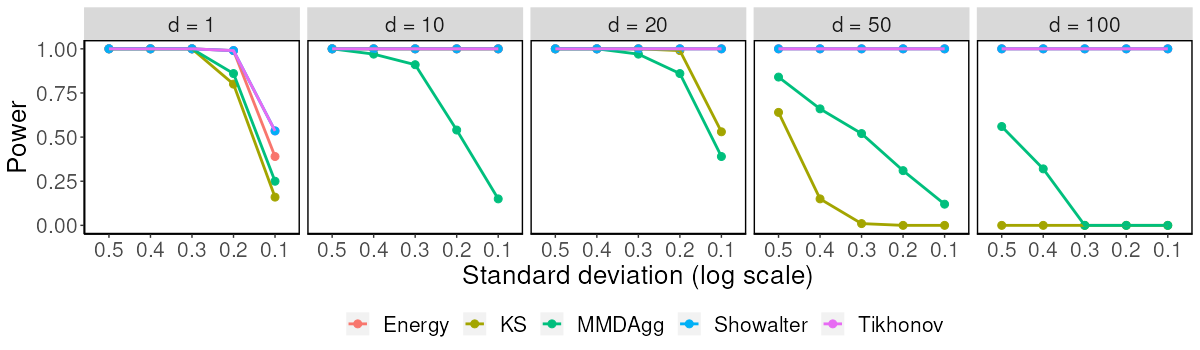}
\label{fig:Gaussain_scale}
\end{minipage}
\begin{minipage}[b]{0.03\linewidth}
\subcaption{\hfill}
\end{minipage}
\begin{minipage}[b]{0.96\linewidth}
\centering
\includegraphics[width=\textwidth]{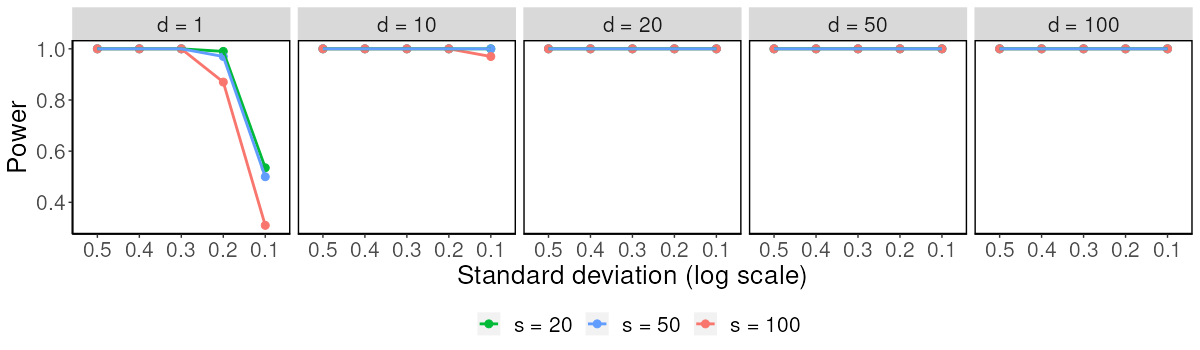}
\label{fig:Gaussain_scale_diff_s}
\end{minipage}
\caption{Power for Gaussian covariance scale experiments with different $d$ and $s$ using $N=M=200$, where the Showalter method is used in (b).}
\label{fig:Gauss-cov}
\vspace{-4mm}
\end{figure}

\subsubsection{Cauchy distribution}\label{subsec:cauchy}
In this section, we investigate the performance of the proposed test for heavy-tailed distribution, specifically the Cauchy distribution with median shift alternatives. Particularly, we test samples from a Cauchy distribution with zero median and unit scale against another set of Cauchy samples with different median shifts and unit scale. 
Figure~\ref{fig:Cau}(a)  
shows that for 
$d\in\{1,10\}$, the KS test achieves the highest power for the majority of considered median shifts, followed closely by our regularized test which achieves better power for the harder problem when the shift is small. Moreover, for $d\in\{20,50\}$, our proposed test achieves the highest power. The effect of $s$ is captured in Figure~\ref{fig:Cau}(b).

\begin{figure}[t]
\begin{minipage}[b]{0.03\linewidth}
\subcaption{\hfill}
\end{minipage}
\begin{minipage}[b]{0.96\linewidth}
\centering
\includegraphics[width=\textwidth]{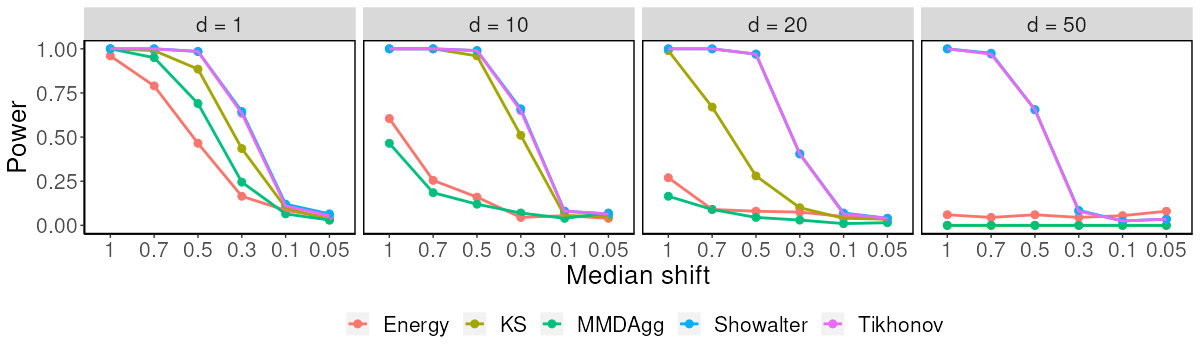}
\label{fig:cauchy}
\end{minipage}
\begin{minipage}[b]{0.03\linewidth}
\subcaption{\hfill}
\end{minipage}
\begin{minipage}[b]{0.96\linewidth}
\centering
\includegraphics[width=\textwidth]{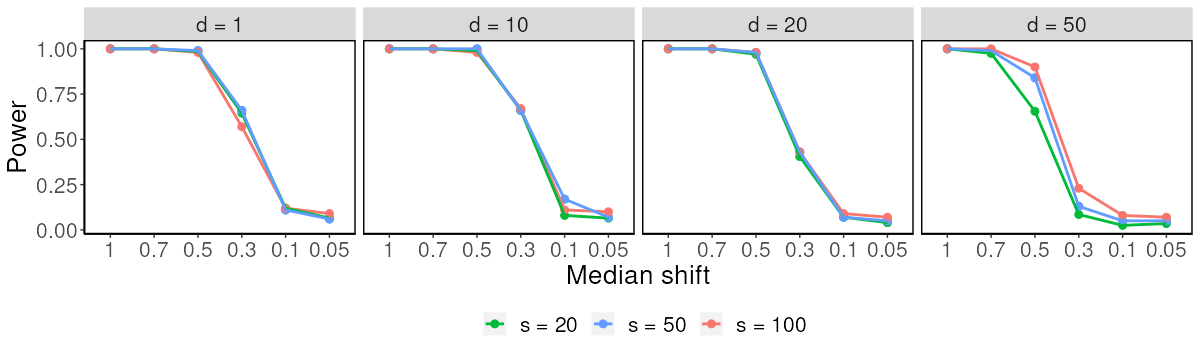}
\label{fig:cauchy_diff_s}
\end{minipage}
\vspace{-8mm}
\caption{Power for Cauchy with median shifts for different $s$ and $d$ with $N=M=500$, where the Showalter method is used in (b).} 
\label{fig:Cau}
\end{figure}

\subsubsection{MNIST dataset}\label{subsec:MNIST}
Next, we investigate the performance of the regularized test on the MNIST dataset \citep{Mnist}, which is a collection of images of digits from $0$--$9$. In our experiments, as in \citet{MMDagg}, the images were downsampled to $7 \times 7$ (i.e. $d=49$) and consider 500 samples drawn with replacement from set $P$ while testing against the set $Q_i$ for $i=1,2,3,4,5$, where $P$ consists of images of the digits\\\vspace{-3mm} 
$$P : 0,1,2,3,4,5,6,7,8,9, $$ and 
$$Q_1 : 1,3,5,7,9,\qquad
Q_2 :0, 1,3,5,7,9, \qquad
Q_3 :0, 1,2,3,5,7,9, $$
$$Q_4 :0, 1,2,3,4,5,7,9, \qquad
Q_5 :0,1,2,3,4,5,6,7,9.$$
Figure~\ref{fig:mnst}(a) 
shows the power of our test for both Gaussian and Laplace kernels in comparison to that of MMDAgg and the other tests which shows the superior performance of the regularized test, particularly in the difficult cases, i.e., distinguishing between $P$ and $Q_i$ for larger $i$. Figure~\ref{fig:mnst}(b) 
shows the effect of $s$ on the test power, from which we can see that the best performance in this case is achieved for $s=50$, however, the overall results are not very sensitive to the choice of $s$.

\begin{figure}[t]
\begin{minipage}[b]{0.495\linewidth}
\centering
\includegraphics[width=\textwidth]{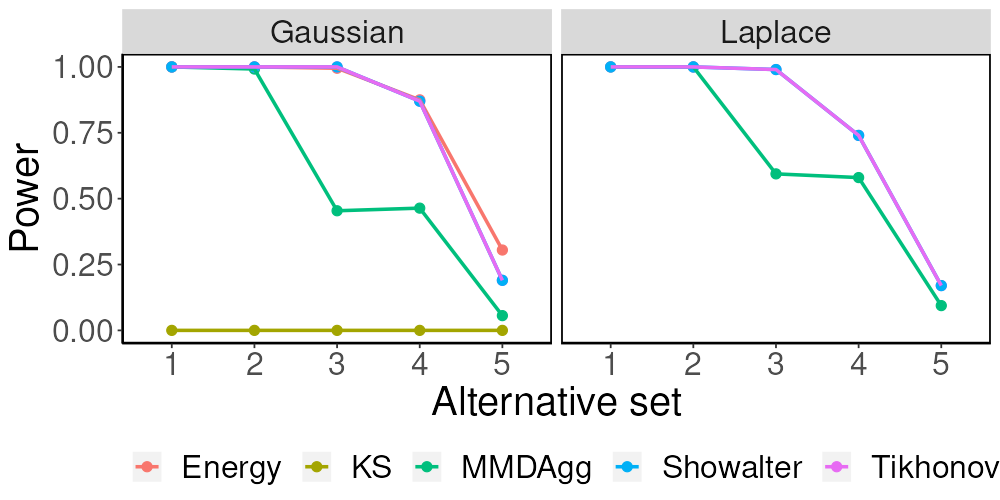}
\subcaption{}
\label{fig:MNIST}
\end{minipage}
\begin{minipage}[b]{0.495\linewidth}
\centering
\includegraphics[width=\textwidth]{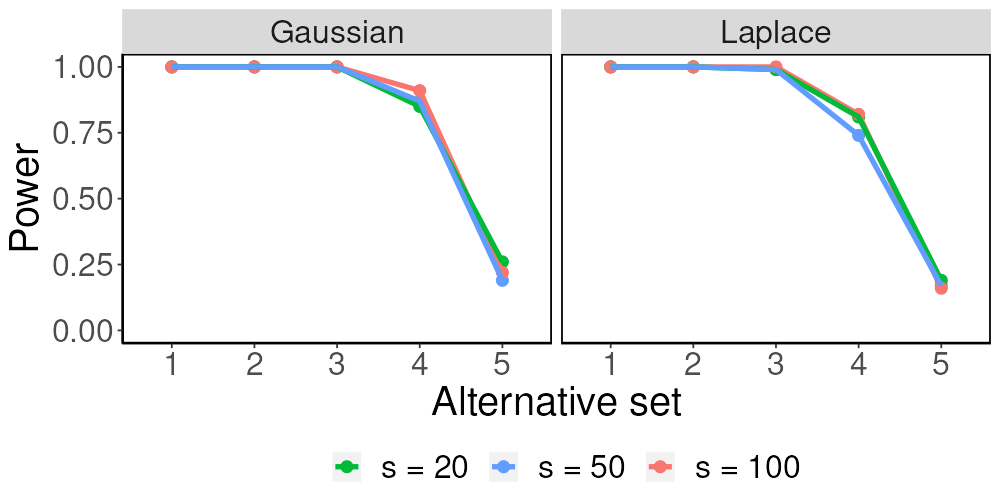}
\subcaption{}
\label{fig:MNIST_diff_s}
\end{minipage}
\vspace{1mm}
\caption{Power for MNIST dataset using $N=M=500$, where the Showalter method is used in (b).}\vspace{-6mm}
\label{fig:mnst}
\end{figure}

\subsubsection{Directional data}
In this section, we consider two experiments with directional domains. First, we consider a multivariate von Mises-Fisher distribution (which is the Gaussian analogue on 
unit-sphere) given by $f(x,\mu,k)=\frac{k^{d/2-1}}{2\pi^{d/2}I_{d/2-1}(k)}\exp(k\mu^Tx),\,x\in S^{d-1},$ where $k\geq 0$ is the concentration parameter, $\mu$ is the mean parameter and $I$ is the modified Bessel function (see Figure~\ref{fig:von}(a)). 
Figure~\ref{fig:mises}(a) 
shows the results for testing von Mises-Fisher distribution against spherical uniform distribution ($k=0$) for different concentration parameters using a Gaussian kernel.  Note that the theoretical properties of MMDAgg do not hold in this case, unlike the proposed test. We can see from Figure~\ref{fig:mises}(a) that the best power is achieved by the Energy test followed closely by our regularized test. Figure~\ref{fig:mises}(b)
shows effect of $s$ on the test power of the regularized test.

Second, we consider a mixture of two multivariate Watson distribution (which is an axially symmetric distribution on a sphere) given by $f(x,\mu,k)=\frac{\Gamma(d/2}{2\pi^{d/2}M(1/2,d/2.k)}\exp(k(\mu^Tx)^2)$, $x\in S^{d-1},$ where $k\geq 0$ is the concentration parameter, $\mu$ is the mean parameter and $M$ is Kummer's confluent hypergeometric function. Using equal weights we drew 500 samples from a mixture of two Watson distributions with similar concentration parameter $k$ and mean parameter $\mu_1 ,\mu_2$ respectively, where $\mu_1=(1, \dots ,1) \in R^d$ and $\mu_2=(-1,1 \dots ,1) \in R^d$ with the first coordinate changed to $-1$ (see Figure~\ref{fig:von}(b)). 
Figure~\ref{Fig:Watson}(a)
shows the results for testing  against spherical uniform distribution for different concentration parameters using a Gaussian kernel and we can see that our regularized test outperforms the other methods. Figure~\ref{Fig:Watson}(b) 
shows the effect of different choices of the parameter $s$ on the test power, which like in previous scenarios, is not very sensitive to the choice of $s$. Moreover, in Figure~\ref{Fig:Watson}(c) 
we illustrate how the power changes with increasing the sample size for the case $d=20$ and $k=6$, which shows that the regularized test power converges to one more quickly than the other methods.

\begin{figure}[t]
\begin{minipage}[b]{0.48\linewidth}
\centering
\includegraphics[scale=0.4]{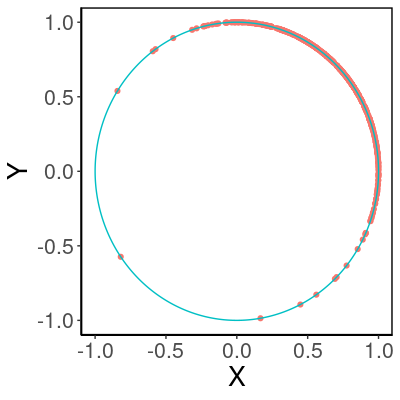}
\subcaption{}
\label{fig:von_mises_shape}
\end{minipage}
\hspace{0.2cm}
\begin{minipage}[b]{0.48\linewidth}
\centering
\includegraphics[scale=0.4]{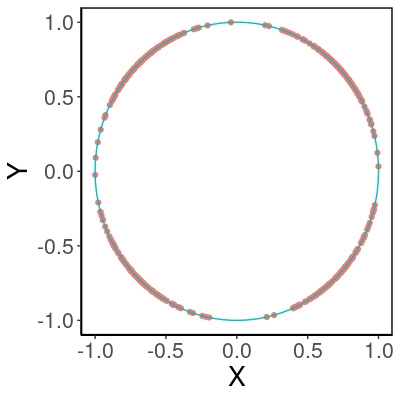}
\subcaption{}
\label{fig:watson_shape}
\end{minipage}
\vspace{-2mm}
\caption{Illustration for 500 samples drawn from directional distributions when $d=2$. (a) von Mises-Fisher distribution with $k=4$, $\mu=(1,1)$, (b) Mixture of two Watson distributions with $k=10$ and $\mu_1=(1,1)$, $\mu_2=(-1,1)$.}
\vspace{-2mm}
\label{fig:von}
\end{figure}
\begin{figure}[t]
\begin{minipage}[b]{0.03\linewidth}
\subcaption{\hfill}
\end{minipage}
\begin{minipage}[b]{0.96\linewidth}
\centering
\includegraphics[width=\textwidth]{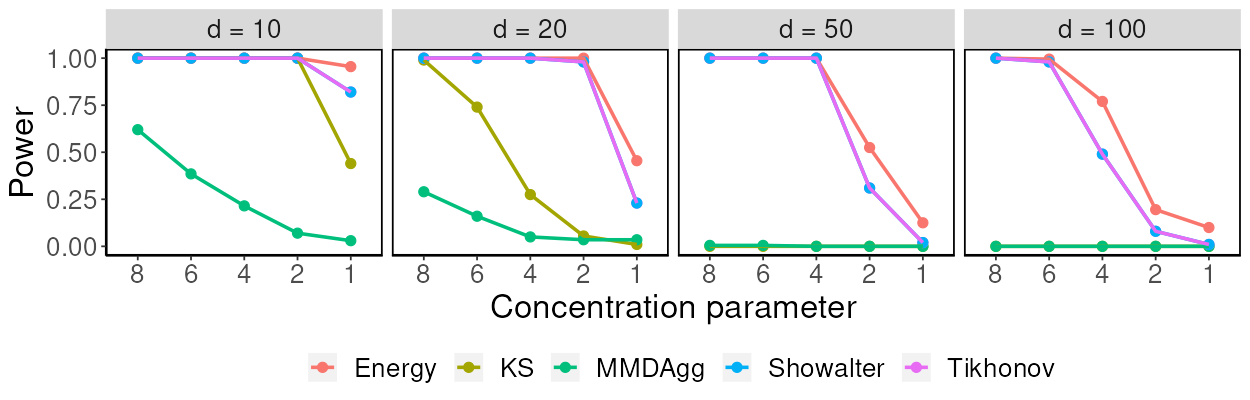}
\label{fig:von-mises}
\end{minipage}
\begin{minipage}[b]{0.03\linewidth}
\subcaption{\hfill}
\end{minipage}
\begin{minipage}[b]{0.96\linewidth}
\centering
\includegraphics[width=\textwidth]{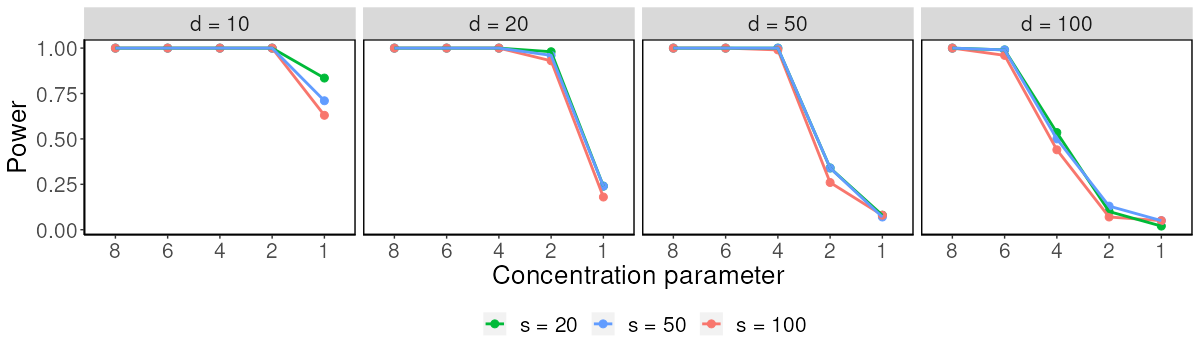}
\label{fig:von-mises_diff_s}
\end{minipage}
\vspace{-4mm}
\caption{Power for von Mises-Fisher distribution with different concentration parameter $k$ and $s$ using $N=M=500$, where the Showalter method is used in (b).}
\vspace{-4mm}
\label{fig:mises}
\end{figure}
\begin{figure}[t]
\begin{minipage}[b]{0.495\linewidth}
\centering
\includegraphics[scale=0.35]{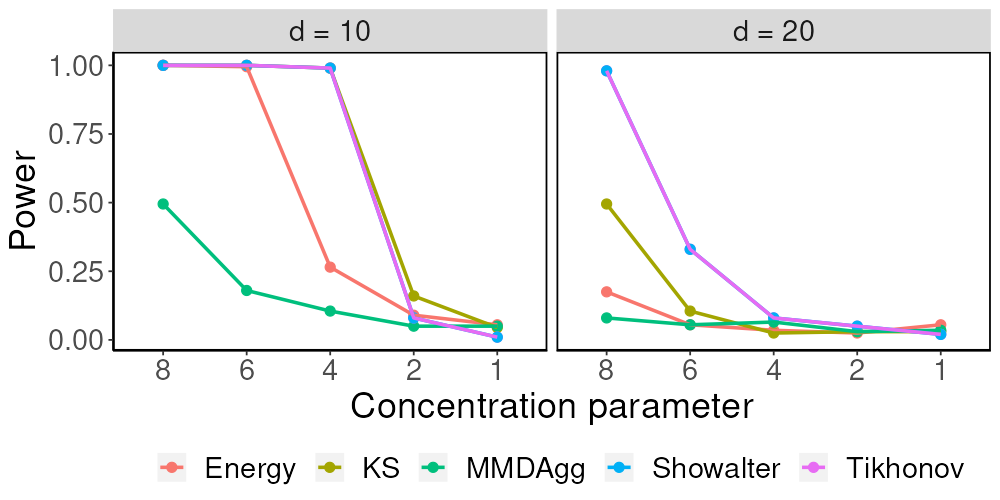}
\subcaption{}
\label{fig:watson}
\end{minipage}
\begin{minipage}[b]{0.495\linewidth}
\centering
\includegraphics[width=\textwidth]{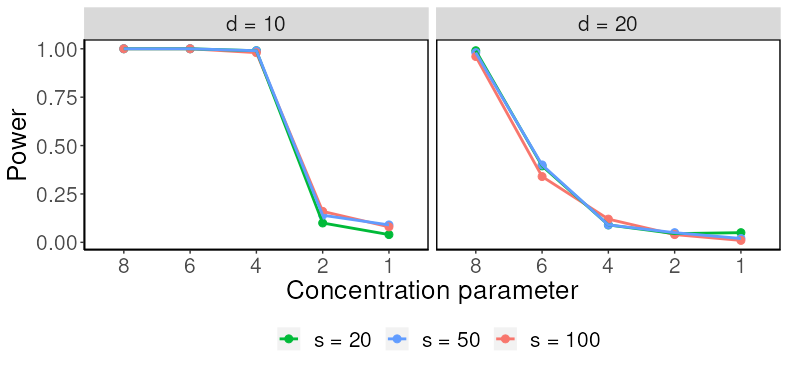}
\subcaption{}
\label{fig:watson_diff_s}
\end{minipage}
\begin{minipage}[b]{\linewidth}
\centering
\includegraphics[scale=0.3]{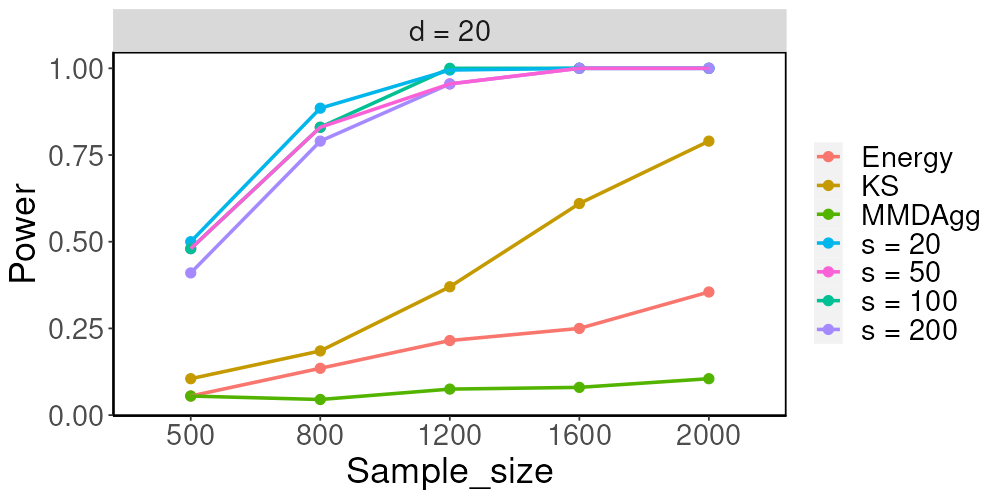}
\subcaption{}
\label{fig:watson_vary_size}
\end{minipage}
\vspace{-1mm}
\caption{Power for mixture of Watson distributions: (a) for different concentration parameter $k$  using $N=M=500$, (b) using different choices of $s$ for Showalter method, and (c) using different samples sizes with $d=20$, $k=6$, and with Showalter regularization.}
\label{Fig:Watson}
\end{figure}

\subsection{Perturbed uniform distribution}\label{subsec:perturbed}
In this section, we consider a simulated data experiment where we are testing a $d$-dimensional uniform distribution against a perturbed version of it. The perturbed density for $x \in \R^d$ is given by 
$$f_w(x) := \II_{[0,1]^d}(x)+c_dP^{-1}\sum_{v \in \{1,\dots,P\}^d} w_v \prod_{i=1}^dG(Px_i-v_i),$$
where $w=(w_v)_{v \in \{1,\dots,P\}^d} \in \{-1,1\}^{P^d}$, $P$ represents the number of perturbations being added and for $x \in \R$, 
$$G(x) := \exp\left(-\frac{1}{1-(4x+3)^2}\right)\II_{(-1,-1/2)}(x)-\exp\left(-\frac{1}{1-(4x+1)^2}\right)\II_{(-1/2,0)}(x).$$

As done in \cite{MMDagg}, we construct two-sample tests for $d=1$ and $d=2$, wherein we set $c_1=2.7$, $c_2 =7.3$. The tests are constructed 500 times with a new value of $w \in \{-1,1\}^{P^d}$ being sampled uniformly each time, and the average power is computed for both our regularized test and MMDAgg. Figure~\ref{fig:pert_uni_results1}(a) shows the power results of our test when $d=1$ for $P=1,2,3,4,5,6$ for both Gaussian and Laplace kernels and also when $d=2$ for $P=1,2,3$ for both Gaussian and Laplace kernels in comparison to that of other methods, where the Laplace kernel is defined as $K(x,y)=\exp\left(\frac{-\norm{x-y}_{1}}{2h}\right)$ with the bandwidth $h$ being chosen as $\text{median}\{\norm{q-q'}_1: q,q' \in X \cup Y\}$. It can be seen in Figure~\ref{fig:pert_uni_results1}(a) that 
our proposed test performs similarly for both Tikhonov and Showalter regularizations, while significantly improving upon other methods, particularly in the difficult case of large perturbations (note that large perturbations make distinguishing the uniform and its perturbed version difficult). Moreover, we can also see from Figure~\ref{fig:pert_uni_results1}(b) that  
the performance of the regularized test is not very sensitive to the choice of $s$. Note that, when $d=2$, it becomes really hard to differentiate the two samples for perturbations $P\geq 3$ when using a sample size smaller than $N=M=2000$, thus we presented the result for this choice of sample size in order to compare with other methods. We also investigated the effect of changing the sample size when $d=2$ for perturbations $P=2,3$ with different choices of $s$ as shown in Figure~\ref{fig:uniform}, which again shows the non-sensitivity of the power to the choice of $s$ while the power improving with increasing sample size. 

\begin{figure}[t]
\begin{minipage}[b]{0.03\linewidth}
\subcaption{\hfill}
\end{minipage}
\begin{minipage}[b]{0.96\linewidth}
\centering
\includegraphics[width=\textwidth]{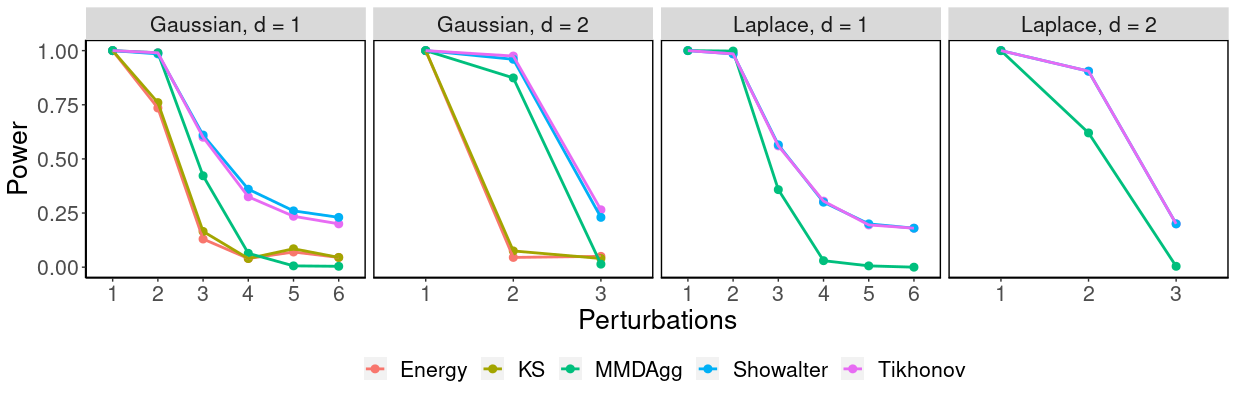}
\end{minipage}
\begin{minipage}[b]{0.03\linewidth}
\subcaption{\hfill}
\end{minipage}
\begin{minipage}[b]{0.96\linewidth}
\centering
\includegraphics[width=\textwidth]{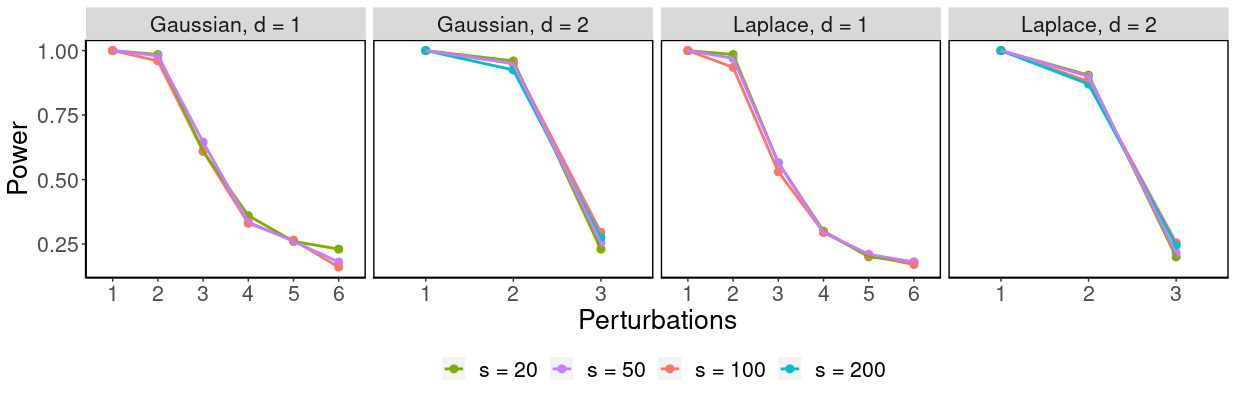}
\end{minipage}
\vspace{-2mm}
\caption{Power for perturbed uniform distributions, for $d=1$, $N=M=500$ and for $d=2$, $N=M=2000$.}
\label{fig:pert_uni_results1}
\end{figure}

\begin{figure}[t]
\centering
\includegraphics[scale=0.5]{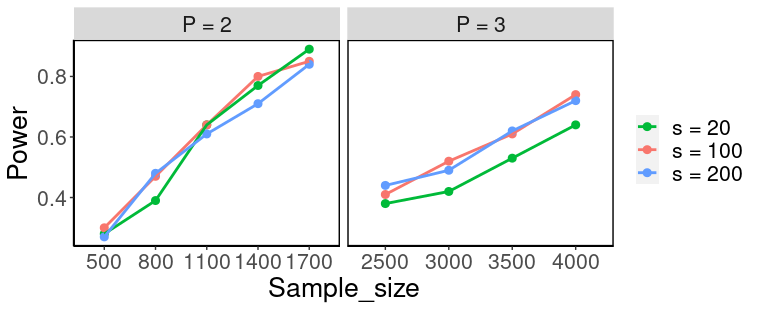}
\label{fig:pert_uni_results_vary_size}
\vspace{-2mm}
\caption{Perturbed uniform experiment with varying sample size for $d=2$ with Showalter regularization and $s=100$.}
\label{fig:uniform}
\end{figure}


\begin{figure}[t]
\begin{minipage}[b]{0.48\linewidth}
\centering
\includegraphics[width=\textwidth]{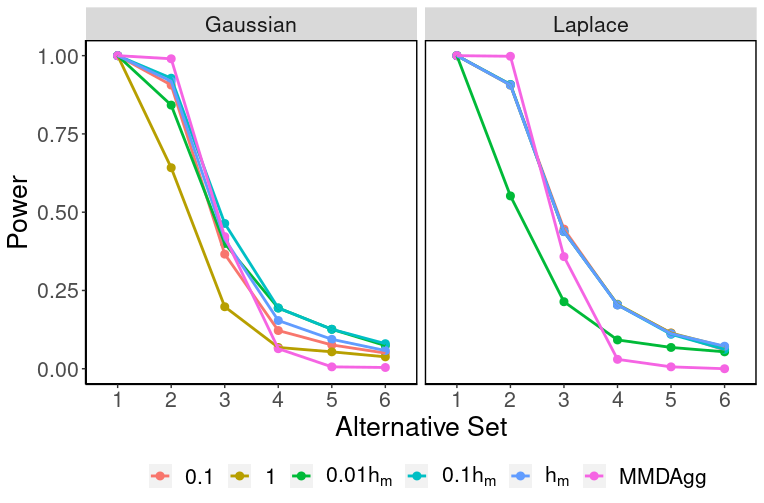}
\subcaption{}
\label{fig:pert_uni1d_results}
\end{minipage}
\hspace{0.2cm}
\begin{minipage}[b]{0.48\linewidth}
\centering
\includegraphics[width=\textwidth]{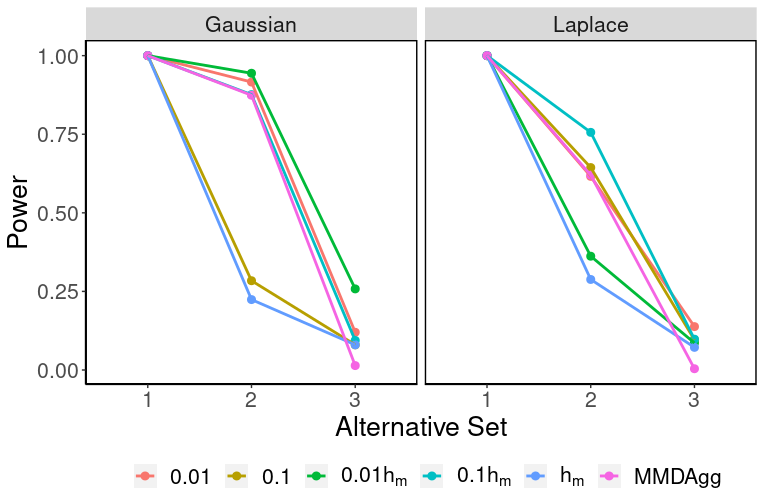}
\subcaption{}
\label{fig:pert_uni2d_results}
\end{minipage}
\begin{center}
\begin{minipage}[c]{0.48\linewidth}
\centering
\includegraphics[width=\textwidth]{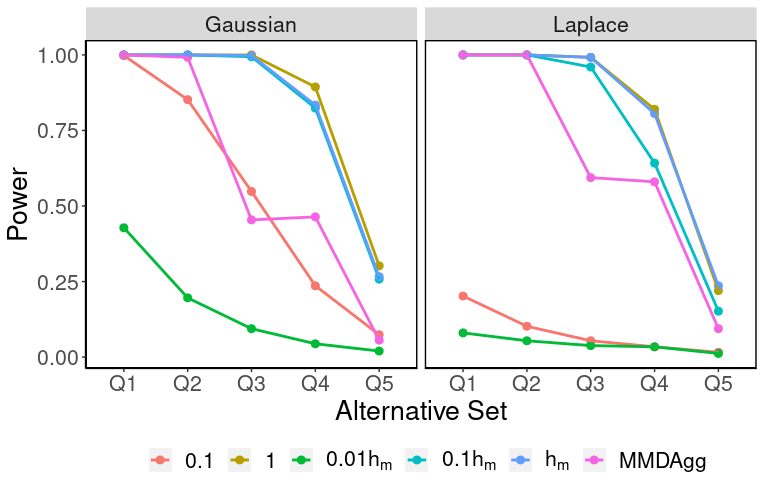}
\subcaption{}
\label{fig:mnist_diff_band_results}
\end{minipage}
\end{center}
\vspace{-2mm}
\caption{Power with different kernel bandwidths. (a) Perturbed uniform distribution with $d=1$, (b) Perturbed uniform distribution with $d=2$, (c) MNIST data using $N=M=500$.}
\label{fig:Pert_diff_band_results}
\vspace{-2mm}
\end{figure}

\begin{figure}[h]

\begin{minipage}[b]{0.48\linewidth}
\centering
\includegraphics[width=\textwidth]{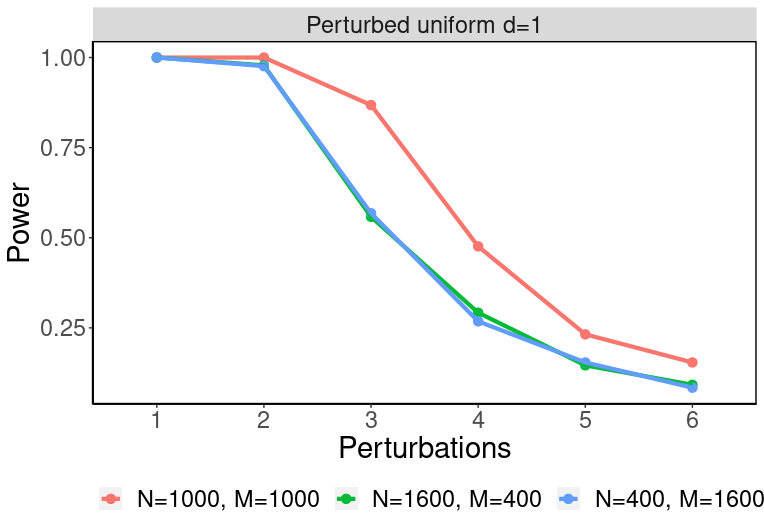}
\end{minipage}
\hspace{0.2cm}
\begin{minipage}[b]{0.48\linewidth}
\centering
\includegraphics[width=\textwidth]{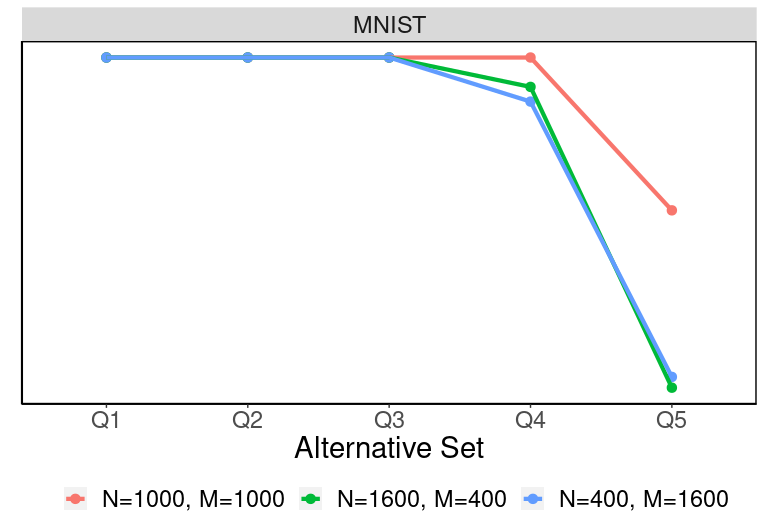}
\end{minipage}

\caption{Power for 1-dimensional perturbed uniform distributions and MNIST data using $N+M=2000$}
\label{fig:unbalanced_results}
\vspace{-1mm}
\end{figure}
\subsection{Effect of kernel bandwidth}
 In this section, we investigate the effect of kernel bandwidth on the performance of the regularized test when no adaptation is done. Figure~\ref{fig:Pert_diff_band_results} shows the performance of the test under different choices of bandwidth, wherein we used both fixed bandwidth choices and bandwidths that are multiples of the median $h_m$. The results in Figures~\ref{fig:Pert_diff_band_results}(a,b) are obtained using the perturbed uniform distribution data with $d=1$ and $d=2$, respectively while Figure~\ref{fig:Pert_diff_band_results}(c) is obtained using the MNIST data with $N=M=500$---basically using the same settings of Sections~\ref{subsec:perturbed} and \ref{subsec:MNIST}. We observe from Figures~\ref{fig:Pert_diff_band_results}(a,b)  that the performance is better at smaller bandwidths for the Gaussian kernel and deteriorates as the bandwidth gets too large, while a too-small or too-large bandwidth affects the performance in the case of a Laplacian kernel.

In Figure \ref{fig:Pert_diff_band_results}(c), we can observe that the performance gets better for large bandwidth and deteriorates when the bandwidth gets too small. Moreover, one can see from the results that for most choices of the bandwidth, the test based on $\stat$ still yields a non-trivial power as the number of perturbations (or the index of $Q_i$ in the case of the MNIST data) increases and eventually outperforms the MMDAgg test.  

\subsection{Unbalanced size for $N$ and $M$}
 We investigated the performance of the regularized test when $N\neq M$ and report the results in Figure~\ref{fig:unbalanced_results} for the 1-dimensional perturbed uniform and MNIST data set using Gaussian kernel and for fixed $N+M=2000$. It can be observed that the best performance is for $N=M=1000$ as we get more representative samples from both of the distributions $P$ and $Q$, which is also expected theoretically, as otherwise the rates are controlled by $\text{min}(M,N)$.

\section{Discussion}
To summarize, we have proposed a two-sample test based on spectral regularization that not only uses the mean element information like the MMD test but also uses the information about the covariance operator and showed it to be minimax optimal w.r.t.~the class of local alternatives $\mathcal{P}$ defined in \eqref{Eq:alternative-theta}. This test improves upon the MMD test in terms of the separation rate as the MMD test is shown to be not optimal w.r.t.~$\mathcal{P}$. We also presented a permutation version of the proposed regularized test along with adaptation over the regularization parameter, $\lambda$, and the kernel $K$ so that the resultant test is completely data-driven. Through numerical experiments, we also established the superiority of the proposed method over the MMD variant.

However, still there are some open questions that may be of interest to address. (i) The proposed test is computationally intensive and scales as $O((N+M)s^2)$ where $s$ is the number of samples used to estimate the covariance operator after sample splitting. Various approximation schemes like random Fourier features \citep{Rahimi-08a}, Nystr\"{o}m method (e.g., \citealp{Williams-01,Drineas-05}) or sketching~\citep{Yang-17} are known to speed up the kernel methods, i.e., computationally efficient tests can be constructed using any of these approximation methods. The question of interest, therefore, is about the computational vs. statistical trade-off of these approximate tests, i.e., are these computationally efficient tests still minimax optimal w.r.t.~$\mathcal{P}$? 
 (ii) The construction of the proposed test statistic requires sample splitting, which helps in a convenient analysis. It is of interest to develop an analysis for the kernel version of Hotelling's $T^2$-statistic (see \citealp{Harchaoui}) which does not require sample splitting. We conjecture that the theoretical results of Hotelling's $T^2$-statistic will be similar to those of this paper, however, it may enjoy a better empirical performance but at the cost of higher computational complexity. (iii) The adaptive test presented in Section~\ref{subsec:kernel-choice} only holds for the class of kernels, $\mathcal{K}$ for which $|\mathcal{K}|<\infty$. It will be interesting to extend the analysis for $\mathcal{K}$ for which $|\mathcal{K}|=\infty$, e.g., the class of Gaussian kernels with bandwidth in $(0,\infty)$, or the class of convex combination of certain base kernels as explained in Section~\ref{subsec:kernel-choice}.

\section{Proofs}
In this section, we present the proofs of all the main results of the paper.

\subsection{Proof of Theorem \ref{thm: MMD}}
Define $a(x)= \kk(\cdot,x)-\mu_P$, and $b(x)=\kk(\cdot,x)-\mu_Q$. Then we have 
\begin{align*}
    \hat{D}_{\mathrm{MMD}}^2& \stackrel{(*)}{=} \frac{1}{N(N-1)} \sum_{i\neq j} \left \langle a(X_i),a(X_j)\right \rangle_{\h}  + \frac{1}{M(M-1)}\sum_{i\neq j} \left \langle a(Y_i),a(Y_j)\right \rangle_{\h} \\ 
   &\qquad\qquad- \frac{2}{NM}\sum_{i,j} \left \langle a(X_i),a(Y_j)\right \rangle_{\h}\\
   & \stackrel{(\dag)}{=} \frac{1}{N(N-1)} \sum_{i\neq j} \left \langle a(X_i),a(X_j)\right \rangle_{\h}+ \frac{1}{M(M-1)}\sum_{i \neq j} \inner{b(Y_i)}{b(Y_j)}_{\h}  \\
   & \qquad \qquad +\frac{2}{M}\sum_{i=1}^{M} \inner{b(Y_i)}{(\mu_Q-\mu_P)}_{\h}+ D_{\mathrm{MMD}}^2 - \frac{2}{NM}\sum_{i,j} \inner{ a(X_i)}{b(Y_j)}_{\h} \\ 
   & \qquad \qquad- \frac{2}{N}\sum_{i=1}^{N} \inner{a(X_i)}{(\mu_Q-\mu_P)}_{\h},
\end{align*}
where $(*)$ follows from Lemma \ref{lemma:format of statistic} by setting $g_{\lambda}(x)=1$ and $(\dag)$ follows by writing  $a(Y)=b(Y) + (\mu_Q-\mu_P)$ in the last two terms. Thus we have 
\begin{align*}
    \hat{D}_{\mathrm{MMD}}^2-D_{\mathrm{MMD}}^2 &= \circled{\footnotesize{1}}+\circled{\footnotesize{2}}+\circled{\footnotesize{3}}-\circled{\footnotesize{4}}-\circled{\footnotesize{5}},
\end{align*}
where $\circled{\footnotesize{1}}:=\frac{1}{N(N-1)} \sum_{i\neq j} \left \langle a(X_i),a(X_j)\right \rangle_{\h}$, $\circled{\footnotesize{2}}:=\frac{1}{M(M-1)} \sum_{i\neq j} \left \langle b(Y_i),b(Y_j)\right \rangle_{\h}$, $$\circled{\footnotesize{3}}:=\frac{2}{M}\sum_{i=1}^{M} \inner{b(Y_i)}{(\mu_Q-\mu_P)}_{\h},\, \circled{\footnotesize{4}}:=\frac{2}{NM}\sum_{i,j} \left \langle a(X_i),b(Y_j)\right \rangle_{\h},$$ and $\circled{\footnotesize{5}}:=\frac{2}{N}\sum_{i=1}^{N} \inner{a(X_i)}{(\mu_Q-\mu_P)}_{\h}$. Next, we bound the terms $\circled{\footnotesize{1}}$--$\circled{\footnotesize{5}}$ as follows (similar to the technique in the proofs of Lemmas \ref{lemma:bound U-statistic1}, \ref{lemma:bound U-statistic2}, and \ref{lemma:bound U-statistic3}):
\begin{align*}
      \E\left(\circled{\footnotesize{1}}^2\right)  \leq \frac{4}{N^2} \norm{\Sigma_P}_{\hs}^2,\qquad
        \E\left(\circled{\footnotesize{2}}^2 \right)
         \leq \frac{4}{M^2} \norm{\Sigma_Q}_{\hs}^2,
\end{align*}
\begin{align*}
        \E\left(\circled{\footnotesize{3}}^2 \right)
        \leq \frac{4}{M} \norm{\Sigma_Q}_{\op}\norm{\mu_P-\mu_Q}_{\h}^2,\quad         \E\left(\circled{\footnotesize{5}}^2 \right)
        \leq \frac{4}{N} \norm{\Sigma_P}_{\op}\norm{\mu_P-\mu_Q}_{\h}^2,
\end{align*}
\begin{align*}
\text{and} \quad       \E\left(\circled{\footnotesize{4}}^2 \right)
        & \leq \frac{4}{NM} \inner{\Sigma_P}{\Sigma_Q}_{\h} \leq \frac{4}{NM}\norm{\Sigma_P}_{\h}\norm{\Sigma_Q}_{\h}.
\end{align*}
Combining these bounds with the fact that $\sqrt{ab} \leq \frac{a}{2}+\frac{b}{2}$, and that $(\sum_{i=1}^k a_k)^2 \leq k \sum_{i=1}^k a_k^2$ for any $a,b,a_k \in \R$, $k \in \N$ yields that 
\begin{equation}
    \E[ (\hat{D}_{\mathrm{MMD}}^2-D_{\mathrm{MMD}}^2)^2] 
    \lesssim \frac{1}{N^2}+\frac{1}{M^2} + \frac{D_{\mathrm{MMD}}^2}{N} + \frac{D_{\mathrm{MMD}}^2}{M} 
    \stackrel{(*)}{\lesssim} \frac{1}{(N+M)^2} + \frac{D_{\mathrm{MMD}}^2}{N+M}, \label{Eq:MMD_var1} 
\end{equation}
where $(*)$ follows from Lemma \ref{lem:mn_bound}.

When $P=Q$, we have $\circled{\footnotesize{3}} = \circled{\footnotesize{5}}= 0$ and $\E\left(\circled{\footnotesize{1}}\cdot\circled{\footnotesize{2}}\right)=\E\left(\circled{\footnotesize{1}}\cdot\circled{\footnotesize{4}}\right)=\E\left(\circled{\footnotesize{2}}\cdot\circled{\footnotesize{4}}\right) = 0$. Therefore under $H_0$,
\begin{equation}
\E[ (\hat{D}_{\mathrm{MMD}}^2)^2] \leq 6\norm{\Sigma_P}_{\hs}^2\left(\frac{1}{N^2}+\frac{1}{M^2}\right) \stackrel{(*)}{\leq} 24\kappa^2 \left(\frac{1}{N^2}+\frac{1}{M^2}\right),\label{Eq:MMD_var2}
\end{equation}
where in $(*)$ we used $\norm{\Sigma_P}_{\hs}^2 \leq 4\kappa^2$.
Thus using \eqref{Eq:MMD_var2} and Chebyshev's inequality yields 
$$P_{H_0}\{\hat{D}_{\mathrm{MMD}}^2 \geq \gamma_1\} \leq \alpha,$$
where $\gamma_1 = \frac{2\sqrt{6}\kappa}{\sqrt{\alpha}}\left(\frac{1}{N}+\frac{1}{M}\right)$.

Moreover using the same exchangeability argument used in the proof of Theorem \ref{thm: permutations typeI}, it is easy to show that 
$$P_{H_0}\{\hat{D}_{\mathrm{MMD}}^2 \geq \gamma_2\} \leq \alpha,$$
where $\gamma_2 = q_{1-\alpha}.$
\vspace{1mm}\\
\underline{\emph{Bounding the power of the tests:}} For the threshold $\gamma_1$, we can use the bound in \eqref{Eq:MMD_var1} to bound the power.
Let $\gamma_3= \frac{1}{\sqrt{\delta}(N+M)}+\frac{\sqrt{D_{\mathrm{MMD}}^2}}{\sqrt{\delta}\sqrt{N+M}}$, then $\tilde{C}\gamma_3 \geq \sqrt{\frac{\text{Var}(\hat{D}_{\mathrm{MMD}}^2)}{\delta}}$, for some constant $\tilde{C}>0$. Thus
\begin{align*}
    P_{H_1}\{\hat{D}_{\mathrm{MMD}}^2  \geq \gamma_1\} & \stackrel{(*)}{\geq} P_{H_1}\{\hat{D}_{\mathrm{MMD}}^2 \geq D_{\mathrm{MMD}}^2-\tilde{C}\gamma_3\} \\ &\geq P_{H_1}\{|\hat{D}_{\mathrm{MMD}}^2 - D_{\mathrm{MMD}}^2|\leq \tilde{C}\gamma_3\} \stackrel{(**)}{\geq} 1- \delta,
\end{align*}
where $(*)$ holds when $D_{\mathrm{MMD}}^2 \geq \gamma_1+\tilde{C}\gamma_3$, which is implied if $D_{\mathrm{MMD}}^2 \gtrsim \frac{1}{N+M}$. $(**)$ follows from \eqref{Eq:MMD_var1} and an application of Chebyshev's inequality. The desired result, therefore, holds by taking infimum over $(P,Q) \in \PP$.

For the threshold $\gamma_2$, using a similar approach to the proof of Lemma \ref{lemma:bound quantile} we can show that  
\begin{equation}
  P_{H_1}\left\{\gamma_2 \leq \gamma_4\right\} \geq 1-\delta,  \label{Eq:bound-MMD_quantile} 
\end{equation}
where $\gamma_4 = \frac{C_1\log\frac{1}{\alpha}}{\sqrt{\delta}(M+N)}(1+D_{\mathrm{MMD}}^2)$, for some constant $C_1>0.$ Thus 
\begin{align*}
    P_{H_1}\{\hat{D}_{\mathrm{MMD}}^2  \geq \gamma_2\} & \stackrel{(*)}{\geq} P_{H_1}\left\{ \{\hat{D}_{\mathrm{MMD}}^2 \geq D_{\mathrm{MMD}}^2-\tilde{C}\gamma_3\} \cap \{\gamma_2 < \gamma_4\}\right\} \\ &\geq 1- P_{H_1}\{|\hat{D}_{\mathrm{MMD}}^2 - D_{\mathrm{MMD}}^2|\geq \tilde{C}\gamma_3\}-P_{H_1}\left\{\gamma_2 \geq \gamma_4\right\} \stackrel{(**)}{\geq} 1- 2\delta,
\end{align*}
where $(*)$ holds when $D_{\mathrm{MMD}}^2 \geq \gamma_4+\tilde{C}\gamma_3$, which is implied if $N+M \geq \frac{\log(1/\alpha)}{\sqrt{\delta}}$ and $D_{\mathrm{MMD}}^2 \gtrsim \frac{1}{N+M}$. $(**)$ follows from \eqref{Eq:MMD_var1} with an application of Chebyshev's inequality and using \eqref{Eq:bound-MMD_quantile}. 

Thus for both thresholds $\gamma_1$ and $\gamma_2$, the condition to control the power is $D_{\mathrm{MMD}}^2 \gtrsim \frac{1}{N+M}$, which in turn is implied if $\norm{u}_{\Lp}^2 \gtrsim (N+M)^{\frac{-2\theta}{2\theta+1}},$ where in the last implication we used Lemma \ref{lem:MMD bounds}. The desired result, therefore, holds by taking infimum over $(P,Q) \in \PP$.

Finally, we will show that we cannot achieve a rate better than $(M+N)^{\frac{-2\theta}{2\theta+1}}$ over $\PP$. To this end, we will first show that if $D_{\mathrm{MMD}}^2=o\left((M+N)^{-1}\right)$, then $$\liminf_{N.M \to \infty} \inf_{(P,Q)\in \PP}P_{H_1}\{\hat{D}_{\mathrm{MMD}}^2  \geq \gamma_k\} < 1$$ for $k \in \{1,2\}$. \citet[Appendix B.2]{gretton12a} show that under $H_1$,
$$(M+N)\hat{D}_{\mathrm{MMD}}^2 \stackrel{D}{\to} S+2c\left(w_x^{-1/2}z_x-w_y^{-1/2}z_y\right)+c^2,$$ 
where $z_x \sim \mathcal{N}(0,1),z_y \sim \mathcal{N}(0,1)$, $S:=\sum_{i=1}^{\infty}\lambda_i\left((w_x^{\frac{-1}{2}}a_i-w_y^{\frac{-1}{2}}b_i)^2-(w_{x}w_{y})^{-1}\right),$ $a_i\sim \mathcal{N}(0,1),$ $b_i\sim \mathcal{N}(0,1)$, $w_x:=\lim_{N,M\to \infty}\frac{N}{N+M},$ $w_y:=\lim_{N,M\to \infty}\frac{M}{N+M}$, $(\lambda_i)_{i}$ are the eigenvalues of the operator $\T$ and $c^2=(M+N)D_{\mathrm{MMD}}^2.$ If $D_{\mathrm{MMD}}^2=o\left((M+N)^{-1}\right)$, then $c \to 0$ as $N,M \to \infty$ and $(M+N)\hat{D}_{\mathrm{MMD}}^2 \stackrel{D}{\to} S$ which is the distribution under $H_0.$ Hence for $k \in \{1,2\}$,
\begin{align*}
P_{H_1}\{\hat{D}_{\mathrm{MMD}}^2  \geq \gamma_k\}=P_{H_1}\left\{(M+N)\hat{D}_{\mathrm{MMD}}^2  \geq (M+N)\gamma_k\right\} \to P\{S \geq d_k\}      
\end{align*}
where $d_k=\lim_{N,M\to \infty} (M+N)\gamma_k$, thus $d_1=\frac{2\sqrt{6}\kappa}{\sqrt{\alpha}}\left(\frac{1}{w_x}+\frac{1}{w_y}\right)>0$, and $d_2\geq 0$ using the symmetry of the permutation distribution  
. In both cases, by the definition of $S$, we can deduce that $P\{S\geq d_k\} <1$, which follows from the fact that the $S$ has a non-zero probability of being negative. Therefore it remains to show that when $\Delta_{N,M}=o\left( (N+M)^{\frac{-2\theta}{2\theta+1}}\right)$, we can construct $(P,Q) \in \PP$ such that $D_{\mathrm{MMD}}^2=o\left((M+N)^{-1}\right).$ To this end, let $R$ be a probability measure. Recall that $\T = \sum_{i \in I} \lambda_i \Tilde{\phi_i} \ltens \Tilde{\phi_i}$. Let $\bar{\phi_i} =\phi_i-\E_R\phi_i$, where $\phi_i=\frac{\id^*\Tilde{\phi_i}}{\lambda_i}$. Then $\id\bar{\phi_i}=\id\phi_i=\frac{\T\tilde{\phi}_i}{\lambda_i}=\Tilde{\phi_i}.$ Suppose $\lambda_i=h(i)$, where $h$ is a strictly decreasing continuous function on $\mathbb{N}$ (for example $h=i^{-\beta}$, and $h=e^{-\tau i}$ respectively correspond to polynomial and exponential decay). Let $L(N+M)=(\Delta_{N,M})^{1/2\theta}=o\left((N+M)^{\frac{-1}{2\theta+1}}\right)$, $k=\lfloor h^{-1}\left(L(N+M)\right)\rfloor$, hence $\lambda_k=L(N+M)$. Define $$f:=\bar{\phi}_k.$$ Then $\norm{f}_{\Lp}^2 =1,$ and thus $f \in \Lp$. Define 
$$\tilde{u}:=\T^{\theta}f=\lambda_k^{\theta}\tilde{\phi}_k,\quad\text{and}\quad u:=\lambda_k^{\theta}\bar{\phi}_k.$$ 
Note that $\E_{R} u = \lambda^\theta_k\E_{R}\bar{\phi_k} =0$. Since $\id u=\Tilde{u}$, we have $u \in [\Tilde{u}]_\sim \in \range(\T^{\theta}),$ $\norm{u}_{\Lp}^2=\lambda_k^{2\theta} = \Delta_{N,M}.$ Next we bound $|u(x)|$ in the following two cases.\\

\underline{\emph{Case I:}}  $\theta\ge\frac{1}{2}$ and $\sup_{i} \norm{\phi_i}_{\infty}$ \emph{is not finite.}\vspace{1mm}\\
Note that
$$|u(x)| =\lambda_k^{\theta}\left|\inner{k(\cdot, x)-\mu_R}{ \phi_k}_{\h}\right|  
    \leq \lambda_k^{\theta}\norm{k(\cdot, x)-\mu_R}_{\h}\norm{ \phi_k}_{\h}  
    \stackrel{(*)}{\leq} 2\sqrt{\kappa} \lambda_k^{\theta-\frac{1}{2}} \stackrel{(\dag)}{\leq} 1,$$
where in $(*)$ we used $\norm{\phi_k}^2_{\h} =\lambda_k^{-2}\inner{\id^*\tilde{\phi}_k}{\id^*\tilde{\phi}_k}=\lambda_k^{-1}.$ In $(\dag)$ we used $\theta>\frac{1}{2}.$\\

\underline{\emph{Case II:}} $\sup_{i}\norm{\phi_i}_{\infty} < \infty$.\vspace{1mm}\\
In this case,
$$|u(x)| \leq 2\sup_k\norm{\phi_k}_{\infty} \lambda_k^{\theta}\leq 1,$$ for $N+M$ large enough. This implies that in both cases we can find $(P, Q) \in \PP$ such that $\frac{dP}{dR}=u+1$ and $ \frac{dQ}{dR}=2-\frac{dP}{dR}.$ Then for such $(P,Q)$, we have $D^2_{\mathrm{MMD}}=4\norm{\T^{1/2}u}_{\Lp}^2=4\lambda_k^{2\theta+1}=o\left((M+N)^{-1}\right)$. 

\newtheorem{lem}{Lemma}[section]
\subsection{Proof of Theorem \ref{thm:minimax}} 
Let $\phi(X_1,\dots,X_N,Y_1,\dots,Y_M)$ be any test that rejects $H_0$ when $\phi=1$ and fail to reject when $\phi=0$. Fix some probability probability measure $R$, and let $\{(P_k,Q_k)\}^J_{k=1} \subset \PP$ such that $P_k+Q_k=2R$. First, we prove the following lemma.
\begin{lem} \label{lemma:conditions for lower bound}
For $0<\delta<1-\alpha$, if the following hold,
\begin{equation}\label{Eq:E.1}
    \E_{R^N}\left[\left(\frac{1}{J}\sum_{k=1}^J\frac{dP_k^N}{dR^N}\right)^2\right] \leq 1+(1-\alpha-\delta)^2,
\end{equation}
\begin{equation}\label{Eq:E.2}
    \E_{R^N}\left[\left(\frac{1}{J}\sum_{k=1}^J\frac{dQ_k^N}{dR^N}\right)^2\right] \leq 1+(1-\alpha-\delta)^2,
\end{equation}
then $$R^*_{\Delta_{N,M}} \geq \delta.$$
\end{lem}
\begin{proof}
For any two-sample test $\phi(X_1,\dots,X_N,Y_1,\dots,Y_M) \in \Phi_{N,M,\alpha}$, define a corresponding one sample test as 
$\Psi_{Q}(X_1,\dots,X_N) := \E_{Q^M}\phi(X_1,\dots,X_N,Y_1,\dots,Y_M),$ and note that it is still an $\alpha$-level test since 
$Q^N(\{\Psi_{Q}=1\})=\E_{Q^N}[\Psi_{Q}]=\E_{Q^N \times Q^M}[\phi] \leq \alpha$. For any $\phi \in \Phi_{N,M,\alpha}$, we have
\begin{align*}
&R_{\Delta_{N,M}}(\phi) =\sup_{(P,Q)\in\PP}\E_{P^N \times Q^M}[1-\phi] \\
& \geq \sup_{(P,Q) \in \PP \ : \ P+Q=2R}\E_{P^N \times Q^M}[1-\phi] 
 \geq \frac{1}{J}\sum_{k=1}^J \E_{P_k^N}[1-\Psi_{Q_k}]\\
& = \frac{1}{J}\sum_{k=1}^{J}\left[P_k^N(\{\Psi_{Q_k}=0\})-Q_k^N(\{\Psi_{Q_k}=0\})+Q_k^N(\{\Psi_{Q_k}=0\})\right]\\
& \geq 1- \alpha +  \frac{1}{J}\sum_{k=1}^{J}\left[P_k^N(\{\Psi_{Q_k}=0\})-Q_k^N(\{\Psi_{Q_k}=0\})\right] \\
& =  \frac{1}{J}\sum_{k=1}^{J}\left[P_k^N(\{\Psi_{Q_k}=0\})-R^N(\{\Psi_{Q_k}=0\})\right]\\ 
&\qquad\qquad\qquad
+\frac{1}{J}\sum_{k=1}^{J}\left[R^N(\{\Psi_{Q_k}=0\})-Q_k^N(\{\Psi_{Q_k}=0\})\right] 
 +1-\alpha\\
& \geq 1-\alpha -\sup_{\A}\left|\frac{1}{J}\sum_{k=1}^{J}\left(P_k^N(\A)-R^N(\A)\right)\right|-\sup_{\A}\left|\frac{1}{J}\sum_{k=1}^{J}\left(R^N(\A)-Q_k^N(\A)\right)\right| 
\end{align*}
\begin{align*}
& \geq 1-\alpha-\frac{1}{2}\sqrt{\E_{R^N}\left[\left (\frac{1}{J}\sum_{k=1}^J\frac{dP_k^N}{dR^N}\right)^2\right]-1}-\frac{1}{2}\sqrt{\E_{R^N}\left[\left (\frac{1}{J}\sum_{k=1}^J\frac{dQ_k^N}{dR^N}\right)^2\right]-1}\ ,
\end{align*}
where the last inequality follows by noting that
\begin{align*}
& 
\sup_{\A}\left|\frac{1}{J}\sum_{k=1}^J\left(P_k^N(\A)-R^N(\A)\right)\right|+\sup_{\A}\left|\frac{1}{J}\sum_{k=1}^J\left(R^N(\A)-Q_k^N(\A)\right)\right| \\
& \stackrel{(\dag)}{=} 
 \frac{1}{2}\norm{\frac{1}{J}\sum_{k=1}^{J}dP_k^N-dR^N}_{1}+\frac{1}{2}\norm{dR^N-\frac{1}{J}\sum_{k=1}^{J}dQ_k^N}_{1} \\
& = 
 \frac{1}{2}\E_{R^N}\left |\frac{1}{J}\sum_{k=1}^{J}\frac{dP_k^N}{dR^N}-1\right|+\frac{1}{2}\E_{R^N}\left |1-\frac{1}{J}\sum_{k=1}^{J}\frac{dQ_k^N}{dR^N}\right|\\
&\le \frac{1}{2}\sqrt{\E_{R^N}\left[\left (\frac{1}{J}\sum_{k=1}^{J}\frac{dP_k^N}{dR^N}\right)^2\right]-1}+\frac{1}{2}\sqrt{\E_{R^N}\left[\left (\frac{1}{J}\sum_{k=1}^{J}\frac{dQ_k^N}{dR^N}\right)^2\right]-1},
\end{align*}
where $(\dag)$ uses the alternative definition of total variation distance using the $L_1$-distance.
Thus taking the infimum over $\phi \in \Phi_{N,M,\alpha}$ yields
\begin{align*}
  R^*_{\Delta_{N,M}} \geq 1-\alpha-\frac{1}{2}\sqrt{\E_{R^N}\left[\left (\frac{1}{J}\sum_{k=1}^{J}\frac{dP_k^N}{dR^N}\right)^2\right]-1}-\frac{1}{2}\sqrt{\E_{R^N}\left[\left (\frac{1}{J}\sum_{k=1}^{J}\frac{dQ_k^N}{dR^N}\right)^2\right]-1}  
\end{align*}
and the result follows.
\end{proof}
Thus, in order to show that a separation boundary $\Delta_{N,M}$ will imply that $R^*_{\Delta_{N,M}} \geq \delta$, it is sufficient to find a set of distributions $\{(P_k,Q_k)\}_k$ such that \eqref{Eq:E.1} and \eqref{Eq:E.2} hold. Note that since $P_k+Q_k=2R$, it is clear that the operator $\T$ is fixed for all $k \in \{1,\dots,J\}$. Recall $\T = \sum_{i \in I} \lambda_i \Tilde{\phi_i} \ltens \Tilde{\phi_i}$. Let $\bar{\phi_i} =\phi_i-\E_R\phi_i$, where $\phi_i=\frac{\id^*\Tilde{\phi_i}}{\lambda_i}$. Then $\id\bar{\phi_i}=\id\phi_i=\frac{\T\tilde{\phi}_i}{\lambda_i}=\Tilde{\phi_i}$. Furthermore, since the lower bound on $R^*_{\Delta_{N,M}}$ is only in terms of $N$, for simplicity, we will write $\Delta_{N,M}$ as $\Delta_{N}$. However, since we assume $M \leq N \leq DM$, we can write the resulting bounds in terms of $(N+M)$ using Lemma~\ref{lem:mn_bound}.

Following the ideas from \cite{Ingester1} and \cite{Ingester2}, we now provide the proof of Theorem~\ref{thm:minimax}. \vspace{.5mm}\\

\underline{\emph{Case I:}}  $\sup_{i}\norm{\phi_i}_{\infty} < D<\infty$.\vspace{.5mm}\\

Let $$B_N=\min\left\{L^{-1}\left(\Delta_{N}^{1/2\theta}\right),\frac{1}{(2D)^4\Delta_N^2}\right\},$$ $C_N=\floor{\sqrt{B_N}}$ and $a_N=\sqrt{\frac{\Delta_N}{C_N}}$. For $k \in \{1,\dots,J\}$,
define $$f_{N,k}:= a_N \sum_{i=1}^{B_N}  \epsilon_{ki}\lambda_i^{-\theta} \bar{\phi_i},$$ where $\epsilon_k := \{\epsilon_{k1}, \epsilon_{k2},\ldots, \epsilon_{kB_N}\}\in \{0,1\}^{B_N}$ such that $\sum_{i=1}^{B_N}\epsilon_{ki}=C_N$, thus $J={B_N \choose C_N}$. Then we have
$\norm{f_{N,k}}^2_{\Lp} = a_N^2 \sum_{i=1}^{B_N} \lambda_i^{-2\theta}\epsilon_{ki}^2 \lesssim a_N^{2} (L(B_N))^{-2\theta}C_N \stackrel{(*)}{\lesssim} (\Delta_N^{1/2\theta})^{-2\theta}\Delta_N   \lesssim 1$, where in $(*)$ we used that $L(\cdot)$ is a decreasing function. Thus,  $f_{N,k} \in \Lp$. Define $$\Tilde{u}_{N,k} := \T^{\theta} f_{N,k} = a_N \sum_{i=1}^{B_N} \lambda_i^{\theta} \inner{\epsilon_{ki}\lambda_i^{-\theta}\bar{\phi_i}}{\Tilde{\phi_i}}_{\Lp}\tilde{\phi_i}=a_N \sum_{i=1}^{B_N}\epsilon_{ki} \Tilde{\phi_i},$$ and $$u_{N,k}:=a_N\sum_{i=1}^{B_N}\epsilon_{ki} \bar{\phi_i}.$$ Note that $\E_{R} u_{N,k} = a_N \sum_{i=1}^{B_N} \epsilon_{ki} \E_{R}\bar{\phi_i} =0$. Since $\id u_{N,k}=\Tilde{u}_{N,k}$,
we have $u_{N,k}\in [\Tilde{u}_{N,k}]_\sim \in \range(\T^{\theta})$, 
$\norm{u_{N,k}}_{\Lp}^2 = a_N^2 C_N = \Delta_N,$ and
\begin{align*}
|u_{N,k}(x)| &\leq 2a_N \sum_{i=1}^{B_N}\epsilon_{ki} \norm{\phi_i}_{\infty} \leq 2a_N C_N D = 2 \sqrt{\Delta_N} B_N^{1/4} D \leq  1.
\end{align*}
This implies that we can find $(P_{k}, Q_{k}) \in \PP$ such that $\frac{dP_{k}}{dR}=u_{N,k}+1$ and $ \frac{dQ_{k}}{dR}=2-\frac{dP_{k}}{dR}$. Thus it remains to verify conditions \eqref{Eq:E.1} and \eqref{Eq:E.2} from Lemma \ref{lemma:conditions for lower bound}. 

For condition \eqref{Eq:E.1}, we have 
\begin{align*}
  &\E_{R^N}\left[\left(\frac{1}{J}\sum_{k=1}^J\frac{dP_k^N}{dR^N}\right)^2\right] = \E_{R^N}\left[\left(\frac{1}{J}\sum_{k=1}^J\prod_{j=1}^N(u_{N,k}(X_j)+1)\right)^2\right]  \\
  &= \E_{R^N} \left[\frac{1}{J} \sum_{k=1}^{J} \prod_{j=1}^{N} \left(1+ a_N \sum_{i=1}^{B_N} \epsilon_{ki} \bar{\phi_i}(X_j)\right)   \right]^2 \\
&=\frac{1}{J^2}\sum_{k,k'=1}^{J}  \prod_{j=1}^{N} \left(1+ a_N \sum_{i=1}^{B_N} \epsilon_{ki} \E_R\bar{\phi_i}(X_j)\right. \\
     & \qquad \qquad \qquad\left.+a_N \sum_{i=1}^{B_N} \epsilon_{k'i} \E_R\bar{\phi_i}(X_j) + a_N^2 \sum_{i,l=1}^{B_N} \epsilon_{ki}\epsilon_{k'l}\E_R[\bar{\phi_i}(X_j)\bar{\phi_l}(X_j)] \right)\\
    & \stackrel{(\dag)}{\leq} \frac{1}{J^2} \sum_{k ,k'=1}^{J} \left(1+a_N^2 \sum_{i=1}^{B_N} \epsilon_{ki} \epsilon_{k'i}\right)^N,
\end{align*}
where $(\dag)$ follows from the fact that $\E_R[\bar{\phi_i}(X)]=0$ and $\E_R[\bar{\phi_i}(X)\bar{\phi_l}(X)]=\langle \bar{\phi}_i,\bar{\phi}_l\rangle_{L^2(R)}=\delta_{il}$.
Similarly for \eqref{Eq:E.2}, we have
\begin{align*}
    \E_{R^N}\left[\left(\frac{1}{J}\sum_{i=1}^J\frac{dQ_k^N}{dR^N}\right)^2\right] \leq \frac{1}{J^2} \sum_{k ,k'=1}^{J} \left(1+a_N^2 \sum_{i=1}^{B_N} \epsilon_{ki} \epsilon_{k'i}\right)^N,
\end{align*}
where $J= {B_N \choose C_N}$.
Thus it is sufficient to show that $\exists c(\alpha,\delta)$ such that if $\Delta_{N}  \leq  c(\alpha,\delta) N^{\frac{-4\theta\beta}{4\theta\beta+1}}$, then $\frac{1}{J^2} \sum_{k ,k'=1}^{J} (1+a_N^2 \sum_{i=1}^{B_N} \epsilon_{ki} \epsilon_{k'i})^N \leq 1+(1-\alpha-\delta)^2$. To this end, consider
\begin{align*}
  &\frac{1}{J^2} \sum_{k ,k'=1}^{J} \left(1+a_N^2 \sum_{i=1}^{B_N} \epsilon_{ki} \epsilon_{k'i}\right)^N =  \frac{1}{J^2}{B_N \choose C_N} \sum_{i=0}^{C_N} {B_N-C_N \choose C_N-i}{C_N \choose i} (1+a_N^2i)^N \\
  &\leq \sum_{i=0}^{C_N} \frac{{B_N-C_N \choose C_N-i}{C_N \choose i}}{{B_N \choose C_N}} \exp(Na_N^{2}i) 
  = \sum_{i=0}^{C_N}\frac{H_i}{i!}\exp(Na_N^{2}i),
\end{align*}
where $H_i:=\frac{((B_N-C_N)!C_N!)^2}{((C_N-i)!)^2B_N!(B_N-2C_N+i)!}$. Then as argued in \cite{Ingester1}, it can be shown for any $r>0$, we have $\frac{H_i}{H_{i-1}}\leq 1+r$ and $H_0\leq \exp(r-1)$, thus $H_i \leq \frac{(1+r)^i}{e}$, which yields that for any $r >0$, we have
\begin{align*}
    \sum_{i=0}^{C_N}\frac{H_i}{i!}\exp(Na_N^{2}i) &\leq \exp\left((1+r)\exp(N\Delta_N B_N^{-1/2})-1\right).
\end{align*}
Since $N\Delta_N \leq c(\alpha,\delta)B_N^{1/2},$ we can choose $c(\alpha,\delta)$ and $r$ such that $$\exp\left((1+r)\exp(N\Delta_N B_N^{-1/2})-1\right) \leq 1+(1-\alpha-\delta)^2$$ and therefore both \eqref{Eq:E.1} and \eqref{Eq:E.2} hold.\vspace{.5mm}\\

\underline{\emph{Case II:}} $\sup_{i} \norm{\phi_i}_{\infty}$ \emph{is not finite.}\vspace{.5mm}\\

Since $\lambda_i \asymp L(i)$, there exists constants $\underbar{A}>0$ and $\bar{A}>0$ such that $\underbar{A}L(i) \leq \lambda_i \leq \bar{A} L(i)$. Let 

$$B_N=\min\left\{L^{-1}\left(\Delta_{N}^{1/2\theta}\right),L^{-1}\left(4\kappa\underline{A}^{-1}\Delta_{N}\right)\right\},$$ $C_N=\floor{\sqrt{B_N}}$ and $a_N=\sqrt{\frac{\Delta_N}{C_N}}$. The proof proceeds similarly to that of \emph{Case I} by noting that  
\begin{align*}
    |u_{N,k}(x)| &=  \left|\inner{k(\cdot, x)-\mu_R}{a_N\sum_{i=1}^{B_N} \epsilon_{ki} \phi_i}_{\h}\right|  
    \leq \norm{k(\cdot, x)-\mu_R}_{\h}\norm{a_N\sum_{i=1}^{B_N} \epsilon_{ki} \phi_i}_{\h}  \\
    & \leq 2\sqrt{\kappa} \norm{a_N\sum_{i=1}^{B_N} \epsilon_{ki} \phi_i}_{\h} \stackrel{(*)}{\leq} 1,
\end{align*}
where $(*)$ follows from
\begin{align*}
    \norm{a_N\sum_{i=1}^{B_N} \epsilon_{ki} \phi_i}_{\h}^2 &= a_N^2\sum_{i,j=1}^{B_N}\epsilon_{ki}\epsilon_{kj}\inner{\phi_i}{\phi_j}_{\h}  
    = a_N^2\sum_{i,j=1}^{B_N}\lambda_i^{-1}\epsilon_{ki}\epsilon_{kj}\inner{\Tilde{\phi_i}}{\Tilde{\phi_j}}_{\Lp} \\ 
    & = a_N^2\sum_{i=1}^{B_N}\lambda_i^{-1}\epsilon_{ki}^2 \leq a_N^2 \underbar{A}^{-1} (L(B_N))^{-1}C_N \stackrel{(*)}{\leq} \Delta_N \underbar{A}^{-1} \frac{1}{4\kappa\underline{A}^{-1}\Delta_{N}}
    \leq  \frac{1}{4\kappa},  
\end{align*}
where $(*)$ follows since $L(\cdot)$ is a decreasing function. The rest of the proof proceeds exactly similar to that of \emph{Case I}.

\subsection{Proof of Corollary~\ref{coro:poly-minimax}}
If $\lambda_i \asymp i^{-\beta},$ $\beta>1$, then Theorem \ref{thm:minimax} yields that $R^*_{\Delta_{N,M}} > \delta$, if $$(N+M)\Delta_{N,M}  \lesssim  \sqrt{\min\left\{\left(\Delta_{N,M}^{1/2\theta}\right)^{-1/\beta},\left(\Delta_{N,M}\right)^{-1/\beta}\right\}},$$
which is equivalent to    
$\Delta_{N,M}\lesssim \min\left\{\left(N+M\right)^{\frac{-4\theta\beta}{1+4\theta\beta}},\left(N+M\right)^{\frac{-2\beta}{1+2\beta}}\right\}.$  This implies the following lower bound on the minimax separation 
 \begin{equation}
\Delta^*_{N,M}\gtrsim \min\left\{\left(N+M\right)^{\frac{-4\theta\beta}{1+4\theta\beta}},\left(N+M\right)^{\frac{-2\beta}{1+2\beta}}\right\}. \label{Eq:minimax-lowerbound-polyI}    
\end{equation} Observe that $\frac{4\theta\beta}{1+4\theta\beta} \geq \frac{2\beta}{1+2\beta}$ when $\theta\geq\frac{1}{2}$. Matching the upper and lower bound on $\Delta^*_{N,M}$ by combining \eqref{Eq:minimax-lowerbound-polyI} with the result from Corollary \ref{coro:poly} when $\xi=\infty$ implies that 
$\Delta^*_{N,M} \asymp \left(N+M\right)^{\frac{-4\theta\beta}{1+4\theta\beta}}$ when $\theta>\frac{1}{2}$. Similarly for the case $\sup_{k}\norm{\phi_k}_{\infty} < \infty$, Theorem \ref{thm:minimax} yields that $R^*_{\Delta_{N,M}} > \delta$ if
$$(N+M)\Delta_{N,M}  \lesssim  \sqrt{\min\left\{\left(\Delta_{N,M}^{1/2\theta}\right)^{-1/\beta},\Delta_{N,M}^{-2}\right\}},$$
which implies that 
 \begin{equation}
\Delta^*_{N,M}\gtrsim \min\left\{\left(N+M\right)^{\frac{-4\theta\beta}{1+4\theta\beta}},\left(N+M\right)^{-1/2}\right\}. \label{Eq:minimax-lowerbound-polyII}    
\end{equation}
Then combining \eqref{Eq:minimax-lowerbound-polyII} with the result from Corollary \ref{coro:poly} when $\xi=\infty$ yields that $\Delta^*_{N,M} \asymp \left(N+M\right)^{\frac{-4\theta\beta}{1+4\theta\beta}}$ when $\theta > \frac{1}{4\beta}$.  

\subsection{Proof of Corollary~\ref{coro:exp-minimax}}
If $\lambda_i \asymp e^{-\tau i}$, then Theorem \ref{thm:minimax} yields that $R^*_{\Delta_{N,M}} > \delta$, if $$(N+M)\Delta_{N,M}  \lesssim  \sqrt{\log\left(\Delta_{N,M}^{-1}\right)},$$ which is implied if 
$$\Delta_{N,M} \lesssim \frac{(\log (N+M))^b}{N+M},$$
for any $b<\frac{1}{2}$ and $N+M$ large enough. 
Furthermore if $\sup_{k}\norm{\phi_k}_{\infty} < \infty$, Theorem \ref{thm:minimax} yields that
$$(N+M)\Delta_{N,M}  \lesssim  \sqrt{\min \left\{\log\left(\Delta_{N,M}^{-1}\right), \Delta_{N,M}^{-2}\right\}},$$ which is implied if 
$\Delta_{N,M} \lesssim \min\left\{\frac{(\log (N+M))^b}{N+M},(N+M)^{-1/2}\right\}= \frac{(\log (N+M))^b}{N+M},$ for any $b<\frac{1}{2}$ and $N+M$ large enough. Thus for both cases, we can deduce that $\Delta^*_{N,M} \gtrsim \frac{(\log (N+M))^b}{N+M},$ for any $b<\frac{1}{2}$, then by taking supremum over $b<\frac{1}{2}$ yields that
\begin{equation}
   \Delta^*_{N,M} \gtrsim \frac{\sqrt{(\log (N+M))}}{N+M}.  \label{Eq:minimax-lowerbound-exp}
\end{equation}
 Matching the upper and lower bound on $\Delta^*_{N,M}$ by combining \eqref{Eq:minimax-lowerbound-exp} with the result from Corollary \ref{coro:exp} when $\xi=\infty$ yields the desired result.

\subsection{Proof of Theorem~\ref{thm: computation}}
Define $$S_z : \h \to \R^s,\ \ f \to \frac{1}{\sqrt{s}}(f(Z_1),\cdot \cdot \cdot,f(Z_s))^\top$$ so that $$S_z^* : \R^s \to \h,\ \ \alpha \to \frac{1}{\sqrt{s}}\sum_{i=1}^{s} \alpha_iK(\cdot,Z_i).$$ It can be shown that $\hat{\Sigma}_{PQ}=S_z^* \Tilde{\hh}_s S_z$ \cite[Proposition C.1]{kpca-nystrom}. Also, it can be shown that if $(\hat{\lambda}_i, \hat{\alpha_i})_i$ is the eigensystem of $\frac{1}{s} \Tilde{\hh_s}^{1/2}K_s\Tilde{\hh_s}^{1/2}$ where $K_s:=[K(Z_i,Z_j)]_{i,j \in [s]}$, $\hh_s= \Id_s -\frac{1}{s}\one_s \one_s^\top$, and $\Tilde{\hh}_s=\frac{s}{s-1}\hh_s$, then $(\hat{\lambda}_i, \hat{\phi_i})_i$ is the eigensystem of $\hat{\Sigma}_{PQ}$, where \begin{equation}\hat{\phi}_i = \frac{1}{\sqrt{\hat{\lambda}_i}}S_z^* \Tilde{\hh_s}^{1/2} \hat{\alpha_i}.\label{Eq:eigfn}\end{equation} We refer the reader to \citet[Proposition 1]{kpca} for details.
Using \eqref{Eq:eigfn} in the definition of $\gSh$, we have 
\begin{align*}
    \gSh &= g_{\lambda}(0)\Id + S_z^*\Tilde{\hh}_{s}^{1/2}\left[\sum_{i}\left( \frac{g_{\lambda}(\hat{\lambda}_i)-g_{\lambda}(0)}{\hat{\lambda}_i}\right)\hat{\alpha}_i\hat{\alpha}_i^{\top}\right]\Tilde{\hh}_{s}^{1/2}S_z \\
    & = g_{\lambda}(0)\Id + S_z^*\Tilde{\hh}_{s}^{1/2} G \Tilde{\hh}_{s}^{1/2} S_z.
\end{align*}
Define $\one_n : =(1,\stackrel{n}{\ldots},1)^{\top}$, and let $\one_n^i$ be a vector of zeros with only the $i^{th}$ entry equal one. Also we define $S_x$ and $S_y$ similar to that of $S_z$ based on samples $(X_i)^n_{i=1}$ and $(Y_i)^m_{i=1}$, respectively. Based on \citet[Proposition C.1]{kpca-nystrom}, it can be shown that $K_s = sS_zS_z^*$, $K_n:=nS_xS_x^*$, $K_m=mS_{y}S_{y}^*$,  $K_{ns}=\sqrt{ns}S_xS_z^*$,  
$K_{ms}=\sqrt{ms}S_{y}S_z^*$, and $K_{mn}=\sqrt{mn}S_{y}S_x^*$. 

Based on these observations, we have
\begin{align*}
    \circled{\small{1}} & = \inner{\gSh\sum_{i=1}^nK(\cdot,X_i)}{\sum_{j=1}^nK(\cdot,X_j)}_{\h}=n\inner{\gSh S_x^*\one_n}{S_x^*\one_n}_{\h}\\
    &= n\inner{S_x\gSh S_x^*\one_n}{\one_n}_{2}\\
     &= n\inner{g_{\lambda}(0)\frac{1}{n}K_n+\frac{1}{ns}K_{ns}\hgh K_{ns}^\top\one_n}{\one_n}_{2} \\
    & = \one_n^\top\left(g_{\lambda}(0)K_n+\frac{1}{s}K_{ns}\Tilde{\hh}^{1/2}_{s}G\Tilde{\hh}^{1/2}_{s}K_{ns}^\top\right)\one_n,
\end{align*}
\begin{align*}
    \circled{\small{2}} & = \sum_{i=1}^n\inner{\gSh K(\cdot,X_i)}{K(\cdot,X_i)}_{\h}=n\sum_{i=1}^{n}\inner{\gSh S_x^*\one_n^i}{S_x^*\one_n^i}_{\h}\\ 
    & = n\sum_{i=1}^{n}\inner{S_x\gSh S_x^*\one_n^i}{\one_n^i}_{2}\\
    &=n\sum_{i=1}^{n}\inner{g_{\lambda}(0)\frac{1}{n}K_n+\frac{1}{ns}K_{ns}\hgh K_{ns}^\top\one_n^i}{\one_n^i}_{2}\\
    & = \text{Tr}\left(g_{\lambda}(0)K_n+\frac{1}{s}K_{ns}\Tilde{\hh}^{1/2}_{s}G\Tilde{\hh}^{1/2}_{s}K_{ns}^\top\right),
\end{align*}
\begin{align*}
    \circled{\small{3}}&= \inner{\gSh\sum_{i=1}^mK(\cdot,Y_i)}{\sum_{j=1}^mK(\cdot,Y_j)}_{\h}\\
    &=\one_m^{\top}\left(g_{\lambda}(0)K_m+\frac{1}{s}K_{ms}\Tilde{\hh}^{1/2}_{s}G\Tilde{\hh}^{1/2}_{s}K_{ms}^\top\right)\one_m,
\end{align*}
\begin{align*}
    \circled{\small{4}}& = \sum_{i=1}^m\inner{\gSh K(\cdot,Y_i)}{K(\cdot,Y_i)}_{\h}\\
    &=\text{Tr}\left(g_{\lambda}(0)K_m+\frac{1}{s}K_{ms}\Tilde{\hh}^{1/2}_{s}G\Tilde{\hh}^{1/2}_{s}K_{ms}^\top\right),
\end{align*}
and
\begin{align*}
    \circled{\small{5}} & = \inner{\gSh\sum_{i=1}^n K(\cdot,X_i)}{\sum_{i=1}^m K(\cdot,Y_i)}_{\h}\\
    &=\sqrt{nm}\inner{\gSh S_x^*\one_n}{S_y^*\one_m}_{\h}\\ 
    & =\sqrt{nm} \inner{S_y\gSh S_x^*\one_n}{\one_m}_{2}\\
    &= \sqrt{nm}\inner{g_{\lambda}(0)\frac{1}{\sqrt{nm}}K_{mn}+\frac{1}{s\sqrt{nm}}K_{ms}\hgh K_{ns}^\top\one_n}{\one_m}_{2} \\
    & = \one_m^\top\left(g_{\lambda}(0)K_{mn}+\frac{1}{s}K_{ms}\Tilde{\hh}^{1/2}_{s}G\Tilde{\hh}^{1/2}_{s}K_{ns}^\top\right)\one_n.
\end{align*}

\subsection{Proof of Theorem \ref{thm: Type1 error}}
Since $\E(\stat | (Z_i)_{i=1}^{s} )=0$, an application of Chebyshev's inequality via Lemma \ref{Lemma: bounding expectations} yields,  
\begin{equation*} 
        P_{H_0}\left\{|\stat| \geq \frac{\sqrt{6}\Cs\norm{\M}_{\op}^{2}\Ntl}{\sqrt{\delta}}\left(\frac{1}{n}+\frac{1}{m}\right) \Big| (Z_i)_{i=1}^{s}\right\} \leq \delta.
\end{equation*} 
By defining $$\gamma_1 : = \frac{2\sqrt{6}\Cs\Ntl}{\sqrt{\delta}}\left(\frac{1}{n}+\frac{1}{m}\right),$$ $$\gamma_2 : = \frac{\sqrt{6}\Cs\norm{\M}_{\op}^{2}\Ntl}{\sqrt{\delta}}\left(\frac{1}{n}+\frac{1}{m}\right),$$ we obtain
\begin{align*}
  P_{H_0}\{\stat \leq \gamma_1\} & \geq P_{H_0}\{ \{\stat \leq \gamma_2\} \ \cap \{\gamma_2 \leq \gamma_1\}\} \\
  & \geq 1- P_{H_0}\{\stat \geq \gamma_2\} - P_{H_0}\{\gamma_2 \geq \gamma_1\} 
   \stackrel{(*)}{\geq} 1-3\delta,
\end{align*}
where $(*)$ follows using
\begin{align*}
P_{H_0}\{\stat \geq \gamma_2\} \leq P_{H_0}\{|\stat| \geq \gamma_2\} = \E_{\PQ^s}\left[P_{H_0}\{|\stat| \geq \gamma_2| (Z_i)_{i=1}^{s}\}\right] \leq \delta,
\end{align*}
and 
\begin{align*}
    P_{H_0}\{\gamma_2 \geq \gamma_1\} = P_{H_0}\{\norm{\M}_{\op}^2 \geq 2\} \stackrel{(\dag)}{\leq} 2\delta, 
\end{align*}
where $(\dag)$ follows from \citep[Lemma B.2\emph{(ii)}]{kpca}, under the condition that $\frac{140\K}{s}\log \frac{16\K s}{\delta} \leq \lambda \leq \norm{\Sigma_{PQ}}_{\op}$.  When $C:=\sup_i\norm{\phi_i}_{\infty} < \infty$, using Lemma \ref{lem: bound M operator}, we can obtain an improved condition on $\lambda$ satisfying $136C^2\Nol\log\frac{8\Nol}{\delta} \leq s$ and $\lambda \leq \norm{\Sigma_{PQ}}_{\op}$. The desired result then follows by setting $\delta = \frac{\alpha}{3}.$

\subsection{Proof of Theorem \ref{thm:Type II}}
Let $\M = \SLh^{-1/2}\SL^{1/2}$, and $$\gamma_1 =\frac{1}{\sqrt{\delta}} \left(\frac{\sqrt{\Cl} \norm{\U}_{\Lp}+\Ntl}{n+m}+\frac{\Cl^{1/4}\norm{\U}_{\Lp}^{3/2}+\norm{\U}_{\Lp}}{\sqrt{n+m}}\right),$$
where $\Cl$ is defined in Lemma \ref{lemma: bound hs and op}. Then Lemma \ref{Lemma: bounding expectations} implies $$\tilde{C}\norm{\M}_{\op}^2\gamma_1\geq\sqrt{\frac{\text{Var}(\stat|(Z_i)_{i=1}^s)}{\delta}}$$ for some constant $\Tilde{C}>0$. By Lemma \ref{thm: general Type2 error}, if \begin{equation}P\left\{\gamma \geq \zeta - \tilde{C}\norm{\M}_{\op}^2\gamma_1\right\} \leq \delta,\label{Eq:alter}\end{equation} for any $(P,Q) \in \PP$, then we obtain $P\{\hat{\eta}_\lambda \geq \gamma\} \geq 1-2\delta.$ The result follows by taking the infimum over $(P,Q) \in \PP$. Therefore, it remains to verify \eqref{Eq:alter}, which we do below. Define $c_2 :=B_3C_4 \Cs^{-1}$. Consider
\begin{align*}
    &P_{H_1}\{\gamma \leq \zeta - \tilde{C}\norm{\M}_{\op}^2\gamma_1\} \\
    &\stackrel{(\ddag)}{\geq} P_{H_1}\left\{\gamma \leq c_2\norm{\M^{-1}}^{-2}_{\op} \norm{u}_{\Lp}^2 - \tilde{C}\norm{\M}_{\op}^2\gamma_1\right\} \\
    & = P_{H_1}\left\{\frac{\norm{\M^{-1}}^{2}_{\op}\gamma+\tilde{C}\gamma_1\norm{\M^{-1}}^{2}_{\op}\norm{\M}_{\op}^2}{c_2 \norm{u}_{\Lp}^2 }\leq 1\right\} \\
    &\stackrel{(*)}{\geq} P_{H_1}\left\{\frac{\norm{\M^{-1}}^{2}_{\op}}{3} + \frac{\norm{\M^{-1}}^{2}_{\op}\norm{\M}_{\op}^2}{6}\leq 1\right\} \\
    & \geq P_{H_1}\left\{\left\{\norm{\M^{-1}}^{2}_{\op} \leq \frac{3}{2}\right\} \ \cap \left\{\norm{\M}^{2}_{\op} \leq 2\right\}\right\}\\
    & \geq 1-P_{H_1}\left\{\norm{\M^{-1}}^{2}_{\op} \geq \frac{3}{2}\right\} - P_{H_1}\left\{\norm{\M}^{2}_{\op} \geq 2\right\} \\
    & \stackrel{(\dag)}{\geq} 1-\delta,
\end{align*}
where $(\ddag)$ follows by using $\zeta \geq  c_2 \norm{\M^{-1}}^{-2}_{\op} \norm{u}_{\Lp}^2$, which is obtained by combining Lemma \ref{lemma: denominator lower bound} with Lemma \ref{lemma: bounds for eta} under the assumptions $u \in \text{\range}(\T^{\theta})$, and \begin{equation}\norm{\U}_{\Lp}^2 \geq  \frac{4C_3}{3B_3} \norm{\T}_{\op}^{2\max(\theta-\xi,0)}\lambda^{2 \Tilde{\theta}} \norm{\T^{-\theta}\U}_{\Lp}^2.\label{Eq:lower}\end{equation}Note that $u \in \text{\range}(\T^{\theta})$ is guaranteed since $(P,Q) \in \PP$ and 
\begin{equation} 
\norm{u}_{\Lp}^2 \geq c_4\lambda^{2 \Tilde{\theta}}\label{Eq:verify-1}  
\end{equation} guarantees \eqref{Eq:lower}
since $\norm{\T}_{\op}=\norm{\Sigma_{PQ}}_{\op}\leq 2\kappa$ and $$c_1:=\sup_{(P,Q) \in \PP} \norm{\T^{-\theta}u}_{\Lp}<\infty,$$ where 
$c_4=\frac{4c_1^{2}C_3(2\kappa)^{2\max(\theta-\xi,0)}}{3B_3}.$ $(*)$ follows when
\begin{equation}
 \norm{u}_{\Lp}^2 \geq \frac{3\gamma}{c_2}   \label{Eq:verify-2}
\end{equation}
and 
\begin{equation}
    \norm{u}_{\Lp}^2 \geq \frac{6\tilde{C}\gamma_1}{c_2},\label{Eq:verify-3}
\end{equation}
and $(\dag)$ follows from \citep[Lemma B.2\emph{(ii)}]{kpca}, under the condition that
\begin{equation}
\frac{140\K}{s}\log \frac{64\K s}{\delta} \leq \lambda \leq \norm{\Sigma_{PQ}}_{\op}.   \label{Eq:verify-4}
\end{equation}
When $C:=\sup_i\norm{\phi_i}_{\infty} < \infty$, $(\dag)$ follows from Lemma \ref{lem: bound M operator} by replacing \eqref{Eq:verify-4} with 
\begin{equation}
136C^2\Nol\log\frac{32\Nol}{\delta} \leq s, \qquad \lambda \leq \norm{\Sigma_{PQ}}_{\op}.\label{Eq:verify-5}
\end{equation}
Below, we will show that \eqref{Eq:verify-2}--\eqref{Eq:verify-5} are satisfied.
Using $\norm{u}^2_{\Lp} \geq \Delta_{N,M}$, it is easy to see that \eqref{Eq:verify-1} is implied when $\lambda=(c_4^{-1}\Delta_{N,M})^{1/2\tilde{\theta}}$. Using $(n+m)=(1-d_1)N+(1-d_2)M \geq (1-d_2)(N+M)$, where in the last inequality we used $d_2\geq d_1$ since $s=d_1N=d_2M$ and $M\leq N$, and applying Lemma \ref{lem:mn_bound}, we can verify that \eqref{Eq:verify-2} is implied if $\Delta_{N,M} \geq \frac{r_1\Ntl}{\sqrt{\alpha}(N+M)}$ for some constant $r_1>0$. It can be also verified that \eqref{Eq:verify-3} is implied if $\Delta_{N,M} \geq \frac{r_2\Cl}{\delta^2(N+M)^2}$ and $\Delta_{N,M}\geq\frac{r_3\Ntl}{\delta(N+M)}$ for some constants $r_2,r_3>0$. Using $s=d_1N \geq \frac{d_1(N+M)}{2}$, $(N+M) \geq \frac{32 \K d_1}{\delta}$ and $\norm{\Sigma_{PQ}}_{\op} \geq (c_4^{-1}\Delta_{N,M})^{\frac{1}{2\Tilde{\theta}}}$, it can be seen that \eqref{Eq:verify-4} is implied when $\Delta_{N,M} \geq r_4(\frac{280\K}{d_1})^{2\tilde{\theta}}\left(\frac{N+M}{\log(N+M)}\right)^{-2\tilde{\theta}}$ for some constant $r_4>0$. On the other hand, when $C:=\sup_i\norm{\phi_i}_{\infty} < \infty$, using $s\geq \frac{d_1(N+M)}{2}$, $(N+M)\geq \max\{\frac{32}{\delta},e^{d_1/272C^2}\}$, it can be verified that \eqref{Eq:verify-5} is implied when $\Nol \leq \frac{d_1(N+M)}{544C\log (N+M)}.$
\subsection{Proof of Corollary \ref{coro:poly}}
When $\lambda_i \lesssim i^{-\beta}$, it follows from (\citealt[Lemma B.9]{kpca}) that $$\Ntl \leq \norm{\SgL\Sigma_{PQ}\SgL}_{\op}^{1/2}\mathcal{N}^{1/2}_{1}(\lambda) \lesssim \lambda^{-1/2\beta}.$$ Using this bound  in the conditions mentioned in Theorem~\ref{thm:Type II}, ensures that these conditions on the separation boundary hold if 
\begin{align*}
\Delta_{N,M} &\gtrsim 
\max\left\{\left(\frac{N+M}{\alpha^{-1/2}+\delta^{-1}}\right)^{-\frac{4\Tilde{\theta}\beta}{4\Tilde{\theta}\beta+1}}, (\delta(N+M))^{-\frac{8\Tilde{\theta}\beta}{4\Tilde{\theta}\beta+2\beta+1}},d_4^{2\tilde{\theta}}\left(\frac{N+M}{\log(N+M)}\right)^{-2\tilde{\theta}} \right\},
\end{align*}
where $d_4>0$ is a constant. The above condition is 
implied if 
$$\Delta_{N,M} =
\left\{
	\begin{array}{ll}
		c(\alpha,\delta,\theta)\left(N+M\right)^{\frac{-4\tilde{\theta}\beta}{4\Tilde{\theta}\beta+1}},  &  \ \  \Tilde{\theta}> \frac{1}{2}-\frac{1}{4\beta} \\
		c(\alpha,\delta,\theta)\left(\frac{N+M}{\log(N+M)}\right)^{-2\tilde{\theta}}, & \ \  \Tilde{\theta} \le \frac{1}{2}-\frac{1}{4\beta}
	\end{array}
\right.,$$
where $c(\alpha,\delta,\theta)\gtrsim(\alpha^{-1/2}+\delta^{-2}+d_4^{2\tilde{\theta}})$ and we used that $\tilde{\theta}> \frac{1}{2}-\frac{1}{4\beta}\Rightarrow\frac{4\tilde{\theta}\beta}{4\tilde{\theta}\beta+1} <  \min\left\{2\tilde{\theta},\frac{8\tilde{\theta}\beta}{4\tilde{\theta}\beta+2\beta+1}\right\}$, $\tilde{\theta}\le \frac{1}{2}-\frac{1}{4\beta}\Rightarrow 2\tilde{\theta} \le  \min\left\{\frac{4\tilde{\theta}\beta}{4\tilde{\theta}\beta+1},\frac{8\tilde{\theta}\beta}{4\tilde{\theta}\beta+2\beta+1}\right\},$ and that $x^{-a}\geq \left(\frac{x}{\log x}\right)^{-b}$, when $a<b$ and $x$ is large enough, for any $a,b,x \in \R$. 

On the other hand when $C:=\sup_i\norm{\phi_i}_{\infty} < \infty$, we obtain the corresponding condition as 
\begin{align*}
\Delta_{N,M} &\gtrsim \max\left\{\left(\frac{N+M}{\alpha^{-1/2}+\delta^{-1}}\right)^{-\frac{4\Tilde{\theta}\beta}{4\Tilde{\theta}\beta+1}},(\delta(N+M))^{-\frac{4\Tilde{\theta}\beta}{2\Tilde{\theta}\beta+1}},d_5^{2\tilde{\theta}\beta}\left(\frac{N+M}{\log(N+M)}\right)^{-2\tilde{\theta}\beta} \right\},
\end{align*}
for some constant $d_5>0$, which in turn is implied for $N+M$ large enough, if 
$$\Delta_{N,M} =
\left\{
	\begin{array}{ll}
		c(\alpha,\delta,\theta,\beta)\left(N+M\right)^{\frac{-4\tilde{\theta}\beta}{4\Tilde{\theta}\beta+1}},  &  \ \  \Tilde{\theta}> \frac{1}{4\beta} \\
		c(\alpha,\delta,\theta,\beta)\left(\frac{N+M}{\log(N+M)}\right)^{-2\tilde{\theta}
		\beta}, & \ \ \Tilde{\theta} \le \frac{1}{4\beta}
	\end{array}
\right.,$$
where $c(\alpha,\delta,\theta,\beta)\gtrsim(\alpha^{-1/2}+\delta^{-2}+d_5^{2\tilde{\theta}\beta}).$

\subsection{Proof of Corollary \ref{coro:exp}}
When $\lambda_i \lesssim e^{-\tau i}$, it follows from (\citealt[Lemma B.9]{kpca}) that $$\Ntl \leq \norm{\SgL\Sigma_{PQ}\SgL}_{\op}^{1/2}\mathcal{N}^{1/2}_{1}(\lambda) \lesssim \sqrt{\log\frac{1}{\lambda}}.$$ Thus substituting this in the conditions from Theorem 2 and assuming that $$(N+M)\geq \max\{e^2,\alpha^{-1/2}+\delta^{-1}\},$$ we can write the separation boundary as 
\begin{align*}
\Delta_{N,M} &\gtrsim 
\max\left\{ \left(\frac{\sqrt{2\tilde{\theta}}(\alpha^{-1/2}+\delta^{-1})^{-1}(N+M)}{\sqrt{\log(N+M)}}\right)^{-1},\right.\\
&\qquad\qquad\left.\left(\frac{\delta(N+M)}{\sqrt{\log(N+M)}}\right)^{-\frac{4\tilde{\theta}}{2\tilde{\theta}+1}}, \left(\frac{d_4^{-1}(N+M)}{\log(N+M)}\right)^{-2\tilde{\theta}} \right\},
\end{align*}
which is implied if 
$$\Delta_{N,M} =
\left\{
	\begin{array}{ll}
		c(\alpha,\delta, \theta)\frac{\sqrt{\log(N+M)}}{N+M}, &  \ \ \Tilde{\theta}> \frac{1}{2} \\
		c(\alpha,\delta,\theta)\left(\frac{\log(N+M)}{N+M}\right)^{2\tilde{\theta}}, & \ \ \Tilde{\theta} \le \frac{1}{2}
	\end{array}
\right.$$
for large enough $N+M$, 
where $c(\alpha,\delta,\theta)\gtrsim \max\left\{\sqrt{\frac{1}{2\tilde{\theta}}},1\right\}(\alpha^{-1/2}+\delta^{-2}+d_4^{2\tilde{\theta}})$ and we used that $\tilde{\theta}> \frac{1}{2}$ implies $1 \leq  \min\left\{2\tilde{\theta},\frac{4\tilde{\theta}}{2\tilde{\theta}+1}\right\}$ and that $\tilde{\theta}\le \frac{1}{2}$ implies $2\tilde{\theta}\leq  \min\left\{ 1,\frac{4\tilde{\theta}}{2\tilde{\theta}+1}\right\}.$

On the other hand when $C:=\sup_i\norm{\phi_i}_{\infty} < \infty$, we obtain
\begin{align*}
\Delta_{N,M} &\gtrsim 
\max\left\{ \sqrt{\frac{1}{2\tilde{\theta}}}\left(\frac{(\alpha^{-1/2}+\delta^{-1})^{-1}(N+M)}{\sqrt{\log(N+M)}}\right)^{-1},\right.\\
&\qquad\qquad\left.\frac{1}{2\tilde{\theta}}\left(\frac{\delta(N+M)}{\sqrt{\log(N+M)}}\right)^{-2}, e^{\frac{-2\tilde{\theta}d_5(N+M)}{\log (N+M)}} \right\}
\end{align*}
for some constant $d_5>0$. We can deduce that the condition is reduced to 
$$\Delta_{N,M} = c(\alpha,\delta,\theta)\frac{\sqrt{\log(N+M)}}{N+M},$$
where $c(\alpha,\delta,\theta) \gtrsim \max\left\{\sqrt{\frac{1}{2\tilde{\theta}}},\frac{1}{2\tilde{\theta}},1\right\}(\alpha^{-1/2}+\delta^{-2})$, and we used $e^{\frac{-ax}{\log x}}\leq (ax\sqrt{\log x})^{-1}$ when $x$ is large enough, for $a,x>0$. 

\subsection{Proof of Theorem \ref{thm: permutations typeI}}
Under $H_0$, we have $\stat(X,Y,Z) \stackrel{d}{=} \stat(X^{\pi},Y^{\pi},Z)$ for any $\pi \in \Pi_{n+m}$, i.e., $\stat \stackrel{d}{=} \hat{\eta}^{\pi}_{\lambda}.$ 
Thus, given samples $(X_i)_{i=1}^n$, $(Y_j)_{j=1}^m$ and $(Z_i)_{i=1}^s$, we have 
\begin{align*}
    1-\alpha \leq \frac{1}{D}\sum_{\pi \in \Pi_{n+m}} \II(\hat{\eta}^{\pi}_{\lambda} \leq \qqe).
\end{align*}
Taking expectations on both sides of the above inequality with respect to the samples yields
\begin{align*}
    1-\alpha & \leq \frac{1}{D}\sum_{\pi \in \Pi_{n+m}} \E\II(\hat{\eta}^{\pi}_{\lambda} \leq \qqe)
     = \E\II(\stat \leq \qqe) = P_{H_0}\{\stat \leq \qqe \}. 
\end{align*}
Therefore,
\begin{align*}
    P_{H_0}\{\stat \leq \hat{q}_{1-w\alpha}^{B,\lambda}\} &\geq P_{H_0}\left\{\{\stat \leq q_{1-w\alpha-\Tilde{\alpha}}^{\lambda}\} \cap \{\hat{q}_{1-w\alpha}^{B,\lambda} \geq q_{1-w\alpha-\Tilde{\alpha}}^{\lambda}\}\right\} \\
    & \geq 1-P_{H_0}\{\stat \geq q_{1-w\alpha-\Tilde{\alpha}}^{\lambda}\} - P_{H_0}\{\hat{q}_{1-w\alpha}^{B,\lambda} \leq q_{1-w\alpha-\Tilde{\alpha}}^{\lambda}\} \\
    & \stackrel{(*)}{\geq} 1-w\alpha-\Tilde{\alpha}-(1-w-\Tilde{w})\alpha,
\end{align*}
where we applied Lemma~\ref{lemma:DKW for quantile} in $(*)$, and the result follows by choosing $\Tilde{\alpha} = \Tilde{w}\alpha$.


\subsection{Proof of Theorem \ref{thm: permutations typeII}}
First, we show that for any $(P,Q) \in \PP$, the following holds under the conditions of Theorem \ref{thm: permutations typeII}:  
\begin{equation}P_{H_1}\left\{\stat \geq q_{1-\alpha}^{\lambda} \right\} \geq 1-4\delta. \label{Eq:perm1}\end{equation}
To this end, let $\M = \SLh^{-1/2}\SL^{1/2}$ , $$\gamma_1 =\frac{1}{\sqrt{\delta}} \left(\frac{\sqrt{\Cl} \norm{\U}_{\Lp}+\Ntl}{n+m}+\frac{\Cl^{1/4}\norm{\U}_{\Lp}^{3/2}+\norm{\U}_{\Lp}}{\sqrt{n+m}}\right),$$
where $\Cl$ as defined in Lemma \ref{lemma: bound hs and op}. Then Lemma \ref{Lemma: bounding expectations} implies $$\tilde{C}\norm{\M}_{\op}^2\gamma_1\geq\sqrt{\frac{\text{Var}(\stat|(Z_i)_{i=1}^s)}{\delta}}$$ for some constant $\tilde{C}>0$. By Lemma \ref{thm: general Type2 error}, if
\begin{equation}
P_{H_1}\left\{\qqe \geq \zeta - \tilde{C}\norm{\M}_{\op}^2\gamma_1\right\} \leq 2\delta, \label{Eq:cond}
\end{equation}
then we obtain \eqref{Eq:perm1}. Therefore, it remains to verify \eqref{Eq:cond} which we do below.  
Define \begin{align*}\gamma &= \frac{\norm{\M}_{\op}^2\log\frac{1}{\alpha}}{\sqrt{\delta}(n+m)}\left(\sqrt{\Cl}\norm{u}_{\Lp}+\Ntl+\Cl^{1/4}\norm{u}^{3/2}_{\Lp}+\norm{u}_{\Lp}\right)\\
&\qquad\qquad+\frac{\zeta\log\frac{1}{\alpha}}{\sqrt{\delta}(n+m)},\end{align*}
and $$\gamma_2 = \frac{\log\frac{1}{\alpha}}{\sqrt{\delta}(n+m)}\left(\sqrt{\Cl}\norm{u}_{\Lp}+\Ntl+\Cl^{1/4}\norm{u}^{3/2}_{\Lp}+\norm{u}_{\Lp}\right).$$ Thus we have $$\gamma = \norm{\M}_{\op}^2 \gamma_2 + \frac{\zeta\log\frac{1}{\alpha}}{\sqrt{\delta}(n+m)}.$$
Then it follows from Lemma \ref{lemma:bound quantile} that there exists a constant $C_5>0$ such that  
\begin{align*}
P_{H_1}\{\qqe \geq C_5\gamma\} \leq \delta.
\end{align*}
Let $c_2 =B_3C_4 \Cs^{-1}$. Then we have
\begin{align*}
    &P_{H_1}\{\qqe \leq \zeta - \norm{\M}_{\op}^2\tilde{C}\gamma_1\} \\
    &\geq P_{H_1}\left\{\{C_5\gamma \leq \zeta - \norm{\M}_{\op}^2\tilde{C}\gamma_1\} \cap \{\qqe \leq C_5\gamma\}\right\} \\
    & = P_{H_1}\left\{\left\{\norm{\M}_{\op}^2 C_5\gamma_2 + \frac{C_5\zeta\log\frac{1}{\alpha}}{\sqrt{\delta}(n+m)} \leq \zeta - \norm{\M}_{\op}^2\tilde{C}\gamma_1\right\}\right.\\
    &\qquad\qquad\qquad\left.\cap \{\qqe \leq C_5\gamma\}\right\}\\
    &\geq  1- P_{H_1}\left\{\norm{\M}_{\op}^2(\tilde{C}\gamma_1+C_5\gamma_2)\geq \zeta\left(1-\frac{C_5\log\frac{1}{\alpha}}{\sqrt{\delta}(n+m)}\right)\right\}\\
    &\qquad\qquad\qquad- P_{H_1}\{\qqe \geq C_5\gamma\} \\
    &\stackrel{(*)} {\geq}  1- P_{H_1}\left\{\frac{\norm{\M}^2_{\op}\norm{\M^{-1}}^{2}_{\op}(\tilde{C}\gamma_1+C_5\gamma_2)}{c_2 \norm{u}_{\Lp}^2 }\geq \frac{1}{2}\right\} - \delta\\
    &= P_{H_1}\left\{\frac{\norm{\M}^2_{\op}\norm{\M^{-1}}^{2}_{\op}(\tilde{C}\gamma_1+C_5\gamma_2)}{c_2 \norm{u}_{\Lp}^2 }\leq \frac{1}{2}\right\} -\delta \\
    \end{align*}
    \begin{align*}
    & \stackrel{(\dag)}{\geq}P_{H_1}\left\{\norm{\M}^2_{\op}\norm{\M^{-1}}^{2}_{\op}\le 3\right\}-\delta\\
    & \geq P_{H_1}\left\{\left\{\norm{\M^{-1}}^{2}_{\op} \leq \frac{3}{2}\right\}  \cap \left\{\norm{\M}^{2}_{\op} \leq 2\right\}\right\}-\delta\\
    & \geq 1-P_{H_1}\left\{\norm{\M^{-1}}^{2}_{\op} \geq \frac{3}{2}\right\} - P_{H_1}\left\{\norm{\M}^{2}_{\op} \geq 2\right\}-\delta \\
    & \stackrel{(\ddag)}{\geq} 1-2\delta,
\end{align*}
where in $(*)$ we assume 
$(n+m) \geq \frac{2C_5\log\frac{2}{\alpha}}{\sqrt{\delta}}$, then it follows by using $$\zeta \geq  c_2 \norm{\M^{-1}}^{-2}_{\op} \norm{u}_{\Lp}^2,$$ which is obtained by combining Lemma \ref{lemma: denominator lower bound} with Lemma \ref{lemma: bounds for eta} under the assumptions of $u \in \text{\range}(\T^{\theta})$, and \eqref{Eq:lower}. Note that $u \in \text{\range}(\T^{\theta})$ is guaranteed since $(P,Q) \in \PP$ and \eqref{Eq:verify-1} guarantees \eqref{Eq:lower} as discussed in the proof of Theorem \ref{Eq:type-2}. $(\dag)$ follows when
\begin{equation} 
 \norm{u}_{\Lp}^2 \geq \frac{6(\tilde{C}\gamma_1+C_5\gamma_2)}{c_2}. \label{Eq:perm-verify-1}   
\end{equation}
$(\ddag)$ follows from \cite[Lemma B.2\emph{(ii)}]{kpca}, under the condition \eqref{Eq:verify-4}. When $C:=\sup_i\norm{\phi_i}_{\infty} < \infty$, $(\ddag)$ follows from Lemma \ref{lem: bound M operator} by replacing \eqref{Eq:verify-4} with \eqref{Eq:verify-5}.

As in the proof of Theorem~\ref{thm:Type II}, it can be shown that \eqref{Eq:lower}, \eqref{Eq:verify-4} and \eqref{Eq:verify-5} are satisfied under the assumptions made in the statement of Theorem~\ref{thm: permutations typeII}. 
It can also be verified that \eqref{Eq:perm-verify-1} is implied if $\Delta_{N,M} \geq \frac{r_1\Cl(\log(1/\alpha))^2}{\delta^2(N+M)^2}$ and $\Delta_{N,M}\geq\frac{r_2\Ntl\log(1/\alpha)}{\delta(N+M)}$ for some constants $r_1,r_2>0$. Finally, we have
\begin{align*}
    P_{H_1}\{\stat \geq \hat{q}_{1-w\alpha}^{B,\lambda}\} &\geq P_{H_1}\left\{\{\stat \geq q_{1-w\alpha+\Tilde{\alpha}}^{\lambda}\} \cap \{\hat{q}_{1-w\alpha}^{B,\lambda} \leq q_{1-w\alpha+\Tilde{\alpha}}^{\lambda}\}\right\} \\
    & \geq 1-P_{H_1}\left\{\stat \leq q_{1-w\alpha+\Tilde{\alpha}}^{\lambda}\right\} - P_{H_1}\left\{\hat{q}_{1-w\alpha}^{B,\lambda} > q_{1-w\alpha+\Tilde{\alpha}}^{\lambda}\right\} \\
    & \stackrel{(*)}{\geq} 1-4\delta-\delta = 1-5\delta,
\end{align*}
 where in $(*)$ we use \eqref{Eq:perm1} and Lemma  \ref{lemma:DKW for quantile} by setting  $\Tilde{\alpha}=\Tilde{w}\alpha$, for $0<\Tilde{w}<w <1$. Then, the desired result follows by taking infimum over $(P,Q) \in \PP$.

\subsection{Proof of Corollary \ref{coro:poly:perm}}
The proof is similar to that of Corollary~\ref{coro:poly}. Since $\lambda_i \lesssim i^{-\beta}$, we have $\Ntl 
\lesssim \lambda^{-1/2\beta}.$ By using this bound in the conditions of Theorem \ref{thm: permutations typeII}, we that the conditions on $\Delta_{N,M}$ hold if 
\begin{align}
\Delta_{N,M} &\gtrsim  \max\left\{\left(\frac{\delta(N+M)}{\log(1/\alpha)}\right)^{-\frac{4\Tilde{\theta}\beta}{4\Tilde{\theta}\beta+1}}, \left(\frac{\delta(N+M)}{\log(1/\alpha)}\right)^{-\frac{8\Tilde{\theta}\beta}{4\Tilde{\theta}\beta+2\beta+1}},\right.\label{Eq:coropoly1}\\ &\qquad\qquad\qquad\left.\left(\frac{d_4^{-1}(N+M)}{\log(N+M)}\right)^{-2\tilde{\theta}} \right\},\nonumber \end{align}
where $d_4>0$ is a constant. By exactly using the same arguments as in the proof of Corollary~\ref{coro:poly}, it is easy to verify that the above condition on $\Delta_{N,M}$ is implied if 
$$\Delta_{N,M} =
\left\{
	\begin{array}{ll}
		c(\alpha,\delta,\theta)\left(N+M\right)^{\frac{-4\tilde{\theta}\beta}{4\Tilde{\theta}\beta+1}},  &\ \ \Tilde{\theta}> \frac{1}{2}-\frac{1}{4\beta} \\
		c(\alpha,\delta,\theta)\left(\frac{N+M}{\log(N+M)}\right)^{-2\tilde{\theta}}, & \ \ \Tilde{\theta} \le \frac{1}{2}-\frac{1}{4\beta}
	\end{array}
\right.,$$
where $c(\alpha,\delta,\theta)\gtrsim\max\{\delta^{-2}(\log 1/\alpha)^2,d_4^{2\tilde{\theta}}\}$.

On the other hand when  $C:=\sup_i\norm{\phi_i}_{\infty} < \infty$, we obtain the corresponding condition as
\begin{align}\Delta_{N,M} &\gtrsim \max\left\{\left(\frac{\delta(N+M)}{\log(1/\alpha)}\right)^{-\frac{4\Tilde{\theta}\beta}{4\Tilde{\theta}\beta+1}},\left(\frac{\delta(N+M)}{\log(1/\alpha)}\right)^{-\frac{4\Tilde{\theta}\beta}{2\Tilde{\theta}\beta+1}},\right.\label{Eq:coropoly2}\\
&\qquad\qquad\qquad \left.d_5^{2\tilde{\theta}\beta}\left(\frac{N+M}{\log(N+M)}\right)^{-2\tilde{\theta}\beta} \right\}, \nonumber\end{align}
for some $d_5>0$, which in turn is implied for $N+M$ large enough, if 
$$\Delta_{N,M} =
\left\{
	\begin{array}{ll}
		c(\alpha,\delta,\theta,\beta)\left(N+M\right)^{\frac{-4\tilde{\theta}\beta}{4\Tilde{\theta}\beta+1}},  &  \ \  \Tilde{\theta}> \frac{1}{4\beta} \\
		c(\alpha,\delta,\theta,\beta)\left(\frac{N+M}{\log(N+M)}\right)^{-2\tilde{\theta}
		\beta}, & \ \  \Tilde{\theta} \le \frac{1}{4\beta}
	\end{array}
\right.,$$
where $c(\alpha,\delta,\theta,\beta)\gtrsim \max\{\delta^{-2}(\log 1/\alpha)^2,d_5^{2\tilde{\theta}\beta}\}.$

\subsection{Proof of Corollary \ref{coro:exp:perm}}
The proof is similar to that of Corollary~\ref{coro:exp}. When $\lambda_i \lesssim e^{-\tau i}$, we have $\Ntl \lesssim \sqrt{\log\frac{1}{\lambda}}$. Thus substituting this in the conditions from Theorem \ref{thm: permutations typeII} and assuming that  $(N+M)\geq \max\{e^2,\delta^{-1}(\log 1/\alpha)\}$, we can write the separation boundary as 
\begin{eqnarray}\Delta_{N,M} &{}\gtrsim {}&\max\left\{ \left(\frac{\sqrt{2\tilde{\theta}}\delta(N+M)}{\log(1/\alpha)\sqrt{\log(N+M)}}\right)^{-1},\right.\nonumber\\
&{}{}&\qquad\left.\left(\frac{\delta(N+M)\log(1/\alpha)^{-1}}{\sqrt{\log(N+M)}}\right)^{-\frac{4\tilde{\theta}}{2\tilde{\theta}+1}},\left(\frac{d_6^{-1}(N+M)}{\log(N+M)}\right)^{-2\tilde{\theta}} \right\}, \label{Eq:coroexp1}\end{eqnarray}
where $d_6>0$ is a constant. This condition in turn is implied if 
$$\Delta_{N,M} =
\left\{
	\begin{array}{ll}
		c(\alpha,\delta, \theta)\frac{\sqrt{\log(N+M)}}{N+M}, &  \ \ \Tilde{\theta}> \frac{1}{2} \\
		c(\alpha,\delta,\theta)\left(\frac{\log(N+M)}{N+M}\right)^{2\tilde{\theta}}, & \ \ \Tilde{\theta} \le \frac{1}{2}
	\end{array}
\right.$$
for large enough $N+M$, where $$c(\alpha,\delta,\theta) \gtrsim \max\left\{\sqrt{\frac{1}{2\tilde{\theta}}},1\right\}\max\{\delta^{-2}(\log 1/\alpha)^2,d_5^{2\tilde{\theta}}\},$$ and we used that $\tilde{\theta}\geq \frac{1}{2}$ implies that $1 \leq  \min\left\{2\tilde{\theta},\frac{4\tilde{\theta}}{2\tilde{\theta}+1}\right\}$ and that $\tilde{\theta}< \frac{1}{2}$ implies that $2\tilde{\theta}\leq  \min\left\{ 1,\frac{4\tilde{\theta}}{2\tilde{\theta}+1}\right\}.$
On the other hand when $C:=\sup_i\norm{\phi_i}_{\infty} < \infty$, we obtain
\begin{align*}
    \Delta_{N,M} &\gtrsim \max\left\{\left(\frac{\sqrt{2\tilde{\theta}}\delta(N+M)}{\log(1/\alpha)\sqrt{\log(N+M)}}\right)^{-1},\right.
    \end{align*}
    
\begin{align*}
    &\qquad\qquad\left.\frac{1}{2\tilde{\theta}}\left(\frac{\delta(N+M)}{\log(1/\alpha)\sqrt{\log(N+M)}}\right)^{-2}, e^{\frac{-2\tilde{\theta}d_7(N+M)}{\log (N+M)}} \right\}
\end{align*}    
for some constant $d_7>0$.  We can deduce that the condition is reduced to
$$\Delta_{N,M} = c(\alpha,\delta,\theta)\frac{\sqrt{\log(N+M)}}{N+M},$$
where $c(\alpha,\delta,\theta) \gtrsim \max\left\{\sqrt{\frac{1}{2\tilde{\theta}}},\frac{1}{2\tilde{\theta}},1\right\}(\delta^{-2}(\log 1/\alpha)^2)$, and we used $e^{\frac{-ax}{\log x}}\leq (ax\sqrt{\log x})^{-1}$ when $x$ is large enough, for $a,x>0$.

\subsection{Proof of Theorem~\ref{thm:perm adp typeI}}
The proof follows from Lemma~\ref{lemma:adaptation} and Theorem~\ref{thm: permutations typeI} by using $\frac{\alpha}{\cd}$ in the place of $\alpha$.

\subsection{Proof Theorem \ref{thm: perm adp typeII}}
Using $\hat{q}_{1-\frac{w\alpha}{\cd}}^{B,\lambda}$ as the threshold, the same steps as in the proof of Theorem \ref{thm: permutations typeII} will follow, with the only difference being $\alpha$ replaced by $\frac{\alpha}{\cd}$. Of course, this leads to an extra factor of $\log\cd$ in the expression of $\gamma_2$ in condition \eqref{Eq:perm-verify-1}, which will show up in the expression for the separation boundary (i.e., there will a factor of $\log\frac{\cd}{\alpha}$ instead of $\log\frac{1}{\alpha}$). Observe that for all cases of Theorem \ref{thm: perm adp typeII}, we have $\cd = 1+\log_2\frac{\lambda_U}{\lambda_L}\lesssim \log(N+M).$

For the case of $\lambda_i \lesssim i^{-\beta}$, we can deduce from the proof of Corollary~\ref{coro:poly:perm} (see~\eqref{Eq:coropoly1}) that when $\lambda = d_3^{-1/2\tilde{\theta}} \Delta_{N,M}^{1/2\Tilde{\theta}}$ for some $d_3>0$, then 
\begin{align*}
    P_{H_1}\left\{\stat \geq \hat{q}_{1-\frac{w\alpha}{\cd}}^{B,\lambda} \right\} \geq 1-5\delta,
\end{align*}
and the condition on the separation boundary becomes
\begin{align*}
\Delta_{N,M} &\gtrsim \max\left\{\left(\frac{\delta(N+M)}{\log\left(\frac{\log(N+M)}{\alpha}\right)}\right)^{-\frac{4\Tilde{\theta}\beta}{4\Tilde{\theta}\beta+1}}, \left(\frac{\delta(N+M)}{\log\left(\frac{\log(N+M)}{\alpha}\right)}\right)^{-\frac{8\Tilde{\theta}\beta}{4\Tilde{\theta}\beta+2\beta+1}},\right.
\end{align*}
\begin{align*}
&\qquad\qquad\qquad\left.\left(\frac{d_5^{-1}(N+M)}{\log(N+M)}\right)^{-2\tilde{\theta}} \right\},
\end{align*}
which in turn is implied if
$$\Delta_{N,M} = c(\alpha,\delta,\theta)\max\left\{\left(\frac{N+M}{\log\log(N+M)}\right)^{-\frac{4\Tilde{\theta}\beta}{4\Tilde{\theta}\beta+1}}, \left(\frac{N+M}{\log(N+M)}\right)^{-2\tilde{\theta}} \right\},$$
where $c(\alpha,\delta,\theta)\gtrsim\max\{\delta^{-2}(\log 1/\alpha)^2,d_4^{2\tilde{\theta}}\}$ for some $d_4>0$ and we used that $\log\frac{x}{\alpha} \leq \log\frac{1}{\alpha}\log x$ for large enough $x \in \R$.

Note that the optimal choice of $\lambda$ is given by \begin{align*}\lambda=\lambda^*&:=d_3^{-1/2\tilde{\theta}}c(\alpha,\delta,\theta)^{1/2\tilde{\theta}} \max\left\{\left(\frac{N+M}{\log\log(N+M)}\right)^{-\frac{2\beta}{4\Tilde{\theta}\beta+1}},\left(\frac{N+M}{\log(N+M)}\right)^{-1} \right\}.\end{align*} Observe that the constant term $d_3^{-1/2\tilde{\theta}}c(\alpha,\delta,\theta)^{1/2\tilde{\theta}}$ can be expressed as $r_1^{1/2\tilde{\theta}}$ for some constant $r_1$ that depends only on $\delta$ and $\alpha$. If $r_1\leq 1$, we can bound $r_1^{1/2\tilde{\theta}}$ as $r_1^{1/2\theta_l}\leq r_1^{1/2\tilde{\theta}}\leq r_1^{1/2\xi}$, and as $r_1^{1/2\xi}\leq r_1^{1/2\tilde{\theta}}\leq r_1^{1/2\theta_l}$ when $r>1 $. Therefore, for any $\theta$ and $\beta$, the optimal lambda can be bounded as $r_2\left(\frac{N+M}{\log(N+M)} \right)^{-1} \leq \lambda \leq r_3\left(\frac{N+M}{\log(N+M)} \right)^\frac{-2}{4\Tilde{\xi}+1}$, where $r_2,r_3$ are constants that depend only on $\delta$, $\alpha$, $\xi$ and $\theta_l$, and $\tilde{\xi}=\max\{\xi,1/4\}$. 

$\blacklozenge$ Define $v^* := \sup\{x \in \Lambda: x \leq \lambda^*\}$. From the definition of $\Lambda$, it is easy to see that $\lambda_L\leq \lambda^* \leq \lambda_U$ and $\frac{\lambda^*}{2}\leq v^* \leq \lambda^*$. Thus $v^*\in\Lambda$ is an optimal choice of $\lambda$ that will yield to the same form of the separation boundary up to constants. Therefore, by Lemma~\ref{lemma:adaptation}, for any $\theta$ and any $(P,Q)$ in $\PP$, we have 
\begin{align*}
    P_{H_1}\left\{\bigcup_{\lambda \in \Lambda}\stat \geq \hat{q}_{1-\frac{w\alpha}{\cd}}^{B,\lambda} \right\} \geq 1-5\delta.
\end{align*}
Thus the desired result holds by taking the infimum over $(P,Q) \in \PP$ and $\theta$. $\spadesuit$

When $\lambda_i \lesssim i^{-\beta}$ and $C:=\sup_i\norm{\phi_i}_{\infty} < \infty$, then using \eqref{Eq:coropoly2}, the conditions on the separation boundary becomes
\begin{align*}
\Delta_{N,M} &= c(\alpha,\delta,\theta,\beta)\max\left\{\left(\frac{N+M}{\log\log(N+M)}\right)^{-\frac{4\Tilde{\theta}\beta}{4\Tilde{\theta}\beta+1}},\left(\frac{N+M}{\log(N+M)}\right)^{-2\tilde{\theta}\beta} \right\},
\end{align*}
where $c(\alpha,\delta,\theta,\beta)\gtrsim \max\{\delta^{-2}(\log 1/\alpha)^2,d_5^{2\tilde{\theta}\beta}\}$. This yields the optimal $\lambda$ to be 
\begin{align*}
\lambda^*&:=d_3^{-1/2\tilde{\theta}}c(\alpha,\delta,\theta,\beta)^{1/2\tilde{\theta}} \max\left\{\left(\frac{N+M}{\log\log(N+M)}\right)^{-\frac{2\beta}{4\Tilde{\theta}\beta+1}},\left(\frac{N+M}{\log(N+M)}\right)^{-\beta} \right\}.
\end{align*}
Using the similar argument as in the previous case, we can deduce that for any $\theta$ and $\beta$, we have $r_4\left(\frac{N+M}{\log(N+M)} \right)^{-\beta_u} \leq \lambda^* \leq r_5\left(\frac{N+M}{\log(N+M)} \right)^\frac{-2}{4\Tilde{\xi}+1}$, where $r_4,r_5$ are constants that depend only on $\delta$, $\alpha$, $\xi$, $\theta_l$, and $\beta_u$. The claim therefore follows by using the argument mentioned between $\blacklozenge$ and $\spadesuit$.

For the case $\lambda_i \lesssim e^{-\tau i}$, $\tau>0$, the condition from \eqref{Eq:coroexp1} becomes 
\begin{align*}
\Delta_{N,M} &= c(\alpha,\delta,\theta)\max\left\{ \left(\frac{N+M}{\sqrt{\log(N+M)}\log\log(N+M)}\right)^{-1},\left(\frac{N+M}{\log(N+M)}\right)^{-2\tilde{\theta}} \right\},
\end{align*}
where $c(\alpha,\delta,\theta) \gtrsim \max\left\{\sqrt{\frac{1}{2\tilde{\theta}}},1\right\}\max\left\{\delta^{-2}(\log 1/\alpha)^2,d_4^{2\tilde{\theta}}\right\}.$ Thus 
\begin{align*}
\lambda^* &= d_3^{-1/2\tilde{\theta}}c(\alpha,\delta,\theta)^{1/2\tilde{\theta}}\max\left\{ \left(\frac{N+M}{\sqrt{\log(N+M)}(\log\log(N+M))}\right)^{-1/2\tilde{\theta}},\right.\\
&\qquad\qquad\qquad\left.\left(\frac{N+M}{\log(N+M)}\right)^{-1} \right\},
\end{align*}
which can be bounded as $r_6\left(\frac{N+M}{\log(N+M)} \right)^{-1}  \leq \lambda \leq r_7\left(\frac{N+M}{\log(N+M)} \right)^{-1/2\xi}$ for any $\theta$, where $r_6,r_7$ are constants that depend only on $\delta$, $\alpha$, $\xi$ and $\theta_l.$ Furthermore  when $C:=\sup_i\norm{\phi_i}_{\infty} < \infty$, the condition on the separation boundary becomes 
$$\Delta_{N,M} = c(\alpha,\delta,\theta)\frac{\sqrt{\log(N+M)}\log\log(N+M)}{N+M},$$
where $c(\alpha,\delta,\theta) \gtrsim \max\left\{\sqrt{\frac{1}{2\tilde{\theta}}},\frac{1}{2\tilde{\theta}},1\right\}\delta^{-2}(\log 1/\alpha)^2.$ Thus $$\lambda^* = d_3^{-1/2\tilde{\theta}}c(\alpha,\delta,\theta)^{1/2\tilde{\theta}} \left(\frac{N+M}{\sqrt{\log(N+M)}\log\log(N+M)}\right)^{-1/2\tilde{\theta}},$$ which can be bounded as 
\begin{align*}
&r_8\left(\frac{N+M}{\sqrt{\log(N+M)}\log\log(N+M)} \right)^{-1/2\theta_l}  \leq \lambda^* \\
&\qquad\qquad\leq r_9\left(\frac{N+M}{\sqrt{\log(N+M)}\log\log(N+M)} \right)^{-1/2\xi}
\end{align*}
for any $\theta$, where $r_8,r_9$ are constants that depend only on $\delta$, $\alpha$, $\xi$ and $\theta_l.$ The claim, therefore, follows by using the same argument as mentioned in the polynomial decay case.

 \section*{Acknowledgments}
The authors thank the reviewers for their valuable comments and constructive feedback, which helped to significantly improve the paper. OH and BKS are partially supported by National Science Foundation (NSF) CAREER
award DMS--1945396. BL is supported by NSF grant DMS--2210775.

\bibliographystyle{plainnat} 
\bibliography{references} 

\begin{thebibliography}{51}
\providecommand{\natexlab}[1]{#1}
\providecommand{\url}[1]{\texttt{#1}}
\expandafter\ifx\csname urlstyle\endcsname\relax
  \providecommand{\doi}[1]{doi: #1}\else
  \providecommand{\doi}{doi: \begingroup \urlstyle{rm}\Url}\fi

\bibitem[Adams and Fournier(2003)]{Adams}
R.~A. Adams and J.~J.~F. Fournier.
\newblock \emph{Sobolev Spaces}.
\newblock Academic Press, 2003.

\bibitem[Albert et~al.(2022)Albert, Laurent, Marrel, and Meynaoui]{Albert}
M.~Albert, B.~Laurent, A.~Marrel, and A.~Meynaoui.
\newblock {Adaptive test of independence based on HSIC measures}.
\newblock \emph{The Annals of Statistics}, 50\penalty0 (2):\penalty0 858 -- 879, 2022.

\bibitem[Aronszajn(1950)]{Aronszajn}
N.~Aronszajn.
\newblock Theory of reproducing kernels.
\newblock \emph{Trans. Amer. Math. Soc.}, pages 68:337--404, 1950.

\bibitem[Balasubramanian et~al.(2021)Balasubramanian, Li, and Yuan]{Krishna}
K.~Balasubramanian, T.~Li, and M.~Yuan.
\newblock On the optimality of kernel-embedding based goodness-of-fit tests.
\newblock \emph{Journal of Machine Learning Research}, 22\penalty0 (1):\penalty0 1--45, 2021.

\bibitem[Bauer et~al.(2007)Bauer, Pereverzev, and Rosasco]{BAUER200752}
F.~Bauer, S.~Pereverzev, and L.~Rosasco.
\newblock On regularization algorithms in learning theory.
\newblock \emph{Journal of Complexity}, 23\penalty0 (1):\penalty0 52--72, 2007.

\bibitem[Burnashev(1979)]{Burnashev}
M.~V. Burnashev.
\newblock {On the minimax detection of an inaccurately known signal in a white {G}aussian noise background}.
\newblock \emph{Theory of Probability \& Its Applications}, 24\penalty0 (1):\penalty0 107--119, 1979.

\bibitem[Caponnetto and Vito(2007)]{Caponnetto-07}
A.~Caponnetto and E.~De Vito.
\newblock Optimal rates for regularized least-squares algorithm.
\newblock \emph{Foundations of Computational Mathematics}, 7:\penalty0 331--368, 2007.

\bibitem[Cucker and Zhou(2007)]{Cucker}
F.~Cucker and D.~X. Zhou.
\newblock \emph{Learning Theory: An Approximation Theory Viewpoint}.
\newblock Cambridge University Press, Cambridge, UK, 2007.

\bibitem[Dinculeanu(2000)]{Dincu}
N.~Dinculeanu.
\newblock \emph{Vector Integration and Stochastic Integration in Banach Spaces}.
\newblock John-Wiley \& Sons, Inc., 2000.

\bibitem[Drineas and Mahoney(2005)]{Drineas-05}
P.~Drineas and M.~W. Mahoney.
\newblock On the {N}ystr{\"{o}}m method for approximating a {G}ram matrix for improved kernel-based learning.
\newblock \emph{Journal of Machine Learning Research}, 6:\penalty0 2153--2175, December 2005.

\bibitem[Dvoretzky et~al.(1956)Dvoretzky, Kiefer, and Wolfowitz]{DKW}
A.~Dvoretzky, J.~Kiefer, and J.~Wolfowitz.
\newblock Asymptotic minimax character of the sample distribution function and of the classical multinomial estimator.
\newblock \emph{The Annals of Mathematical Statistics}, 27\penalty0 (3):\penalty0 642--669, 1956.

\bibitem[Engl et~al.(1996)Engl, Hanke, and Neubauer]{Engl.et.al}
H.~W. Engl, M.~Hanke, and A.~Neubauer.
\newblock \emph{Regularization of Inverse Problems}.
\newblock Kluwer Academic Publishers, Dordrecht, The Netherlands, 1996.

\bibitem[Fasano and Franceschini(1987)]{Fasano}
G.~Fasano and A.~Franceschini.
\newblock {A multidimensional version of the Kolmogorov–Smirnov test}.
\newblock \emph{Monthly Notices of the Royal Astronomical Society}, 225\penalty0 (1):\penalty0 155--170, 03 1987.
\newblock ISSN 0035-8711.

\bibitem[G{{\"o}}nen and Alpaydin(2011)]{gonen11a}
M.~G{{\"o}}nen and E.~Alpaydin.
\newblock Multiple kernel learning algorithms.
\newblock \emph{Journal of Machine Learning Research}, 12\penalty0 (64):\penalty0 2211--2268, 2011.

\bibitem[Gretton et~al.(2006)Gretton, Borgwardt, Rasch, Sch{\"{o}}lkopf, and Smola]{NipsGretton}
A.~Gretton, K.~Borgwardt, M.~Rasch, B.~Sch{\"{o}}lkopf, and A.~Smola.
\newblock A kernel method for the two-sample problem.
\newblock In B.~Sch\"{o}lkopf, J.~Platt, and T.~Hoffman, editors, \emph{Advances in Neural Information Processing Systems}, volume~19, pages 513--520. MIT Press, 2006.

\bibitem[Gretton et~al.(2009)Gretton, Fukumizu, Harchaoui, and Sriperumbudur]{FastKernel}
A.~Gretton, K.~Fukumizu, Z.~Harchaoui, and B.~K. Sriperumbudur.
\newblock A fast, consistent kernel two-sample test.
\newblock In Y.~Bengio, D.~Schuurmans, J.~Lafferty, C.~Williams, and A.~Culotta, editors, \emph{Advances in Neural Information Processing Systems}, volume~22. Curran Associates, Inc., 2009.

\bibitem[Gretton et~al.(2012)Gretton, Borgwardt, Rasch, Sch{{\"o}}lkopf, and Smola]{gretton12a}
A.~Gretton, K.~M. Borgwardt, M.~J. Rasch, B.~Sch{{\"o}}lkopf, and A.~Smola.
\newblock A kernel two-sample test.
\newblock \emph{Journal of Machine Learning Research}, 13\penalty0 (25):\penalty0 723--773, 2012.

\bibitem[Harchaoui et~al.(2007)Harchaoui, Bach, and Moulines]{Harchaoui}
Z.~Harchaoui, F.~R. Bach, and E.~Moulines.
\newblock Testing for homogeneity with kernel fisher discriminant analysis.
\newblock In J.~Platt, D.~Koller, Y.~Singer, and S.~Roweis, editors, \emph{Advances in Neural Information Processing Systems}, volume~20. Curran Associates, Inc., 2007.

\bibitem[Hoeffding(1992)]{Hoeffding}
W.~Hoeffding.
\newblock A class of statistics with asymptotically normal distribution.
\newblock \emph{In Breakthroughs in Statistics}, pages 308--334, 1992.

\bibitem[Ingster(2000)]{Ingster2000}
Y~Ingster.
\newblock Adaptive chi-square tests.
\newblock \emph{Journal of Mathematical Sciences}, 99:\penalty0 1110--1119, 04 2000.

\bibitem[Ingster(1987)]{Ingester1}
Y.~I. Ingster.
\newblock Minimax testing of nonparametric hypotheses on a distribution density in the ${L}_p$ metrics.
\newblock \emph{Theory of Probability \& Its Applications}, 31\penalty0 (2):\penalty0 333--337, 1987.

\bibitem[Ingster(1993)]{Ingester2}
Y.~I. Ingster.
\newblock Asymptotically minimax hypothesis testing for nonparametric alternatives i, ii, iii.
\newblock \emph{Mathematical Methods of Statistics}, 2\penalty0 (2):\penalty0 85--114, 1993.

\bibitem[Kim et~al.(2022)Kim, Balakrishnan, and Wasserman]{permutations}
I.~Kim, S.~Balakrishnan, and L.~Wasserman.
\newblock {Minimax optimality of permutation tests}.
\newblock \emph{The Annals of Statistics}, 50\penalty0 (1):\penalty0 225--251, 2022.

\bibitem[Le~Cam(1986)]{LeCam}
L.~Le~Cam.
\newblock \emph{Asymptotic Methods In Statistical Decision Theory}.
\newblock Springer, 1986.

\bibitem[LeCun et~al.(2010)LeCun, Cortes, and Burges]{Mnist}
Y.~LeCun, C.~Cortes, and C.~Burges.
\newblock {MNIST} handwritten digit database. {AT} \&{T} {L}abs, 2010.

\bibitem[Lehmann and Romano(2006)]{lehmann}
E.~L. Lehmann and J.~P. Romano.
\newblock \emph{Testing Statistical Hypotheses}.
\newblock Springer Science \& Business Media, 2006.

\bibitem[Li and Yuan(2019)]{MingYuan}
T.~Li and M.~Yuan.
\newblock On the optimality of {G}aussian kernel based nonparametric tests against smooth alternatives.
\newblock 2019.
\newblock https://arxiv.org/pdf/1909.03302.pdf.

\bibitem[Massart(1990)]{Massart}
P.~Massart.
\newblock {The tight constant in the Dvoretzky-Kiefer-Wolfowitz Inequality}.
\newblock \emph{The Annals of Probability}, 18\penalty0 (3):\penalty0 1269--1283, 1990.

\bibitem[Mendelson and Neeman(2010)]{Mendelson-10}
S.~Mendelson and J.~Neeman.
\newblock Regularization in kernel learning.
\newblock \emph{The Annals of Statistics}, 38\penalty0 (1):\penalty0 526--565, 2010.

\bibitem[Minh et~al.(2006)Minh, Niyogi, and Yao]{unibound}
H.~Q. Minh, P.~Niyogi, and Y.~Yao.
\newblock Mercer's theorem, feature maps, and smoothing.
\newblock In G{\'a}bor Lugosi and Hans~Ulrich Simon, editors, \emph{Learning Theory}, pages 154--168, Berlin, 2006. Springer.

\bibitem[Muandet et~al.(2017)Muandet, Fukumizu, Sriperumbudur, and Schölkopf]{rkhs}
K.~Muandet, K.~Fukumizu, B.~Sriperumbudur, and B.~Schölkopf.
\newblock Kernel mean embedding of distributions: A review and beyond.
\newblock \emph{Foundations and Trends® in Machine Learning}, 10\penalty0 (1-2):\penalty0 1--141, 2017.

\bibitem[Pesarin and Salmaso(2010)]{Pesarin}
F.~Pesarin and L.~Salmaso.
\newblock \emph{Permutation Tests for Complex Data: Theory, Applications and Software}.
\newblock John Wiley \& Sons, 2010.

\bibitem[Puritz et~al.(2022)Puritz, Ness-Cohn, and Braun]{KS}
C.~Puritz, E.~Ness-Cohn, and R.~Braun.
\newblock fasano.franceschini.test: An implementation of a multidimensional ks test in r, 2022.

\bibitem[Rahimi and Recht(2008)]{Rahimi-08a}
A.~Rahimi and B.~Recht.
\newblock Random features for large-scale kernel machines.
\newblock In J.~C. Platt, D.~Koller, Y.~Singer, and S.~T. Roweis, editors, \emph{Advances in Neural Information Processing Systems 20}, pages 1177--1184. Curran Associates, Inc., 2008.

\bibitem[Reed and Simon(1980)]{Reed}
M.~Reed and B.~Simon.
\newblock \emph{Methods of Modern Mathematical Physics: Functional Analysis I}.
\newblock Academic Press, New York, 1980.

\bibitem[Romano and Wolf(2005)]{Romano}
J.~P. Romano and M.~Wolf.
\newblock Exact and approximate stepdown methods for multiple hypothesis testing.
\newblock \emph{Journal of the American Statistical Association}, 100\penalty0 (469):\penalty0 94--108, 2005.

\bibitem[Schrab et~al.(2021)Schrab, Kim, Albert, Laurent, Guedj, and Gretton]{MMDagg}
A.~Schrab, I.~Kim, M.~Albert, B.~Laurent, B.~Guedj, and A.~Gretton.
\newblock {MMD} aggregated two-sample test.
\newblock 2021.
\newblock https://arxiv.org/pdf/2110.15073.pdf.

\bibitem[Simon-Gabriel and Sch{\"{o}}lkopf(2018)]{Carl}
C.~Simon-Gabriel and B.~Sch{\"{o}}lkopf.
\newblock Kernel distribution embeddings: Universal kernels, characteristic kernels and kernel metrics on distributions.
\newblock \emph{Journal of Machine Learning Research}, 19\penalty0 (44):\penalty0 1--29, 2018.

\bibitem[Smola et~al.(2007)Smola, Gretton, Song, and Sch{\"{o}}lkopf]{Smola}
A.~J. Smola, A.~Gretton, L.~Song, and B.~Sch{\"{o}}lkopf.
\newblock A {H}ilbert space embedding for distributions.
\newblock In Marcus Hutter, Rocco~A. Servedio, and Eiji Takimoto, editors, \emph{Algorithmic Learning Theory}, pages 13--31. Springer-Verlag, Berlin, Germany, 2007.

\bibitem[Sriperumbudur(2016)]{bernouli2016}
B.~K. Sriperumbudur.
\newblock {On the optimal estimation of probability measures in weak and strong topologies}.
\newblock \emph{Bernoulli}, 22\penalty0 (3):\penalty0 1839 -- 1893, 2016.

\bibitem[Sriperumbudur and Sterge(2022)]{kpca}
B.~K. Sriperumbudur and N.~Sterge.
\newblock Approximate kernel {PCA} using random features: {C}omputational vs. statistical trade-off.
\newblock \emph{The Annals of Statistics}, 50\penalty0 (5):\penalty0 2713--2736, 2022.

\bibitem[Sriperumbudur et~al.(2009)Sriperumbudur, Fukumizu, Gretton, Lanckriet, and Sch{\"{o}}lkopf]{classifiability}
B.~K. Sriperumbudur, K.~Fukumizu, A.~Gretton, G.~R.~G. Lanckriet, and B.~Sch{\"{o}}lkopf.
\newblock Kernel choice and classifiability for {RKHS} embeddings of probability distributions.
\newblock In Y.~Bengio, D.~Schuurmans, J.~Lafferty, C.~K.~I. Williams, and A.~Culotta, editors, \emph{Advances in Neural Information Processing Systems 22}, pages 1750--1758, Cambridge, MA, 2009. MIT Press.

\bibitem[Sriperumbudur et~al.(2010)Sriperumbudur, Gretton, Fukumizu, Sch{\"{o}}lkopf, and Lanckriet]{JMLR_metrics}
B.~K. Sriperumbudur, A.~Gretton, K.~Fukumizu, B.~Sch{\"{o}}lkopf, and G.~R.~G. Lanckriet.
\newblock Hilbert space embeddings and metrics on probability measures.
\newblock \emph{Journal of Machine Learning Research}, 11:\penalty0 1517--1561, 2010.

\bibitem[Sriperumbudur et~al.(2011)Sriperumbudur, Fukumizu, and Lanckriet]{JMLR_universal}
B.~K. Sriperumbudur, K.~Fukumizu, and G.~R. Lanckriet.
\newblock Universality, characteristic kernels and {RKHS} embedding of measures.
\newblock \emph{Journal of Machine Learning Research}, 12:\penalty0 2389--2410, 2011.

\bibitem[Steinwart and Christmann(2008)]{svm}
I.~Steinwart and A.~Christmann.
\newblock \emph{Support Vector Machines.}
\newblock Springer, New York, 2008.

\bibitem[Steinwart and Scovel(2012)]{Steinwart2012MercersTO}
I.~Steinwart and C.~Scovel.
\newblock Mercer's theorem on general domains: On the interaction between measures, kernels, and {RKHS}s.
\newblock \emph{Constructive Approximation}, 35:\penalty0 363--417, 2012.

\bibitem[Steinwart et~al.(2006)Steinwart, Hush, and Scovel]{Steinwart2006AnED}
I.~Steinwart, D.~R. Hush, and C.~Scovel.
\newblock An explicit description of the reproducing kernel hilbert spaces of gaussian rbf kernels.
\newblock \emph{IEEE Transactions on Information Theory}, 52:\penalty0 4635--4643, 2006.

\bibitem[Sterge and Sriperumbudur(2022)]{kpca-nystrom}
N.~Sterge and B.~K. Sriperumbudur.
\newblock Statistical optimality and computational efficiency of {N}ystr\"{o}m kernel {PCA}.
\newblock \emph{Journal of Machine Learning Research}, 23\penalty0 (337):\penalty0 1--32, 2022.

\bibitem[Szekely and Rizzo(2004)]{Energy}
G~Szekely and M~Rizzo.
\newblock Testing for equal distributions in high dimension.
\newblock \emph{InterStat}, 5, 11 2004.

\bibitem[Williams and Seeger(2001)]{Williams-01}
C.K.I. Williams and M.~Seeger.
\newblock Using the {N}ystr{\"{o}}m method to speed up kernel machines.
\newblock In V.~Tresp T.~K.~Leen, T. G.~Diettrich, editor, \emph{Advances in Neural Information Processing Systems 13}, pages 682--688, Cambridge, MA, 2001. MIT {P}ress.

\bibitem[Yang et~al.(2017)Yang, Pilanci, and Wainwright]{Yang-17}
Y.~Yang, M.~Pilanci, and M.~J. Wainwright.
\newblock Randomized sketches for kernels: {F}ast and optimal nonparametric regression.
\newblock \emph{Annals of Statistics}, 45\penalty0 (3):\penalty0 991--1023, 2017.

\end{thebibliography}
\setcounter{section}{0}
\renewcommand\thesection{\Alph{section}}
\newtheorem{thmm}{Theorem}[section]
\section{Technical results}
In this section, we collect technical results used to prove the main results of the paper. Unless specified otherwise, the notation used in this section matches that of the main paper.
\begin{appxlem} \label{thm: general Type2 error}
Let $\gamma$ be a function of a random variable $Y$, and define $\zeta=\E[X|Y]$. Suppose for all $\delta>0$, $$P\left\{\zeta\ge \gamma(Y)+\sqrt{\frac{\emph{Var}(X|Y)}{\delta}}\right\}\geq 1-\delta.$$ Then $$P\{X\geq \gamma(Y)\}\geq 1-2\delta.$$
\end{appxlem}
\begin{proof}
Define $\gamma_1=\sqrt{\frac{\text{Var}(X|Y)}{\delta}}$. Consider
\begin{align*}
P\{X \geq \gamma(Y)\} & \geq P\{ \{X \geq \zeta- \gamma_1\} \ \cap \{\gamma(Y) \leq \zeta- \gamma_1\}  \} \\
& \geq 1- P\{X \leq \zeta-\gamma_1\} - P\{\gamma(Y) \geq \zeta- \gamma_1\} \\
&\geq 1- P\{|X-\zeta| \geq \gamma_1\} - P\{\gamma(Y) \geq \zeta- \gamma_1\} \\
& \geq 1-2\delta,
\end{align*}
where in the last step we invoked Chebyshev's inequality: 
$P\{|X-\zeta| \geq \gamma_1\} \leq \delta.$
\end{proof}

\begin{appxlem}\label{lemma:format of statistic}
Define $a(x)=g_{\lambda}^{1/2}(\hat{\Sigma}_{PQ})(\kk(\cdot,x)-\mu)$ where $\mu$ is an arbitrary function in $\h$. Then $\stat$ defined in \eqref{stat_def} can be written as
\begin{eqnarray*}
    \stat&{}={}& \frac{1}{n(n-1)} \sum_{i\neq j} \inner{ a(X_i)}{a(X_j)}_{\h}  + \frac{1}{m(m-1)}\sum_{i\neq j} \inner{ a(Y_i)}{a(Y_j)}_{\h}\\
    &&\qquad\qquad- \frac{2}{nm}\sum_{i,j} \inner{a(X_i)}{a(Y_j)}_{\h}.
\end{eqnarray*}
\end{appxlem}
\begin{proof}
The proof follows by using $a(x)+\gShh\mu$ for $\gShh\kk(\cdot,x)$ in $\stat$ as shown below:
\begin{eqnarray*}
    \stat &{}={}& \frac{1}{n(n-1)}\sum_{i\neq j}\inner{ a(X_i)+\gShh\mu}{a(X_j)+\gShh\mu}_{\h} \\ 
    &&\,\,  + \frac{1}{m(m-1)}\sum_{i\neq j}\inner{ a(Y_i)+\gShh\mu}{a(Y_j)+\gShh\mu}_{\h} \\ 
    &&\quad- \frac{2}{nm}\sum_{i,j}\inner{a(X_i)+\gShh\mu}{a(Y_j)+\gShh\mu}_{\h}, 
\end{eqnarray*}
and noting that all the terms in expansion of the inner product cancel except for the terms of the form $\left \langle a(\cdot),a(\cdot)\right \rangle_\h$.
\end{proof}

\begin{appxlem} \label{lemma: Expectation squared}
Let $(G_i)_{i=1}^{n}$ and $(F_i)_{i=1}^{m}$ be independent sequences of zero-mean $\h$-valued random elements, and let f be an arbitrary function in $\h$. Then the following hold.
\begin{enumerate}[label=(\roman*)]
    \itemsep0em 
    \item $\E \left(\sum_{i,j}\left \langle G_i,F_j \right \rangle_{\h}\right)^2=\sum_{i,j}\E\left \langle G_i,F_j \right \rangle_{\h}^2;$
    \item $\E \left(\sum_{i\neq j}\left \langle G_i,G_j \right \rangle_{\h}\right)^2=2 \sum_{i \neq j}\E\left \langle G_i,G_j \right \rangle_{\h}^2;$
    \item $\E \left(\sum_{i}\left \langle G_i,f \right \rangle_{\h}\right)^2=\sum_{i}\E\left \langle G_i,f \right \rangle_{\h}^2.$
\end{enumerate}
\end{appxlem}
\begin{proof} \emph{(i)} can be shown as follows:
\begin{eqnarray*}
&& \E \left(\sum_{i,j}\left \langle G_i,F_j \right \rangle_{\h}\right)^2\\
&{}={}&\sum_{i,j,k,l} \E[\left \langle G_i,F_j \right \rangle_{\h} \left \langle G_k,F_l \right \rangle_{\h}]\\
&{}={}& \sum_{i,j}\E \left \langle G_i,F_j \right \rangle_{\h}^2 + \sum_j \sum_{i\neq k} \E \left[\left \langle G_i,F_j \right \rangle_{\h} \left \langle G_k,F_j \right \rangle_{\h}\right]\\
&&+ \sum_i \sum_{j\neq l} \E \left[\left \langle G_i,F_j \right \rangle_{\h} \left \langle G_i,F_l \right \rangle_{\h}\right]+ \sum_{j\neq l} \sum_{i\neq k} \E \left[\left \langle G_i,F_j \right \rangle_{\h} \left \langle G_k,F_l \right \rangle_{\h}\right] \\
&{}={}&  \sum_{i,j}\E \left \langle G_i,F_j \right \rangle_{\h}^2 + \sum_j \sum_{i\neq k} \left \langle \E(G_i),\E(F_j\left \langle G_k,F_j \right \rangle_{\h}) \right \rangle_{\h} \\
&&+ \sum_i \sum_{j\neq l}  \left \langle \E(G_i\left \langle G_i,F_l \right \rangle_{\h}),\E(F_j) \right \rangle_{\h} + \sum_{j\neq l} \sum_{i\neq k} \left \langle \E(G_i),\E(F_j\left \langle G_k,F_l \right \rangle_{\h}) \right \rangle_{\h} \\
&{}={}& \sum_{i,j}\E \left \langle G_i,F_j \right \rangle_{\h}^2.
\end{eqnarray*}
For \emph{(ii)}, we have
\begin{eqnarray*}
&{}{}&\E \left(\sum_{i\neq j}\left \langle G_i,G_j \right \rangle_{\h}\right)^2
= \sum_{i\neq j} \sum_{k\neq l} \E \left[\left \langle G_i,G_j \right \rangle_{\h} \left \langle G_k,G_l \right \rangle_{\h}\right] \\
&{}={}& \sum_{i\neq j} \E\left[\left \langle G_i,G_j \right \rangle_{\h} \left \langle G_i,G_j \right \rangle_{\h}\right]+\sum_{i\neq j} \E\left[\left \langle G_i,G_j \right \rangle_{\h} \left \langle G_j,G_i \right \rangle_{\h}\right] \\
&&\qquad\qquad+ \sum_{\substack{\{i,j\}\neq \{k,l\}\\{i\neq j}\\{k\neq l}}} \E\left[\left \langle G_i,G_j \right \rangle_{\h} \left \langle G_k,G_l \right \rangle_{\h}\right]. 
\end{eqnarray*}
Consider the last term and note that $\{i,j\} \neq \{k,l\}$ and $k \neq l$, implies that either $k \notin \{i,j\}$, or $l \notin \{i,j\}$. If $k \notin \{i,j\}$, then $$\E\left[\left \langle G_i,G_j \right \rangle_{\h} \left \langle G_k,G_l \right \rangle_{\h}\right] = \left \langle \E(G_k),\E(\left \langle G_i,G_j \right \rangle_{\h}G_l) \right \rangle_{\h} =0,$$ and the same result holds when $l \notin \{i,j\}$. Therefore we conclude that the third term equals zero and the result follows.\vspace{1mm}\\
\emph{(iii)} Note that
\begin{eqnarray*}
    \E \left(\sum_{i}\left \langle G_i,f \right \rangle_{\h}\right)^2  &{}={}& \sum_{i} \E \inner{G_i}{f}_{\h}^2+ \sum_{i\neq j} \E \inner{G_i}{f}_{\h} \inner{G_j}{f}_{\h} \\
    &{}={}& \sum_{i} \E \inner{G_i}{f}_{\h}^2+ \sum_{i\neq j} \inner{\E G_i}{f}_{\h} \inner{\E G_j}{f}_{\h}\\
    &{}={}& \sum_{i} \E \inner{G_i}{f}_{\h}^2
\end{eqnarray*}
and the result follows.
\end{proof}
\begin{appxlem} \label{lemma:bound U-statistic1}
Let $(X_i)_{i=1}^{n} \stackrel{i.i.d.}{\sim} Q$ with $n\geq 2$. Define $$I=\frac{1}{n(n-1)}\sum_{i\neq j}\inner{a(X_i)}{a(X_j)}_\h,$$ where $a(x)= \B\SgL(\kk(\cdot,x)-\mu)$ with $\mu = \int_{\X} \kk(\cdot,x)\,dQ(x)$ and $\B : \h \to \h$ is a bounded operator. Then the following hold.
\begin{enumerate}[label=(\roman*)]
\item $\E\inner{a(X_i)}{a(X_j)}_{\h}^2\leq \norm{\B}_{\op}^4\norm{\SgL\Sigma_Q\SgL}_{\hs}^2;$

\item $\E\left(I^2\right)\leq \frac{4}{n^2}\norm{\B}_{\op}^4\norm{\SgL\Sigma_Q\SgL}_{\hs}^2.$
\end{enumerate}
\end{appxlem}
\begin{proof}
For \emph{(i)}, we have
\begin{eqnarray*}
    \E \inner{ a(X_i)}{a(X_j)}_{\h}^2  &{}={}&\E\inner{ a(X_i) \otimes_{\h}a(X_i)}{a(X_j) \otimes_{\h}a(X_j)}_{\hs} \\
    &{} ={}&\norm{\E \left[a(X_i)\htens a(X_i)\right]}_{\hs}^2 \\ 
    &{} ={}&\norm{\B\SgL \E[(\kk(\cdot,X_i)-\mu)\htens(\kk(\cdot,X_i)-\mu)]\SgL\B^*}^2_{\hs} \\ 
    &{}={}& \norm{\B\SgL \Sigma_Q\SgL\B^*}^2_{\hs} \\ 
    &{} \leq{}&\norm{\B}_{\op}^2\norm{\B^*}_{\op}^2\norm{\SgL \Sigma_Q\SgL}^2_{\hs} \\
    &{} \stackrel{(\ddagger)}{\leq}{}&\norm{\B}_{\op}^4\norm{\SgL \Sigma_Q\SgL}^2_{\hs},
\end{eqnarray*}
where we used the fact that $\norm{\B}_{\op}=\norm{\B^*}_{\op}$ in $(\ddagger)$. On the other hand, \emph{(ii)} can be written as
\begin{eqnarray*}
\E\left(I^2\right)&{}={}&\frac{1}{n^2(n-1)^2}\E\left(\sum_{i\neq j}\inner{ a(X_i)}{a(X_j)}_{\h}\right)^2\\
    &{}\stackrel{(*)}{=}{}& \frac{2}{n^2(n-1)^2}\sum_{i\neq j}\E \inner{ a(X_i)}{a(X_j)}_{\h}^2 \\  &{}\leq{}&\frac{4}{n^2}\norm{\B}_{\op}^4\norm{\SgL \Sigma_Q\SgL}^2_{\hs},
\end{eqnarray*}where $(*)$ follows from Lemma~\ref{lemma: Expectation squared}\emph{(ii)}
\end{proof}

\begin{appxlem} \label{lemma:bound U-statistic2}
Let $(X_i)_{i=1}^{n} \stackrel{i.i.d.}{\sim} Q$. Suppose $G \in \h$ is an arbitrary function and $\B : \h \to \h$ is a bounded operator. Define  $$I=\frac{2}{n}\sum^n_{i=1}\inner{ a(X_i)}{\B\SgL(G-\mu) }_\h,$$ where $a(x)= \B\SgL\left(\kk(\cdot,x)-\mu\right)$ and 
$\mu = \int_{\X} \kk(\cdot,x)\,dQ(x)$. Then
\begin{enumerate}[label=(\roman*)]
\item \begin{align*}&\E\inner{ a(X_i)}{\B\SgL(G-\mu) }_{\h}^2 \leq \norm{\B}_{\op}^4\norm{\SgL\Sigma_Q\SgL}_{\op}\\
&\qquad\qquad\qquad\qquad\qquad\qquad\qquad\times\norm{\SgL(G-\mu)}_{\h}^2;\end{align*}
\item $\E\left(I^2\right)\leq  \frac{4}{n}\norm{\B}_{\op}^4\norm{\SgL\Sigma_Q\SgL}_{\op}\norm{\SgL(G-\mu)}_{\h}^2.$
\end{enumerate}

\end{appxlem}
\begin{proof} 
We prove \emph{(i)} as
\begin{eqnarray*}
   &&\E\inner{a(X_i)}{\B\SgL(G-\mu) }_{\h}^2 \\
    \qquad&{} ={}&\E \inner{ a(X_i)\htens a(X_i)}{ \B\SgL(G-\mu)\htens(G-\mu)\SgL\B^*}_{\hs} \\
    &{} ={}& \inner{ \B\SgL\Sigma_Q\SgL\B^*}{ \B\SgL(G-\mu)\htens(G-\mu)\SgL\B^*}_{\hs} \\
    &{}={}& \text{Tr}\left[\SgL\Sigma_Q\SgL\B^*\B \SgL(G-\mu)\htens(G-\mu)\SgL\B^*\B\right] \\
    &{} \leq{}& \norm{\SgL\Sigma_Q\SgL}_{\op}\norm{\B^*\B}_{\op}^2\text{Tr}\left[\SgL(G-\mu)\htens(G-\mu)\SgL\right] \\ 
    &{}\leq{}& \norm{\SgL\Sigma_Q\SgL}_{\op}\norm{\B}_{\op}^4 \norm{\SgL(G-\mu)}_{\h}^2.
\end{eqnarray*}
\\
\emph{(ii)} From Lemma~\ref{lemma: Expectation squared}\emph{(iii)}, we have
\begin{eqnarray*}
\E\left(I^2\right) &{}={}&\frac{4}{n^2}\sum^n_{i=1}\E\inner{a(X_i)}{\B\SgL(G-\mu) }_{\h}^2,
\end{eqnarray*}
and the result follows from \emph{(i)}.
\end{proof}

\begin{appxlem} \label{lemma:bound U-statistic3}
Let $(X_i)_{i=1}^n \stackrel{i.i.d.}{\sim} Q , \ (Y_i)_{i=1}^m \stackrel{i.i.d.}{\sim} P$ and $\B: \h \to \h$ be a bounded operator. Define
$$I=\frac{2}{nm}\sum_{i,j}\inner{ a(X_i)}{b(Y_j)}_{\h},$$ where $a(x)= \B\SgL(\kk(\cdot,x)-\mu_Q)$, $b(x)= \B\SgL(\kk(\cdot,x)-\mu_P)$, $\mu_Q = \int_{\X} \kk(\cdot,x)\,dQ(x)$ and $\mu_P = \int_{\Y} \kk(\cdot,y)\,dP(y)$. Then
\begin{enumerate}[label=(\roman*)]
\item $\E \inner{ a(X_i)}{b(Y_j)}_{\h}^2 \leq \norm{\B}_{\op}^4\norm{\SgL\Sigma_{PQ}\SgL}_{\hs}^2 $;
\item $\E\left(I^2\right) \leq  \frac{4}{nm}\norm{\B}_{\op}^4\norm{\SgL\Sigma_{PQ}\SgL}_{\hs}^2.$
\end{enumerate}
\end{appxlem}
\begin{proof}
\emph{(i)}
Define $\A := \int_{\X} (K(\cdot,x)-\mu_R) \htens (K(\cdot,x)-\mu_R) \ u(x) \, d\PQ(x) $, where $u=\frac{dP}{dR}-1$ with $R=\frac{P+Q}{2}$. Then it can be verified that  $\Sigma_P=\Sigma_{PQ}+\A-(\mu_R-\mu_P)\htens(\mu_R-\mu_P)$, $\Sigma_Q=\Sigma_{PQ}-\A-(\mu_Q-\mu_R)\htens(\mu_Q-\mu_R)$. Thus $\Sigma_P \preccurlyeq \Sigma_{PQ} + \A$ and $\Sigma_Q \preccurlyeq  \Sigma_{PQ}- \A$. Therefore we have
\begin{eqnarray*} 
    \E \inner{ a(X_i)}{b(Y_j)}_{\h}^2 
    &{} ={}& \E \inner{ a(X_i)\htens a(X_i)}{b(Y_j) \htens b(Y_j)}_{\hs} \\
    &{} ={}&  \inner{ \B\SgL\Sigma_Q\SgL\B^*}{\B\SgL\Sigma_P\SgL\B^*}_{\hs} \\
    &{} \leq{}& \inner{ \B\SgL(\Sigma_{PQ}-\A)\SgL\B^*}{\B\SgL(\Sigma_{PQ}+\A)\SgL\B^*}_{\hs} \\
    &{} ={}& \norm{\B\SgL\Sigma_{PQ}\SgL\B^*}_{\hs}^2-\norm{\B\SgL\A\SgL\B^*}_{\hs}^2 \\
    &{} \leq {}& \norm{\B}_{\op}^4\norm{\SgL\Sigma_{PQ}\SgL}_{\hs}^2.
\end{eqnarray*}
\\
\emph{(ii)} follows by noting that
\begin{align*}
    \E\left(I^2\right) \stackrel{(\dag)}{\leq} \frac{4}{n^2m^2}\sum_{i,j} \E\inner{ a(X_i)}{b(Y_j)}_{\h}^2,
\end{align*}
where $(\dag)$ follows from Lemma~\ref{lemma: Expectation squared}\emph{(i)}.
\end{proof}

\begin{appxlem} \label{lemma: bounds for eta}
Let $u=\frac{dP}{d\PQ}-1  \in \Lp$ and $\eta= \norm{\gSL(\mu_Q-\mu_P)}_{\h}^{2}$, where $g_{\lambda}$ satisfies $(A_1)$--$(A_4)$. Then
$$\eta \leq 4C_1 \norm{\U}_{\Lp}^{2}.$$ 
Furthermore, if $u \in \emph{\range}(\T^{\theta})$, $\theta>0$ and $$\norm{\U}_{\Lp}^2 \geq  \frac{4C_3}{3B_3} \norm{\T}_{\opl}^{2\max(\theta-\xi,0)}\lambda^{2 \Tilde{\theta}} \norm{\T^{-\theta}\U}_{\Lp}^2,$$  where $\Tilde{\theta}= \min(\theta,\xi)$, then, $$\eta \geq B_3 \norm{\U}_{\Lp}^2.$$
\end{appxlem}
\begin{proof}
Note that \begin{equation*}
\begin{split}
    \eta & =\norm{\gSL(\mu_P-\mu_Q)}_{\h}^2 = \inner{g_{\lambda}(\Sigma_{PQ})(\mu_P-\mu_Q)}{\mu_P-\mu_Q}_{\h} \\ & =4 \inner{g_{\lambda}(\Sigma_{PQ})\id^*\U}{\id^*\U}_{\h} =4\inner{\id \gS \id^* \U}{\U}_{\Lp}\\
    &\stackrel{(*)}{=}4\inner{\T\gT \U}{\U}_{\Lp},
\end{split}
\end{equation*}
where we used Lemma~\ref{lemma: bounds for g}\emph{(i)} in $(*)$. The upper bound therefore follows by noting that
$$\eta=4\inner{\T\gT \U}{\U}_{\Lp}  \leq 4\norm{\T\gT}_{\opl}\norm{\U}_{\Lp}^2 
      \leq 4C_1 \norm{\U}_{\Lp}^2.$$ 
For the lower bound, consider
\begin{align*}
    B_3\eta & = 4B_3\inner{\T\gT \U}{\U}_{\Lp}\\
    &= \norm{\T\gT \U}_{\Lp}^2+4B_3^2\norm{\U}_{\Lp}^2-\norm{\T\gT \U-2B_3\U}_{\Lp}^2.
\end{align*}
Since $\U \in \range(\T^{\theta})$, there exists $f\in \Lp$ such that $\U= \T^{\theta}f$. Therefore, we have 
$$\norm{\T\gT \U}_{\Lp}^2 = \sum_{i} \lambda_i^{2\theta+2}\gl^{2}(\lambda_i) \inner{f}{\Tilde{\phi}_i}_{\Lp}^2,$$ and $$\norm{\T\gT \U-2B_3 u}_{\Lp}^2 = \sum_{i} \lambda_i^{2\theta}(\lambda_i\gl(\lambda_i)-2B_3)^2 \inner{f}{\Tilde{\phi}_i}_{\Lp}^2,$$ where $(\lambda_i,\tilde{\phi}_i)_i$ are the eigenvalues and eigenfunctions of $\T$. Using these expressions we have 
\begin{align*}
    \norm{\T\gT \U}_{\Lp}^2-\norm{\T\gT \U-2B_3\U}_{\Lp}^2
    = \sum_{i} 4B_3 \lambda_i^{2\theta}(\lambda_i\gl(\lambda_i)-B_3)\inner{f}{\Tilde{\phi}_i}_{\Lp}^2.
\end{align*}
Thus 
\begin{align*}
    B_3\eta & = 4B_3^2\norm{u}_{\Lp}^2+\sum_{i} 4B_3 \lambda_i^{2\theta}\left(\lambda_i\gl(\lambda_i)-B_3\right)\inner{f}{\Tilde{\phi}_i}_{\Lp}^2 \\
    & \geq 4B_3^2\norm{u}_{\Lp}^2 - \sum_{\{i:\lambda_{i}\gl(\lambda_i)< B_3\}} 4B_3 \lambda_i^{2\theta}\left(B_3-\lambda_i\gl(\lambda_i)\right)\inner{f}{\Tilde{\phi}_i}_{\Lp}^2.
\end{align*}
When $\theta \leq \xi$, by Assumption $(A_3)$, we have
\begin{equation*}
    \sup_{\{i:\lambda_{i}\gl(\lambda_i)< B_3\}} \lambda_i^{2\theta}\left(B_3-\lambda_i\gl(\lambda_i)\right) \leq C_3 \lambda^{2\theta}. 
\end{equation*}
On the other hand, for $\theta > \xi$,
\begin{align*}
    &\sup_{\{i:\lambda_{i}\gl(\lambda_i)< B_3\}} \lambda_i^{2\theta}\left(B_3-\lambda_i\gl(\lambda_i)\right)\\
    &\le \sup_{\{i:\lambda_{i}\gl(\lambda_i)< B_3\}} \lambda_i^{2\theta-2\xi}\sup_{\{i:\lambda_{i}\gl(\lambda_i)< B_3\}}\lambda_i^{2\xi}\left(B_3-\lambda_i\gl(\lambda_i)\right)\\ &\stackrel{(*)}{\le} C_3\norm{\T}^{2\theta-2\xi}_{\opl} \lambda^{2\xi},
\end{align*}
where $(*)$ follows by Assumption $(A_3)$. Therefore we can conclude that
\begin{align*}
    \eta \geq 4B_3\norm{u}^2_{\Lp}-4C_3\norm{\T}_{\opl}^{2\max(\theta-\xi,0)}\lambda^{2 \Tilde{\theta}} \norm{\T^{-\theta}\U}_{\Lp}^2 \stackrel{(\dag)}{\geq} B_3\norm{u}^2_{\Lp},
\end{align*}
where we used $\norm{\U}_{\Lp}^2 \geq  \frac{4C_3}{3B_3} \norm{\T}_{\opl}^{2\max(\theta-\xi,0)}\lambda^{2 \Tilde{\theta}} \norm{\T^{-\theta}\U}_{\Lp}^2$ in $(\dag)$. 
%
\end{proof}

\begin{appxlem} \label{lemma: bounds for g}
Let $g_{\lambda}$ satisfies $(A_1)$--$(A_4)$. Then the following hold.
\begin{enumerate}[label=(\roman*)]
    \item $\id \gS \id^*= \T\gT=\gT\T$;
    \item $\norm{\gSL\SL^{1/2}}_{\op} \leq (C_1+C_2)^{1/2}$;
    \item $\norm{\gSLh\SLh^{1/2}}_{\op} \leq (C_1+C_2)^{1/2}$;
    \item $\norm{\SgL\gShalf}_{\op} \leq C_4^{-1/2}$;
    \item $\norm{\SLh^{-1/2}\igSLh}_{\op} \leq C_4^{-1/2}$.
\end{enumerate}
\end{appxlem}
\begin{proof}
Let $(\lambda_i,\Tilde{\phi_i})_i$ be the eigenvalues and eigenfunctions of $\T.$ Since $\T=\id\id^*$ and $\Sigma_{PQ}=\id^*\id$, we have $\id\id^*\Tilde{\phi}_i=\lambda_i\Tilde{\phi}_i$ which implies $\id^*\id\left(\id^*\Tilde{\phi}_i\right)=\lambda_i\left(\id^*\Tilde{\phi}_i\right)$, i.e., $\Sigma_{PQ}\psi_i=\lambda_i\psi_i$, where $\psi_i:=\id^*\Tilde{\phi}_i/\sqrt{\lambda_i}$. Note that $(\psi_i)$, which are the eigenfunctions of $\Sigma_{PQ}$, form an orthonormal system in $\h$. Define $R:=\frac{P+Q}{2}$.
\vspace{1mm}\\
\emph{(i)} Using the above, we have
\begin{align*}
    &\id \gS \id^* \\& = \id \sum_{i} g_{\lambda}(\lambda_i)\left(\frac{\id^*\Tilde{\phi}_i}{\sqrt{\lambda_i}} \htens \frac{\id^*\Tilde{\phi}_i}{\sqrt{\lambda_i}}\right) \id^* + \gl(0)\left(\T-\id\sum_{i}\left(\frac{\id^*\Tilde{\phi}_i}{\sqrt{\lambda_i}} \htens \frac{\id^*\Tilde{\phi}_i}{\sqrt{\lambda_i}}\right)\id^*\right)\\
    &= \sum_{i}  \lambda_i^{-1} g_{\lambda}(\lambda_i) \id \id^* \left(\Tilde{\phi}_i \ltens \Tilde{\phi}_i\right) \id \id^* +\gl(0)\left(\T-\sum_{i}\lambda_i^{-1}\id \id^* \left(\Tilde{\phi}_i \ltens \Tilde{\phi}_i\right) \id \id^*\right) \\ 
    & \stackrel{(*)}{=} \sum_{i} g_{\lambda}(\lambda_i) \left(\Tilde{\phi}_i \ltens \Tilde{\phi}_i\right) \T+\gl(0)\left(\Id-\sum_{i}\left(\Tilde{\phi}_i \ltens \Tilde{\phi}_i\right)\right) \T \\
    &= \gT \T,
\end{align*}
where $(*)$ follows using $\T (\Tilde{\phi}_i \ltens \Tilde{\phi}_i) = \lambda_i \Tilde{\phi}_i \ltens \Tilde{\phi}_i$. On the other hand,
\begin{align*}
    \gT \T &\stackrel{(\dag)}{=} \sum_{i} g_{\lambda}(\lambda_i) \left(\Tilde{\phi}_i \ltens \Tilde{\phi}_i\right) \sum_{j} \lambda_j \left(\Tilde{\phi}_j \ltens \Tilde{\phi}_j\right)\\
    &= \sum_{i}\sum_{j} g_{\lambda}(\lambda_i) \lambda_j \inner{\Tilde{\phi}_i}{\Tilde{\phi}_j}_{\Lp} \left(\Tilde{\phi}_i \ltens \Tilde{\phi}_j\right) \\
    & = \sum_{i} g_{\lambda}(\lambda_i) \lambda_i \left(\Tilde{\phi}_i \ltens \Tilde{\phi}_i\right)  = \sum_{i} g_{\lambda}(\lambda_i) \T \left(\Tilde{\phi}_i \ltens \Tilde{\phi}_i\right) \\ & = \T \sum_{i \geq 1} g_{\lambda}(\lambda_i) \left(\Tilde{\phi}_i \ltens \Tilde{\phi}_i\right) = \T \gT,
\end{align*}
where $(\dag)$ follows using $\left(I-\sum_{i}\left(\Tilde{\phi}_i \ltens \Tilde{\phi}_i\right)\right)\T=0.$

\emph{(ii)} 
\begin{equation*}
\begin{split}
    \norm{\gSL\SL^{1/2}}_{\op}& = \norm{\gSL\SL\ \gSL}_{\op}^{1/2} \\
    &= \sup_{i}\left| g_{\lambda}(\lambda_i)(\lambda_i+\lambda) \right|^{1/2} \stackrel{(\dag)}{\leq} (C_1+C_2)^{1/2},
\end{split}
\end{equation*}
where $(\dag)$ follows from Assumptions $(A_1)$ and $(A_2)$.\vspace{1mm} \\
\emph{(iii)} The proof is exactly same as that of \emph{(ii)} but with $\Sigma_{PQ}$ being replaced by $\hat{\Sigma}_{PQ}$.\vspace{1mm}\\
\emph{(iv)} 
\begin{equation*}
    \begin{split}
        &\norm{\SgL\gShalf}_{\op}\\& =\norm{\SgL g^{-1}_\lambda(\Sigma_{PQ})\SgL}_{\op}^{1/2} = \sup_i\left|(\lambda+\lambda_i)g_{\lambda}(\lambda_i)\right|^{-1/2} \\ & \leq \left|\inf_i (\lambda+\lambda_i)g_{\lambda}(\lambda_i)\right|^{-1/2} \stackrel{(\ddag)}{\leq} C_4^{-1/2}, 
    \end{split}
\end{equation*}
where $(\ddag)$ follows from Assumption $(A_4)$.\vspace{1mm}\\
\emph{(v)} The proof is exactly same as that of \emph{(iv)} but with $\Sigma_{PQ}$ being replaced by $\hat{\Sigma}_{PQ}$.
\end{proof}
\begin{appxlem} \label{lemma: bound hs and op}
Define $u := \frac{dP}{d\PQ}-1 \in L^2(R)$, $\Nol := \emph{Tr}(\SgL\Sigma_{PQ} \SgL)$, and 
$\Ntl := \norm{\SgL\Sigma_{PQ}\SgL}_{\hs},$ where $R=\frac{P+Q}{2}$. Then the following hold:\vspace{-2mm}
\begin{enumerate}[label=(\roman*)]
    \item $\norm{\SgL\Sigma_V\SgL}_{\hs}^2 \leq 4 \Cl \norm{u}_{\Lp}^2  + 2\Ntlsq;$\vspace{-1mm}
    \item 
    $\norm{\SgL\Sigma_V\SgL}_{\op} \leq 2 \sqrt{\Cl} \norm{u}_{\Lp} + 1,$
    \end{enumerate}\vspace{-1mm}
where $V$ can be either $P$ or $Q$, and $$\Cl =
\left\{
	\begin{array}{ll}
		\Nol \sup_{i} \norm{\phi_i}^2_{\infty},   &  \ \ \sup_i \norm{\phi_i}^2_{\infty} < \infty \\
		\frac{2\Ntl}{\lambda} \sup_x \norm{K(\cdot,x)}_{\h}^2,  & \ \  \text{otherwise}
	\end{array}
\right..$$    
\end{appxlem}
\begin{proof}
Let $(\lambda_i,\Tilde{\phi_i})_i$ be the eigenvalues and eigenfunctions of $\T.$ Since $\T=\id\id^*$ and $\Sigma_{PQ}=\id^*\id$, we have $\id\id^*\Tilde{\phi}_i=\lambda_i\Tilde{\phi}_i$ which implies $\id^*\id\left(\id^*\Tilde{\phi}_i\right)=\lambda_i\left(\id^*\Tilde{\phi}_i\right)$, i.e., $\Sigma_{PQ}\psi_i=\lambda_i\psi_i$, where $\psi_i:=\id^*\Tilde{\phi}_i/\sqrt{\lambda_i}$. Note that $(\psi_i)_i$ form an orthonormal system in $\h$. Define $\phi_i=\frac{\id^*\Tilde{\phi_i}}{\lambda_i}$, thus $\psi_i=\sqrt{\lambda_i}\phi_i.$ Let $\bar{\phi_i} =\phi_i-\E_\PQ\phi_i$. Then $\id\bar{\phi_i}=\id\phi_i=\frac{\T\tilde{\phi}_i}{\lambda_i}=\Tilde{\phi_i}$.
\vspace{2mm}\\
\emph{(i)}
Define $\A := \int_{\X} (K(\cdot,x)-\mu_R) \htens (K(\cdot,x)-\mu_R) \ u(x) \, d\PQ(x) $. Then it can be verified that  $\Sigma_P=\Sigma_{PQ}+\A-(\mu_P-\mu_R)\htens(\mu_P-\mu_R)$, $\Sigma_Q=\Sigma_{PQ}-\A-(\mu_Q-\mu_R)\htens(\mu_Q-\mu_R)$. Thus $\Sigma_P \preccurlyeq \Sigma_{PQ} + \A$ and $\Sigma_Q \preccurlyeq  \Sigma_{PQ}- \A$. Therefore we have,
\begin{eqnarray*}
\norm{\SgL\Sigma_V\SgL}_{\hs}^2 
&{}\leq{}&  2\norm{\SgL \Sigma_{PQ} \SgL}_{\hs}^2+ 2 \norm{\SgL \A \SgL}_{\hs}^2 \\
&{}={}& 2\Ntlsq + 2\norm{\SgL \A \SgL}_{\hs}^2.
\end{eqnarray*}
Next we bound the second term in the above inequality in the two cases for $\Cl$.

\emph{Case 1:} $\sup_{i} \norm{\phi_i}_{\infty}^2 < \infty$. Define $a(x) =K(\cdot,x)-\mu_R$. Then we have, 
\begin{align*}
    &\norm{\SgL \A \SgL}_{\hs}^2 = \text{Tr}\left(\SL^{-1} \A\SL^{-1} \A\right) \\
    &= \int_{\X} \int_{\X} \text{Tr}\left[\SL^{-1}a(x)\htens a(x) \SL^{-1} a(y)\htens a(y) \right] u(x) u(y) \, d\PQ(y)\,d\PQ(x)\\
    &\stackrel{(*)}{=} \sum_{i,j} \left(\frac{\lambda_i}{\lambda_i+\lambda}\right) \left(\frac{\lambda_j}{\lambda_j+\lambda}\right) \left[\int_{\X}\bar{\phi_i}(x)\bar{\phi_j}(x) u(x) \, d\PQ(x)\right]^2 \\
    & \leq \sum_{i,j} \frac{\lambda_i}{\lambda_i+\lambda} \inner{\bar{\phi_j}}{\bar{\phi_i}u}_{\Lp}^2 = \sum_{i} \frac{\lambda_i}{\lambda_i+\lambda} \sum_{j}\inner{\id{\bar{\phi}_j}}{\bar{\phi_i}u}_{\Lp}^2 \\
    & = \sum_{i} \frac{\lambda_i}{\lambda_i+\lambda} \sum_{j}\inner{\Tilde{\phi_j}}{\bar{\phi_i}u}_{\Lp}^2 \leq \sum_{i} \frac{\lambda_i}{\lambda_i+\lambda} \norm{\bar{\phi_i}u}_{\Lp}^2 \\
    & \leq 2 \Nol \sup_i \norm{\phi_i}_{\infty}^2 \norm{u}_{\Lp}^2,
\end{align*}
where $(*)$ follows from
\begin{align*}
    &\text{Tr}\left[\SL^{-1}a(x)\htens a(x) \SL^{-1} a(y)\htens a(y) \right] \\
    &= \inner{a(y)}{\SL^{-1}a(x)}_{\h}\inner{a(x)}{\SL^{-1}a(y)}_{\h} = \inner{a(y)}{\SL^{-1}a(x)}_{\h}^2 \\
    &= \left\langle \sum_i \frac{1}{\lambda_i+\lambda} \inner{a(x)}{\sqrt{\lambda_i}\phi_i}_{\h} \sqrt{\lambda_i}\phi_i\right.\\
    &\qquad\qquad \left. +\frac{1}{\lambda}\left(a(x)-\sum_{i}\inner{a(x)}{\sqrt{\lambda_i}\phi_i}_{\h}\sqrt{\lambda_i}\phi_i\right),a(y)\right\rangle_{\h}^2 \\
    & \stackrel{(\dag)}{=} \left[\sum_i \frac{\lambda_i}{\lambda_i+\lambda}\bar{\phi_i}(x)\bar{\phi_i}(y)+\frac{1}{\lambda}\left(\inner{a(x)}{a(y)}_{\h}-\sum_i\lambda_i\bar{\phi_i}(x)\bar{\phi_i}(y)\right) \right]^2\\
    &\stackrel{(\ddag)}{=}\left[\sum_i \frac{\lambda_i}{\lambda_i+\lambda}\bar{\phi_i}(x)\bar{\phi_i}(y)\right]^2,
\end{align*}
where $(\dag)$ follows from $\inner{K(\cdot,x)-\mu_R}{\sqrt{\lambda_i}\phi_i}_{\h}= \sqrt{\lambda_i}\bar{\phi_i}(x),$ and in $(\ddag)$ we used $$\inner{a(x)}{a(y)}_{\h}=\sum_i\lambda_i\bar{\phi_i}(x)\bar{\phi_i}(y),$$ which is proved below. Consider
\begin{align*}
&\inner{a(x)}{a(y)}_{\h} \\
&= \inner{K(\cdot,x)-\mu_R}{K(\cdot,y)-\mu_R}_{\h} \\
&=K(x,y)-\E_{R}K(x,Y)-\E_{R}K(X,y)-\E_{R\times R}K(X,Y) := \bar{K}(x,y).
\end{align*}
Furthermore, from the definition of $\T$, we can equivalently it write as $\T: \Lp \to \Lp$, $f \mapsto \int \bar{K}(\cdot,x)f(x)\,d\PQ(x)$. Thus by Mercer's theorem (see \citealt[Lemma 2.6]{Steinwart2012MercersTO}), we obtain $\bar{K}(x,y)=\sum_i\lambda_i\bar{\phi_i}(x)\bar{\phi_i}(y)$.

\emph{Case 2:} Suppose $\sup_{i} \norm{\phi_i}_{\infty}^2$ is not finite. From the calculations in \emph{Case 1}, we have
\begin{align*}
    & \norm{\SgL\A\SgL}_{\hs}^2 = \int_{\X} \int_{\X} \inner{a(x)}{\SL^{-1}a(y)}_{\h}^2 u(x)u(y) \,d\PQ(x) \,d\PQ(y) \\
    &= \int_{\X} \int_{\X} \inner{\SgL a(x)\htens a(x) \SgL}{\SgL a(y) \htens a(y) \SgL}_{\hs}\\
    &\qquad\qquad\qquad \times u(x)u(y)\,d\PQ(x) \,d\PQ(y) \\
    &\leq \left(\int_{\X} \int_{\X} \inner{\SgL a(x)\htens a(x) \SgL}{\SgL a(y) \htens a(y) \SgL}_{\hs}^2 \right.\\
    &\qquad\qquad\times\,d\PQ(x) \,d\PQ(y) \Big)^{1/2}\left(\int_{\X}\int_{\X}u(x)^2u(y)^2 \,d\PQ(x) d\PQ(y)\right)^{1/2}\\
    &\stackrel{(*)}{\leq} \left(\int_{\X} \int_{\X} \inner{\SgL a(x)\htens a(x) \SgL}{\SgL a(y) \htens a(y) \SgL}_{\hs}\right.\\
    &\qquad\qquad\qquad\times\,d\PQ(x) \,d\PQ(y) \Big)^{1/2}\\
    &\qquad\times \sup_{x,y} \inner{\SgL a(x)\htens a(x) \SgL}{\SgL a(y) \htens a(y) \SgL}^{1/2}_{\hs}\\
&\qquad\qquad\times\left(\int_{\X}\int_{\X}u(x)^2u(y)^2 \,d\PQ(x) d\PQ(y)\right)^{1/2}\\
    & \leq \norm{\SgL\Sigma_{PQ}\SgL}_{\hs} \norm{\SL^{-1}}_{\op} \sup_{x}  \norm{a(x)}^2_{\h} \norm{u}_{\Lp}^2 \\
& \leq \frac{4\Ntl}{\lambda} \sup_x \norm{K(\cdot,x)}_{\h}^2 \norm{u}_{\Lp}^2,
    \end{align*}    
where we used \begin{align*}&\inner{\SgL a(x)\htens a(x) \SgL}{\SgL a(y) \htens a(y) \SgL}_{\hs}\\ &= \text{Tr}\left[\SL^{-1}a(x)\htens a(x) \SL^{-1} a(y)\htens a(y) \right]\\
&= \inner{a(y)}{\SL^{-1}a(x)}_{\h}\inner{a(x)}{\SL^{-1}a(y)}_{\h} = \inner{a(y)}{\SL^{-1}a(x)}_{\h}^2 \geq 0\end{align*}
in $(*)$.\vspace{2mm}

\noindent \emph{(ii)} Note that
\begin{align*}
     \norm{\SgL\Sigma_{V}\SgL}_{\op} 
     & \leq \norm{\SgL\Sigma_{PQ}\SgL}_{\op} + \norm{\SgL\A\SgL}_{\op} \\
    & \leq 1 + \norm{\SgL\A\SgL}_{\hs}.
\end{align*}
The result therefore follows by using the bounds in part $(i)$.
\end{proof}


\begin{appxlem}\label{lemma:intermediate}
Let $A:H\rightarrow H$ and $B:H\rightarrow H$ be bounded operators on a Hilbert space, $H$ such that $AB^{-1}$ is bounded. Then $\Vert B\Vert_{\EuScript{L}^\infty(H)}\ge \Vert AB^{-1}\Vert^{-1}_{\EuScript{L}^\infty(H)}\Vert A\Vert_{\EuScript{L}^\infty(H)}$. Also for any $f\in H$, $$\Vert Bf\Vert_H\ge \Vert AB^{-1}\Vert^{-1}_{\EuScript{L}^\infty(H)}\Vert Af\Vert_H.$$ 
\end{appxlem}
\begin{proof}
The result follows by noting that $$\Vert A\Vert_{\EuScript{L}^\infty(H)}=\Vert AB^{-1}B\Vert_{\EuScript{L}^\infty(H)}\le \Vert AB^{-1}\Vert_{\EuScript{L}^\infty(H)}\Vert B\Vert_{\EuScript{L}^\infty(H)}$$ and $\Vert Af\Vert_H=\Vert AB^{-1}Bf\Vert_H\le \Vert AB^{-1}\Vert_{\EuScript{L}^\infty(H)}\Vert Bf\Vert_H$.
\end{proof}
\begin{appxlem} \label{lemma: denominator lower bound}
Let $\zeta = \norm{\gSLh(\mu_Q-\mu_P)}_{\h}^2$. Then $$\zeta \geq C_4 \Cs^{-1}\norm{\M^{-1}}^{-2}_{\op} \eta,$$
where $\eta= \norm{\gSL(\mu_Q-\mu_P)}_{\h}^{2}$, and $\M = \SLh^{-1/2}\SL^{1/2}$.
\end{appxlem}
\begin{proof}
Repeated application of Lemma~\ref{lemma:intermediate} in conjunction with assumption $(A_4)$ yields
\begin{align*}
    \zeta &
     \geq \norm{\SgLh \igSLh}_{\op}^{-2}\norm{\SLh^{-1/2}(\mu_Q-\mu_P)}_{\h}^2 \\ &\geq C_4 \norm{\M^{-1}}^{-2}_{\op} \norm{\SgL(\mu_Q-\mu_P)}_{\h}^2 \\
    &\geq C_4 \norm{\M^{-1}}^{-2}_{\op} \norm{\gSL\SL^{1/2}}^{-2}_{\h} \eta\\
    &\geq C_4 \Cs^{-1}\norm{\M^{-1}}^{-2}_{\op} \eta  ,     
\end{align*}
where the last inequality follows from Lemma \ref{lemma: bounds for g}\emph{(ii)}.
\end{proof}
\begin{appxlem} \label{Lemma: bounding expectations}
Let $\zeta = \norm{\gSLh(\mu_P-\mu_Q)}_{\h}^2$, $\M = \SLh^{-1/2}\SL^{1/2}$, and $m \leq n \leq Dm$ for some constant $D\geq1$. Then
\begin{align*}
    &\E \left[ (\stat-\zeta)^2 |(Z_i)_{i=1}^{s}\right ] \\
    & \quad \le \tilde{C} \norm{\M}_{\op}^4 \left(\frac{\Cl \norm{\U}_{\Lp}^2+\Ntlsq}{(n+m)^2}+\frac{\sqrt{\Cl}\norm{\U}_{\Lp}^3+\norm{\U}_{\Lp}^2}{n+m}\right),
\end{align*}
where $\Cl$ is defined in Lemma~\ref{lemma: bound hs and op} and $\tilde{C}$ is a constant that depends only on $C_1$, $C_2$ and $D$. 
Furthermore, if $P=Q$, then
\begin{align*}
    &\E[\stat^2|(Z_i)_{i=1}^{s}]\leq 6\Cs^2\norm{\M}_{\op}^4\Ntlsq\left(\frac{1}{n^2}+\frac{1}{m^2}\right).
\end{align*}

\end{appxlem}

\begin{proof}
Define $a(x)= \B\SgL(\kk(\cdot,x)-\mu_P)$, and $b(x)=\B\SgL(\kk(\cdot,x)-\mu_Q)$, where $\B= \gSLh \SL^{1/2}$. Then we have 
\begin{align*}
    \stat& \stackrel{(*)}{=} \frac{1}{n(n-1)} \sum_{i\neq j} \left \langle a(X_i),a(X_j)\right \rangle_{\h}  + \frac{1}{m(m-1)}\sum_{i\neq j} \left \langle a(Y_i),a(Y_j)\right \rangle_{\h} \\ 
   &\qquad- \frac{2}{nm}\sum_{i,j} \left \langle a(X_i),a(Y_j)\right \rangle_{\h} \\
   & \stackrel{(\dag)}{=} \frac{1}{n(n-1)} \sum_{i\neq j} \left \langle a(X_i),a(X_j)\right \rangle_{\h}+ \frac{1}{m(m-1)}\sum_{i \neq j} \inner{b(Y_i)}{b(Y_j)}_{\h}  \\   
   & \qquad  +\frac{2}{m}\sum_{i=1}^{m} \inner{b(Y_i)}{\B\SgL(\mu_Q-\mu_P)}_{\h}+ \zeta - \frac{2}{nm}\sum_{i,j} \inner{ a(X_i)}{b(Y_j)}_{\h} \\ 
   & \qquad- \frac{2}{n}\sum_{i=1}^{n} \inner{a(X_i)}{\B\SgL(\mu_Q-\mu_P)}_{\h},
\end{align*}
where $(*)$ follows from Lemma \ref{lemma:format of statistic} and $(\dag)$ follows by writing  $a(Y)=b(Y) + \B\SgL(\mu_Q-\mu_P)$ in the last two terms.
Thus we have 
\begin{multline*}
    \stat-\zeta = \underbrace{\frac{1}{n(n-1)} \sum_{i\neq j} \left \langle a(X_i),a(X_j)\right \rangle_{\h}}_{\circled{\footnotesize{1}}} + \underbrace{ \frac{1}{m(m-1)}\sum_{i \neq j} \inner{b(Y_i)}{b(Y_j)}_{\h}}_{\circled{\footnotesize{2}}}  \\   +\underbrace{\frac{2}{m}\sum_{i=1}^{m} \inner{b(Y_i)}{\B\SgL(\mu_Q-\mu_P)}_{\h}}_{\circled{\footnotesize{3}}} - \underbrace{\frac{2}{nm}\sum_{i,j} \left \langle a(X_i),b(Y_j)\right \rangle_{\h}}_{\circled{\footnotesize{4}}} \\ 
    - \underbrace{\frac{2}{n}\sum_{i=1}^{n} \inner{a(X_i)}{\B\SgL(\mu_Q-\mu_P)}_{\h}}_{\circled{\footnotesize{5}}}.
\end{multline*}
Furthermore using Lemma \ref{lemma: bounds for g}\emph{(iii)},
\begin{eqnarray*}
    \norm{\B}_{\op} &{}={}&\norm{\B^*}_{\op}=\norm{\gSLh\SL^{1/2}}_{\op}\\
    &{}\leq{}& \norm{\gSLh\SLh^{1/2}}_{\op}\norm{\SLh^{-1/2}\SL^{1/2}}_{\op} \\ 
    &{} \leq{}& \Cs^{1/2} \norm{\M}_{\op}.
\end{eqnarray*}
Next we bound the terms $\circled{\footnotesize{1}}$--$\circled{\footnotesize{5}}$ using Lemmas \ref{lemma:bound U-statistic1}, \ref{lemma:bound U-statistic2}, \ref{lemma:bound U-statistic3} and \ref{lemma: bound hs and op}. 
It follows from Lemmas~\ref{lemma:bound U-statistic1}\emph{(ii)} and \ref{lemma: bound hs and op}\emph{(i)} that
\begin{align*}
        \E\left(\circled{\footnotesize{1}}^2 |(Z_i)_{i=1}^{s}\right) & \leq \frac{4}{n^2}\|\B\|_{\op}^4\left\|\SgL\Sigma_P\SgL\right\|_{\hs}^2 \\
        &\leq \frac{4}{n^2}\Cs^2\norm{\M}_{\op}^4 \left(4\Cl \norm{\U}_{\Lp}^{2}+2 \Ntlsq\right),
\end{align*}
and
\begin{align*}
        \E\left(\circled{\footnotesize{2}}^2 |(Z_i)_{i=1}^{s}\right)
        & \leq \frac{4}{m^2}\|\B\|_{\op}^4\left\|\SgL\Sigma_Q\SgL\right\|_{\hs}^2 \\
        &\leq \frac{4}{m^2}\Cs^2\norm{\M}_{\op}^4 \left(4\Cl \norm{\U}_{\Lp}^{2}+2 \Ntlsq\right).
\end{align*}
Using Lemma~\ref{lemma:bound U-statistic2}\emph{(ii)} and \ref{lemma: bound hs and op}\emph{(ii)}, we obtain
\begin{align*}
        &\E\left(\circled{\footnotesize{3}}^2 |(Z_i)_{i=1}^{s}\right)\\ & \leq \frac{4}{m}\left\|\SgL\Sigma_{Q}\SgL\right\|_{\op}\|\B\|_{\op}^4\left\|\SgL(\mu_P-\mu_Q)\right\|_{\h}^2 \\ & \leq \frac{4}{m}\Cs^2\norm{\M}_{\op}^4\left(1+2\sqrt{\Cl}\norm{u}_{\Lp}\right)\left\|\SgL(\mu_P-\mu_Q)\right\|_{\h}^2 \\
        & \stackrel{(*)}{\leq} \frac{16}{m}\Cs^2\norm{\M}_{\op}^4 \left(1+2\sqrt{\Cl}\norm{u}_{\Lp}\right)\norm{\U}_{\Lp}^2,
\end{align*}
and
\begin{align*}
     &\E\left(\circled{\footnotesize{5}}^2 |(Z_i)_{i=1}^{s} \right)\\ &\leq \frac{4}{n} \left\|\SgL\Sigma_P\SgL\right\|_{\op}  \|\B\|_{\op}^4\left\|\SgL(\mu_Q-\mu_P)\right\|_{\h}^2 \\
     &  \leq \frac{4}{n}   \Cs^2 \norm{\M}_{\op}^4(1+2\sqrt{\Cl}\norm{\U}_{\Lp}) \left\|\SgL(\mu_Q-\mu_P)\right\|_{\h}^2  \\
     & \stackrel{(*)}{\leq} \frac{16}{n}\Cs^2\norm{\M}_{\op}^4 \left(1+2\sqrt{\Cl}\norm{u}_{\Lp}\right)\norm{\U}_{\Lp}^2,
\end{align*}
where $(*)$ follows from using $g_\lambda(x)=(x+\lambda)^{-1}$ with $C_1=1$ in Lemma \ref{lemma: bounds for eta}. For term \circled{\footnotesize{4}}, using Lemma \ref{lemma:bound U-statistic3} yields,
\begin{align*}
     \E\left(\circled{\footnotesize{4}}^2 |(Z_i)_{i=1}^{s} \right) \leq  \frac{4}{nm}\|\B\|_{\op}^4\Ntlsq 
     \leq \frac{4}{nm}\Cs^2\norm{\M}_{\op}^4\Ntlsq. 
\end{align*}
Combining these bounds with the fact that $\sqrt{ab} \leq \frac{a}{2}+\frac{b}{2}$, and that $(\sum_{i=1}^k a_k)^2 \leq k \sum_{i=1}^k a_k^2$ for any $a,b,a_k \in \R$, $k \in \N$ yields that 
\begin{align*}
    &\E[ (\stat-\zeta)^2 |(Z_i)_{i=1}^{s} ] \\
    & \lesssim \norm{\M}_{\op}^4 \left(\Cl \norm{\U}_{\Lp}^2+\Ntlsq\right)(n^{-2}+m^{-2})\nonumber\\
&\qquad\qquad+\norm{\M}_{\op}^4\left(\sqrt{\Cl}\norm{\U}_{\Lp}^3+\norm{\U}_{\Lp}^2\right)(n^{-1}+m^{-1}) \\
    &\stackrel{(*)}{\lesssim} \norm{\M}_{\op}^4 \left(\frac{\Cl \norm{\U}_{\Lp}^2+\Ntlsq}{(n+m)^2}+\frac{\sqrt{\Cl}\norm{\U}_{\Lp}^3+\norm{\U}_{\Lp}^2}{n+m}\right),
\end{align*}
where $(*)$ follows by using Lemma \ref{lem:mn_bound}.

When $P=Q$, and using the same Lemmas as above, we have 
\begin{align*}
        \E\left(\circled{\footnotesize{1}}^2 |(Z_i)_{i=1}^{s}\right)&\leq \frac{4}{n^2}\Cs^2\norm{\M}_{\op}^4\Ntlsq,\\
        \E\left(\circled{\footnotesize{2}}^2 |(Z_i)_{i=1}^{s}\right)
        &\leq \frac{4}{m^2}\Cs^2\norm{\M}_{\op}^4\Ntlsq,\\
        \E\left(\circled{\footnotesize{4}}^2 |(Z_i)_{i=1}^{s}\right) &\leq \frac{4}{nm}\Cs^2\norm{\M}_{\op}^4 \Ntlsq,
\end{align*}
and $\circled{\footnotesize{3}} = \circled{\footnotesize{5}}= 0$.
Therefore, 
\begin{align*}
    \E[\stat^2 | (Z_i)_{i=1}^{s}] &= \E \left[\left(\circled{\footnotesize{1}} + \circled{\footnotesize{2}} + \circled{\footnotesize{4}}\right)^2 | (Z_i)_{i=1}^{s}\right]\\
    & \stackrel{(*)}{\leq} \E\left(\circled{\footnotesize{1}}^2 + \circled{\footnotesize{2}}^2 + \circled{\footnotesize{4}}^2|(Z_i)_{i=1}^{s} \right) \\ & \stackrel{(\dag)}{\leq} \Cs^2\norm{\M}_{\op}^4\Ntlsq\left(\frac{6}{m^2}+\frac{6}{n^2}\right),
\end{align*}
where $(*)$ follows by noting that $\E\left(\circled{\footnotesize{1}}\cdot\circled{\footnotesize{2}}\right)=\E\left(\circled{\footnotesize{1}}\cdot\circled{\footnotesize{4}}\right)=\E\left(\circled{\footnotesize{2}}\cdot\circled{\footnotesize{4}}\right) = 0$ under the assumption $P=Q$, and $(\dag)$ follows using  $\sqrt{ab} \leq \frac{a}{2}+\frac{b}{2}.$
\end{proof}

\begin{appxlem} \label{lem:mn_bound}
For any $n,m \in (0,\infty)$, if $m \leq n \leq Dm$ for some $D\geq1$, then for any $a>0$ $$\frac{1}{m^a}+\frac{1}{n^a} \leq \frac{2^a(D^a+1)}{(m+n)^a}.$$
\end{appxlem}
\begin{proof}
Observe that using $m \leq n \leq Dm$ yields
$    \frac{1}{m^a}+\frac{1}{n^a} \leq \frac{D^a+1}{n^a} = \frac{2^a(D^a+1)}{(2n)^a} \leq \frac{2^a(D^a+1)}{(m+n)^a}.$
\end{proof}

\begin{appxlem}\label{lemma:DKW for quantile}
  Define $q_{1-\alpha}^{\lambda}:= \inf\{q \in \R: F_{\lambda}(q) \geq 1-\alpha\},$ where
$$F_{\lambda}(x):= \frac{1}{D}\sum_{\pi \in \Pi_{n+m}}\II(\hat{\eta}^{\pi}_{\lambda} \leq x),$$ is the permutation distribution function. Let $(\pi^i)_{i=1}^B$ be $B$ randomly selected permutations from $\Pi_{n+m}$ and define
$$\hat{F}^{B}_{\lambda}(x):= \frac{1}{B}\sum_{i=1}^{B}\II(\hat{\eta}^{^i}_{\lambda} \leq x),$$ where $\hat{\eta}^{i}_{\lambda}:=\stat(X^{\pi^i},Y^{\pi^i},Z)$ is the statistic based on the permuted samples. Define $$\hat{q}_{1-\alpha}^{B,\lambda}:= \inf\{q \in \R: \hat{F}^B_{\lambda}(q) \geq 1-\alpha\}.$$ Then, for any $\alpha>0,\,\Tilde{\alpha}>0, \ \delta>0$, if $B\geq \frac{1}{2\Tilde{\alpha}^2}\log2\delta^{-1}$, the following hold:
\begin{enumerate}[label=(\roman*)]
    \item $P_{\pi}(\hat{q}_{1-\alpha}^{B,\lambda} \geq q_{1-\alpha-\Tilde{\alpha}}^{\lambda}) \geq 1-\delta$;
    \item $P_{\pi}(\hat{q}_{1-\alpha}^{B,\lambda} \leq q_{1-\alpha+\Tilde{\alpha}}^{\lambda}) \geq 1-\delta.$ 
\end{enumerate}
\end{appxlem}
\begin{proof}
We first use the Dvoretzky-Kiefer-Wolfowitz (see \citealt{DKW}, \citealt{Massart}) inequality to get uniform error bound for the empirical permutation distribution function, and then use it to obtain bounds on the empirical quantiles. Let $$\A:= \left\{\sup_{x \in R} |\hat{F}^{B}_{\lambda}(x)-F_{\lambda}(x)| \leq \sqrt{\frac{1}{2B}\log(2 \delta^{-1})}\right\}.$$
Then DKW inequality yields
$P_{\pi}(\A) \geq 1-\delta$. Now assuming $\A$ holds, we have 
\begin{align*}
    \qqh &= \inf\{ q \in \R : \hat{F}^B_{\lambda}(q) \geq 1-\alpha\} \\
    & \geq \inf\left\{q \in \R : F_{\lambda}(q)+\sqrt{\frac{1}{2B}\log(2 \delta^{-1})} \geq 1-\alpha\right\} \\
    & = \inf \left\{q \in \R : F_{\lambda}(q) \geq 1-\alpha-\sqrt{\frac{1}{2B}\log(2 \delta^{-1})}\right\}.
\end{align*}
Furthermore, we have 
\begin{align*}
    \qqh &= \inf\{ q \in \R : \hat{F}^B_{\lambda}(q) \geq 1-\alpha\} \\
    & \leq \inf\left\{q \in \R : F_{\lambda}(q)-\sqrt{\frac{1}{2B}\log(2 \delta^{-1})} \geq 1-\alpha\right\} \\
    & = \inf \left\{q \in \R : F_{\lambda}(q) \geq 1-\alpha+\sqrt{\frac{1}{2B}\log(2 \delta^{-1})}\right\}.
\end{align*}
Thus \emph{(i)} and \emph{(ii)} hold if $\sqrt{\frac{1}{2B}\log(2 \delta^{-1})} \leq \Tilde{\alpha}$, which is equivalent to the condition
$B\geq \frac{1}{2\Tilde{\alpha}^2}\log(2 \delta^{-1}).$
\end{proof}

\begin{appxlem} \label{lemma:bound quantile}
For $0<\alpha\leq e^{-1}$, $\delta>0$ and $m \leq n \leq Dm$, there exists a constant $ C_5>0$ such that
\begin{align*}
    P_{H_1}(\qqe \leq C_5\gamma ) \geq 1-\delta ,
\end{align*}
where \begin{align*}\gamma &= \frac{\norm{\M}_{\op}^2\log\frac{1}{\alpha}}{\sqrt{\delta}(n+m)}\left(\sqrt{\Cl}\norm{u}_{\Lp}+\Ntl+\Cl^{1/4}\norm{u}^{3/2}_{\Lp}+\norm{u}_{\Lp}\right)\\
&\qquad\qquad\qquad+\frac{\zeta\log\frac{1}{\alpha}}{\sqrt{\delta}(n+m)},\end{align*} $\zeta = \norm{\gSLh(\mu_Q-\mu_P)}_{\h}^2$, and $\Cl$ is defined in Lemma~\ref{lemma: bound hs and op}.
\end{appxlem}
\begin{proof}
Let $\B=\gSLh\SL^{1/2}$, $a(x)= \B\SgL(\kk(\cdot,x)-\mu_P)$ , and $$b(x)= \B\SgL(\kk(\cdot,x)-\mu_Q).$$ By \cite[Equation 59]{permutations}, we can conclude that given the samples $(Z_i)_{i=1}^s$ there exists a constant $C_6>0$ such that
\begin{equation*}
    \qqe \leq C_6I \log \frac{1}{\alpha},
\end{equation*}
almost surely, where
\begin{align*}
I^2 &:= \frac{1}{m^2(m-1)^2}\sum_{i\neq j}\inner{a(X_i)}{a(X_j)}^2_{\h} +\frac{1}{m^2(m-1)^2}\sum_{i\neq  j}\inner{a(Y_i)}{a(Y_j)}^2_{\h}\\
& \qquad +\frac{2}{m^2(m-1)^2}\sum_{i, j}\inner{a(X_i)}{a(Y_j)}^2_{\h}.
\end{align*}
We bound $I^2$ as 
\begin{align*}
I^2  &\stackrel{(*)}{\lesssim} \frac{1}{m^2(m-1)^2}\sum_{i\neq j}\inner{a(X_i)}{a(X_j)}^2_{\h} + \frac{1}{m^2(m-1)^2}\sum_{i\neq j}\inner{b(Y_i)}{b(Y_j)}^2_{\h} \\
& \qquad + \frac{1}{m^2(m-1)}\sum_{i=1}^{m} \inner{b(Y_i)}{\B\SgL(\mu_Q-\mu_P)}^2_{\h} +\frac{\zeta^2}{m(m-1)} \\ 
& \qquad  + \frac{1}{m^2(m-1)}\sum_{i=1}^{m} \inner{a(X_i)}{\B\SgL(\mu_Q-\mu_P)}^2_{\h}\\
&\qquad +\frac{1}{m^2(m-1)^2}\sum_{i,j}^{m} \inner{a(X_i)}{b(Y_j)}^2_{\h},
\end{align*}
where $(*)$ follows by writing $a(Y)=b(Y) + \B\SgL(\mu_Q-\mu_P)$ then using  $(\sum_{i=1}^k a_k)^2 \leq k \sum_{i=1}^k a_k^2,$ for any $a_k \in \R$, $k \in \N$. Then following the procedure similar to that in the proof of Lemma \ref{Lemma: bounding expectations}, we can bound the expectation of each term using Lemma \ref{lemma:bound U-statistic1}, \ref{lemma:bound U-statistic2},  \ref{lemma:bound U-statistic3}, \ref{lemma: bounds for eta}, and \ref{lemma: bound hs and op} resulting in
\begin{align*}
& \E(I^2 |(Z_i)_{i=1}^s) \\
&\lesssim \frac{\norm{\M}_{\op}^4}{m^2}\left(\Cl\norm{u}_{\Lp}^2+\Ntlsq+\sqrt{\Cl}\norm{u}^{3}_{\Lp}+\norm{u}_{\Lp}^2\right)+\frac{\zeta^2}{m^2} \\
& \lesssim \frac{\norm{\M}_{\op}^4}{(n+m)^2}\left(\Cl\norm{u}_{\Lp}^2+\Ntlsq+\sqrt{\Cl}\norm{u}^{3}_{\Lp}+\norm{u}_{\Lp}^2\right)
+\frac{\zeta^2}{(n+m)^2},
\end{align*}
where in the last inequality we used Lemma \ref{lem:mn_bound}.
 Thus using $ \qqe \leq C_6I \log \frac{1}{\alpha}$ and Markov's inequality, we obtain the desired result.
\end{proof}

\begin{appxlem} \label{lemma:adaptation}
Let $f$ be a function of a random variable $X$ and some (deterministic) parameter $\lambda\in \Lambda$, where $\Lambda$ has finite cardinality $|\Lambda|$.
Let $\gamma(\alpha,\lambda)$ be any function of $\lambda$ and $\alpha>0$. If for all $\lambda \in \Lambda$ and $\alpha>0$, $P\{f(X,\lambda) \geq \gamma(\alpha,\lambda)\} \leq \alpha$, then $$P\left\{\bigcup_{\lambda \in \Lambda}f(X,\lambda) \geq \gamma\left(\frac{\alpha}{|\Lambda|}, \lambda\right) \right\} \leq \alpha.$$
Furthermore, if $P\{f(X,\lambda^{*})\} \geq \gamma(\alpha,\lambda^*)\} \geq \delta$ for some $\lambda^* \in \Lambda$ and $\delta>0$, then
$$P\left\{\bigcup_{\lambda \in \Lambda}f(X,\lambda) \geq \gamma(\alpha,\lambda)\right\} \geq \delta.$$
\end{appxlem}

\begin{proof}
The proof follows directly from the fact that for any sets $A$ and $B$, $P(A \cup B) \leq P(A)+P(B)$:
\begin{align*}
    P\left\{\bigcup_{\lambda \in \Lambda}f(X,\lambda) \geq \gamma\left(\frac{\alpha}{|\Lambda|}, \lambda\right) \right\} & \leq \sum_{\lambda \in \Lambda} P\left\{f(X,\lambda) \geq \gamma\left(\frac{\alpha}{|\Lambda|}, \lambda\right) \right\} \\
    & \leq \sum_{\lambda \in \Lambda}\frac{\alpha}{|\Lambda|} = \alpha.
\end{align*}
For the second part, we have 
\begin{align*}
   P\left\{\bigcup_{\lambda \in \Lambda}f(X,\lambda) \geq \gamma(\alpha,\lambda)\right\} \geq  P\left\{f(X,\lambda^{*})\geq \gamma(\alpha,\lambda^*)\right\},
\end{align*}
and the result follows.
\end{proof}

\begin{appxlem} \label{lem: bound M operator}
Let $H$ be an RKHS with reproducing kernel k that is defined on a separable topological space, $\Y$. Define 
$$\Sigma=\frac{1}{2}\int_{\Y}\int_{\Y}(s(x)-s(y))\Htens(s(x)-s(y)) \, dR(x)dR(y),$$
where $s(x):=k(\cdot,x)$. Let $(\psi_i)_{i}$ be orthonormal eigenfunctions of $\Sigma$ with corresponding eigenvalues $(\lambda_i)_i$ that satisfy $C:= \sup_i\norm{\frac{\psi_i}{\sqrt{\lambda_i}}}_{\infty} < \infty$. Given $(Y_i)_{i=1}^r \stackrel{i.i.d.}{\sim} R$ with $r\geq 2$, define 
$$\hat{\Sigma}:=\frac{1}{2r(r-1)}\sum_{i\neq j}^{r} (s(Y_i)-s(Y_j)) \Htens (s(Y_i)-s(Y_j)).$$
Then for any $0\leq \delta \leq \frac{1}{2}$, $r\ge 136C^2\Nol\log\frac{8\Nol}{\delta} $ and $\lambda \leq \norm{\Sigma}_{\opH}$ where $\Sigma_{\lambda}:=\Sigma+\lambda\Id$ and $\Nol := \emph{Tr}(\Sigma_{\lambda}^{-1/2}\Sigma\Sigma_{\lambda}^{-1/2})$, the following hold:
\begin{enumerate}[label=(\roman*)]
    \item $P^r\left\{(Y_i)_{i=1}^r : \norm{\Sigma_{\lambda}^{-1/2}(\hat{\Sigma}-\Sigma)\Sigma_{\lambda}^{-1/2}}_{\opH}\leq \frac{1}{2}\right\} \geq 1-2\delta$;
    \item $P^r\left\{(Y_i)_{i=1}^r : \sqrt{\frac{2}{3}}\leq \norm{\Sigma_{\lambda}^{1/2}(\hat{\Sigma}+\lambda \Id)^{-1/2}}_{\opH} \leq \sqrt{2} \right\} \geq 1-2\delta$;
    \item $P^r\left\{(Y_i)_{i=1}^r :  \norm{\Sigma_{\lambda}^{-1/2}(\hat{\Sigma}+\lambda \Id)^{1/2}}_{\opH} \leq \sqrt{\frac{3}{2}} \right\} \geq 1-2\delta$.
\end{enumerate}
\end{appxlem}
\begin{proof}
\emph{(i)} Define $A(x,y):=\frac{1}{\sqrt{2}}(s(x)-s(y))$, $U(x,y) := \Sigma_{\lambda}^{-1/2} A(x,y)$, $Z(x,y)= U(x,y) \Htens U(x,y)$. Then
$$\Sigma_{\lambda}^{-1/2}(\hat{\Sigma}-\Sigma)\Sigma_{\lambda}^{-1/2} = \frac{1}{r(r-1)}\sum_{i\neq j}Z(X_i,Y_j)-\E(Z(X,Y)). $$
Also, 
\begin{align}
  \sup_{x,y}\norm{Z(x,y)}_{\hsH} &= \sup_{x,y} \norm{U(x,y)}_{H}^2 = \frac{1}{2} \sup_{x,y}\norm{\Sigma_{\lambda}^{-1/2}(s(x)-s(y))}_{H}^2\nonumber\\
  &= \frac{1}{2} \sup_{x,y} \inner{s(x)-s(y)}{\Sigma_{\lambda}^{-1}(s(x)-s(y))}_{H} \nonumber\\
   & \stackrel{(*)}{=}\frac{1}{2}\sup_{x,y}\sum_{i} \frac{\lambda_i}{\lambda_i+\lambda}\left(\frac{\psi_i(x)-\psi_i(y)}{\sqrt{\lambda_i}}\right)^2 
  \leq 2C^2 \Nol,\label{Eq:Z}
\end{align}
where in $(*)$ we used that $\inner{s(x)-s(y)}{s(x)-s(y)}_H=\sum_{i}(\psi_i(x)-\psi_i(y))^2$ which is proved below. To this end, define $a(x)=s(x)-\mu_R$ so that
\begin{align*}
\inner{a(x)}{a(y)}_{H} &= \inner{k(\cdot,x)-\mu_R}{k(\cdot,y)-\mu_R}_{H} \\
&=k(x,y)-\E_{R}k(x,Y)-\E_{R}k(X,y)-\E_{R\times R}k(X,Y) := \bar{k}(x,y).
\end{align*}
Therefore,
\begin{align*}
    \inner{s(x)-s(y)}{s(x)-s(y)}_{H}&=\inner{a(x)-a(y)}{a(x)-a(y)}_{H}\\
    &=\bar{k}(x,x)-2\bar{k}(x,y)+\bar{k}(y,y).
\end{align*}
Following the same argument as in proof of Lemma \ref{lemma: bound hs and op}\emph{(i)}, we obtain 
$$\bar{k}(x,y)=\sum_i\bar{\psi_i}(x)\bar{\psi_i}(y),$$ where $\bar{\psi_i}=\psi_i-\E_R\psi_i$, yielding
\begin{align*}
    \sum_{i}(\psi_i(x)-\psi_i(y))^2=\sum_{i}(\bar{\psi}_i(x)-\bar{\psi}_i(y))^2 = \bar{k}(x,x)-2\bar{k}(x,y)+\bar{k}(y,y).
\end{align*}
Define $\zeta(x) := \E_{Y}[Z(x,Y)]$. Then 
\begin{align*}
    \sup_{x} \norm{\zeta(x)}_{\opH} &\leq \sup_{x,y}\norm{U(x,y)\Htens U(x,y)}_{\opH}
    =\sup_{x,y}\norm{U(x,y)}_{H}^2 \leq 2C\Nol,
\end{align*}
where the last inequality follows from \eqref{Eq:Z}. Furthermore, we have
\begin{align*}
    &\E (\zeta(X)-\Sigma)^2 \preccurlyeq  \E \zeta^2(X) = \E \left(\Sigma_{\lambda}^{-1/2} \E_{Y}[A(X,Y)\Htens A(X,Y)]\Sigma_{\lambda}^{-1/2}\right)^2 \\
    & = \E  \left(\Sigma_{\lambda}^{-1/2} \E_{Y}[A(X,Y)\Htens A(X,Y)] \Sigma_{\lambda}^{-1} \E_{Y}[A(X,Y)\Htens A(X,Y)] \Sigma_{\lambda}^{-1/2} \right) \\
    & \preccurlyeq  \sup_{x} \norm{\zeta(x)}_{\opH} \E  \left(\Sigma_{\lambda}^{-1/2} \E_{Y}[A(X,Y)\Htens A(X,Y)] \Sigma_{\lambda}^{-1/2}\right)  \\
    & \preccurlyeq 2C\Nol \Sigma_{\lambda}^{-1/2}\Sigma\Sigma_{\lambda}^{-1/2} := S.
\end{align*}
Note that $\norm{S}_{\opH} \leq 2C\Nol := \sigma^2$ and $$d := \frac{\text{Tr}(S)}{\norm{S}_{\opH}} \le \frac{\Nol(\norm{\Sigma}_{\opH}+\lambda)}{\norm{\Sigma}_{\opH}} \stackrel{(*)}{\leq} 2 \Nol,$$ where $(*)$ follows by using $\lambda \leq \norm{\Sigma}_{\opH}$. Using Theorem D.3\emph{(ii)} from \citet{kpca} , we get 
\begin{align*}
    &P^r\left\{(Y_i)^r_{i=1}:\norm{\Sigma_{\lambda}^{-1/2}(\hat{\Sigma}-\Sigma)\Sigma_{\lambda}^{-1/2}}_{\opH}\leq \frac{4C\beta\Nol}{r}+\sqrt{\frac{24C\beta\Nol}{r}}\right.\\
    &\qquad\qquad\qquad\left.+\frac{16C\Nol\log\frac{3}{\delta}}{r}\right\}\geq 1-2\delta,
\end{align*} 
where $\beta:=\frac{2}{3}\log\frac{4d}{\delta}$. Then using $\Nol \geq \frac{\norm{\Sigma}_{\opH}}{\lambda+\norm{\Sigma}_{\opH}} \geq \frac{1}{2} > \frac{3}{8}$, it can be verified that
$\frac{4C\beta\Nol}{r}+\sqrt{\frac{24C\beta\Nol}{r}}+\frac{16C\Nol\log\frac{3}{\delta}}{r} \leq \frac{32}{3}w+\sqrt{8w}$, where $w:=\frac{2C\Nol\log\frac{8\Nol}{\delta}}{r}$. Note that $w \leq \frac{1}{68}$ implies $\frac{32}{3}w+\sqrt{8w} \leq \frac{1}{2}.$ This means, if $r\ge 136C^2\Nol\log\frac{8\Nol}{\delta}$ and $\lambda \leq \norm{\Sigma}_{\opH}$, we get 
$$P^r\left\{(Y_i)_{i=1}^r : \norm{\Sigma_{\lambda}^{-1/2}(\hat{\Sigma}-\Sigma)\Sigma_{\lambda}^{-1/2}}_{\opH}\leq \frac{1}{2}\right\} \geq 1-2\delta.$$
\\
\emph{(ii)} By defining $B_r := \Sigma_{\lambda}^{-1/2}(\Sigma-\hat{\Sigma})\Sigma_{\lambda}^{-1/2}$, we have 
\begin{align*}
    \norm{\Sigma_{\lambda}^{1/2}(\hat{\Sigma}+\lambda\Id)^{-1/2}}_{\opH} &= \norm{(\hat{\Sigma}+\lambda\Id)^{-1/2}\Sigma_{\lambda}(\hat{\Sigma}+\lambda\Id)^{-1/2}}_{\opH}^{1/2} \\
    &= \norm{\Sigma_{\lambda}^{1/2}(\hat{\Sigma}+\lambda\Id)^{-1}\Sigma_{\lambda}^{1/2}}_{\opH}^{1/2} \\
    &= \norm{(\Id-B_r)^{-1}}_{\opH}^{1/2} \leq (1-\norm{B_r}_{\opH})^{-1/2},
\end{align*}
where the last inequality holds whenever $\norm{B_r}_{\opH} < 1$. Similarly,
$$ \norm{\Sigma_{\lambda}^{1/2}(\hat{\Sigma}+\lambda\Id)^{-1/2}}_{\opH} = \norm{(\Id+(-B_r))^{-1}}_{\opH}^{1/2} \geq (1+\norm{B_r}_{\opH})^{-1/2}.$$
The result therefore follows from \emph{(i)}.\\

\noindent \emph{(iii)} Since
$$\norm{\Sigma_{\lambda}^{-1/2}(\hat{\Sigma}+\lambda\Id)^{1/2}}_{\opH} = \norm{\Id-B_r}_{\opH}^{1/2} \leq (1+\norm{B_r}_{\opH})^{1/2},$$
the result follows from \emph{(i)}.
\end{proof}

\begin{appxlem} \label{Lem: distance}
For probability measures $P$ and $Q$,   $$d(P,Q)=\sqrt{\chi^2\left(P||\frac{P+Q}{2}\right)}$$ is a metric. Futhermore $H^2(P,Q) \leq d^2(P,Q) \leq 2H^2(P,Q)$, where $H(P,Q)$ denotes the Hellinger distance between $P$ and $Q$.
\end{appxlem}
\begin{proof}
Observe that $d^2(P,Q)= \chi^2\left(P||\frac{P+Q}{2}\right) = \frac{1}{2}\int\frac{(dP-dQ)^2}{d(P+Q)}$. Thus it is obvious that $d(P,P)=0$, $d(P,Q)=d(Q,P)$ and $d(P,Q)>0$ if $P \neq Q$. Hence, it remains just to check the triangular inequality. For that matter, we will first show that
\begin{equation}
    \frac{dP-dQ}{\sqrt{d(P+Q)}}\leq \frac{dP-dZ}{\sqrt{d(P+Z)}}+\frac{dZ-dQ}{\sqrt{d(Z+Q)}},\label{Eq:tri}
\end{equation}
where $Z$ is a probability measure. Defining $\alpha=\frac{dP-dZ}{dP-dQ}$, note that $d(P+Q)=\alpha\cdot d(P+Z)+ (1-\alpha)\cdot `d(Z+Q)$. Therefore, using the convexity of the function $\frac{1}{\sqrt{x}}$ over $[0,\infty)$ yields 
\begin{align*}
    \frac{dP-dQ}{\sqrt{d(P+Q)}} &\leq (dP-dQ)\left(\frac{\alpha}{\sqrt{d(P+Z)}}+\frac{1-\alpha}{\sqrt{d(Z+Q)}}\right)\\
    &= \frac{dP-dZ}{\sqrt{d(P+Z)}}+\frac{dZ-dQ}{\sqrt{d(Z+Q)}}.
\end{align*}
Then by squaring \eqref{Eq:tri} and applying Cauchy-Schwartz inequality, we get 
\begin{align*}
    &\frac{1}{2}\int\frac{(dP-dQ)^2}{d(P+Q)}\\ &\leq \frac{1}{2}\int\frac{(dP-dZ)^2}{d(P+Z)}+\frac{1}{2}\int\frac{(dZ-dQ)^2}{d(Z+Q)}+\int\frac{(dP-dZ)(dZ-dQ)}{\sqrt{d(P+Z)}\sqrt{d(Z+Q)}} \\
    & \leq \frac{1}{2}\int\frac{(dP-dZ)^2}{d(P+Z)}+ \frac{1}{2}\int\frac{(dZ-dQ)^2}{d(Z+Q)}+ \left(\int \frac{(dP-dZ)^2}{d(P+Z)}\right)^{1/2}\left(\int\frac{(dZ-dQ)^2}{d(Z+Q)}\right)^{1/2} \\ & =
    \left(d(P,Z)+d(Z,Q)\right)^2,
\end{align*}
which is equivalent to $d(P,Q) \leq d(P,Z)+d(Z,Q)$.
For the relation with Hellinger distance, observe that $H^2(P,Q) = \frac{1}{2}\int (\sqrt{dP}-\sqrt{dQ})^2$, and $d^2(P,Q) = \frac{1}{2}\int (\sqrt{dP}-\sqrt{dQ})^2 \frac{(\sqrt{dP}+\sqrt{dQ})^2}{d(P+Q)}$. Since $d(P+Q)\leq (\sqrt{dP}+\sqrt{dQ})^2\leq 2 d(P+Q)$, the result follows.
\end{proof}

\begin{appxlem} \label{lem:MMD bounds}
Let $u=\frac{dP}{dR}-1 \in \emph{\range}(\T^{\theta})$ and $D_{\mathrm{MMD}}^2=\norm{\mu_P-\mu_Q}_{\h}^2$. Then 
$$D_{\mathrm{MMD}}^2 \geq 4\norm{u}_{\Lp}^{\frac{2\theta+1}{\theta}}  \norm{\T^{-\theta}u}_{\Lp}^{\frac{-1}{\theta}}.$$
\end{appxlem}
\begin{proof}
Since $u \in \range(\T^\theta)$, then $u=\T^{\theta}f$ for some $f \in \Lp$. Thus,
\begin{align*}
  \norm{u}^2_{\Lp}&=\sum_{i}\lambda_i^{2\theta}\inner{f}{\tilde{\phi_i}}_{\Lp}^2  =\sum_{i}\lambda_i^{2\theta}\inner{f}{\tilde{\phi_i}}_{\Lp}^{\frac{4\theta}{2\theta+1}}\inner{f}{\tilde{\phi_i}}_{\Lp}^{\frac{2}{2\theta+1}} \\ & \stackrel{(*)}{\leq} \left(\sum_{i}\lambda_i^{2\theta+1}\inner{f}{\tilde{\phi_i}}_{\Lp}^2\right)^{\frac{2\theta}{2\theta+1}}\left(\sum_{i}\inner{f}{\tilde{\phi_i}}_{\Lp}^2\right)^{\frac{1}{2\theta+1}} \\ & = \norm{\T^{1/2}u}_{\Lp}^{\frac{4\theta}{2\theta+1}} \norm{\T^{-\theta}u}_{\Lp}^{\frac{2}{2\theta+1}},
\end{align*}
where $(*)$ holds by Holder's inequality. The desired result follows by noting that $D_{\mathrm{MMD}}^2 = 4 \norm{\T^{1/2}u}^2_{\Lp}.$
\end{proof}

\begin{appxlem}\label{lem:cnst} Suppose $(A_1)$ and $(A_3)$ hold. Then
$\lim_{\lambda\rightarrow 0}xg_\lambda(x)\asymp 1$ for all $x\in\Gamma$.
\end{appxlem}
\begin{proof}
$(A_3)$ states that $$\sup_{\{x\in \Gamma: x\gl(x)<B_3\}} |B_3-x\gl(x)|x^{2\varphi} \leq C_3 \lambda^{2\varphi},$$ where  $\Gamma : = [0,\K]$, $\varphi \in (0,\xi]$, $C_3, B_3$ are positive constants all independent of $\lambda>0$. So by taking limit as $\lambda \to 0$ on both sides of the above inequality, we get $$0 \leq \lim_{\lambda \to 0}\sup_{\{x\in \Gamma: x\gl(x)<B_3\}} |B_3-x\gl(x)|x^{2\varphi} \leq 0,$$ which yields that $$\lim_{\lambda \to 0}\sup_{\{x\in \Gamma: x\gl(x)<B_3\}} |B_3-x\gl(x)|x^{2\varphi} = 0.$$ This implies  $\lim_{\lambda \to 0}\sup_{\{x\in \Gamma: x\gl(x)<B_3\}} |B_3-x\gl(x)| = 0$, which implies that  $$\lim_{\lambda \to 0} x\gl(x) = B_3$$ for all $\{x\in \Gamma: x\gl(x)<B_3\}$. For $\{x \in \Gamma: x\gl(x)>B_3\}$ the result is trivial since $(A_1)$ dictates that $x\gl(x) < C_1$.
\end{proof}
\end{document}